\documentclass[a4paper,10pt,oneside]{article} 
\usepackage{geometry}
\usepackage{standalone}
\usepackage[T1] {fontenc}
\usepackage{lmodern}
\usepackage [utf8] {inputenc}
\usepackage[english] {babel}
\usepackage{xcolor}
\usepackage{graphicx}
\usepackage{authblk}
\usepackage{bbm}
\usepackage{amsmath,amssymb,amsthm, xcolor, listings}
\usepackage{mathrsfs}
\usepackage{tikz}
\usetikzlibrary{patterns}
\usetikzlibrary{decorations}
\usetikzlibrary{decorations.pathreplacing}
\usepackage{paralist}
\usepackage{fancyhdr}
\usepackage{mathtools}
\usepackage{titlesec}
\usepackage{enumitem}
\usepackage{float,placeins}
\usepackage{url}
\usepackage{hyperref}
\usepackage{bm}
\usepackage{mathdots}
\usepackage{microtype} 
\usepackage{eurosym}

\newtheorem{theorem}{Theorem}[section]
\newtheorem{proposition}[theorem]{Proposition}
\newtheorem{lemma}[theorem]{Lemma}
\newtheorem{corollary}[theorem]{Corollary}
\theoremstyle{definition}
\newtheorem{definition}[theorem]{Definition}
\newtheorem{remark}[theorem]{Remark}

\numberwithin{equation}{section}
\newtheorem{assumption}[theorem]{Assumption}

\title{Metastability of the three-state Potts model with general interactions}
\author[
         {}\hspace{0.5pt}\protect\hyperlink{hyp:email1}{1},\protect\hyperlink{hyp:affil1}{a}
        ]
        {\protect\hypertarget{hyp:author1}{Gianmarco Bet}}

\author[
         {}\hspace{0.5pt}\protect\hyperlink{hyp:email2}{2},\protect\hyperlink{hyp:affil3}{b}
        ]
        {\protect\hypertarget{hyp:author2}{Anna Gallo}}

\author[
         {}\hspace{0.5pt}\protect\hyperlink{hyp:email3}{3},\protect\hyperlink{hyp:affil2}{c},\protect\hyperlink{hyp:corresponding}{$\dagger$}
        ]
        {\protect\hypertarget{hyp:author3}{Seonwoo Kim}}

\affil[ ]{
          \small\parbox{365pt}{
             \parbox{5pt}{\textsuperscript{\protect\hypertarget{hyp:affil2}{a}}}Università degli Studi di Firenze (Italy),
            \enspace
             \parbox{5pt}{\textsuperscript{\protect\hypertarget{hyp:affil1}{b}}} IMT School for Advanced Studies Lucca (Italy),
             \enspace
             \parbox{5pt}{\textsuperscript{\protect\hypertarget{hyp:affil3}{c}}} Seoul National University (Korea)
            }
          }

\affil[ ]{
          \small\parbox{365pt}{
             \parbox{5pt}{\textsuperscript{\protect\hypertarget{hyp:email1}{1}}}\texttt{\footnotesize\href{mailto:gianmarco.bet@unifi.it}{gianmarco.bet@unifi.it}},
             \parbox{5pt}{\textsuperscript{\protect\hypertarget{hyp:email2}{2}}}\texttt{\footnotesize\href{mailto:anna.gallo@imtlucca.it}{anna.gallo@imtlucca.it}},
             \parbox{5pt}{\textsuperscript{\protect\hypertarget{hyp:email3}{3}}}\texttt{\footnotesize\href{mailto:ksw6leta@snu.ac.kr}{ksw6leta@snu.ac.kr}}
            }
          }

\affil[ ]{
          \small\parbox{365pt}{
             \parbox{5pt}{\textsuperscript{\protect\hypertarget{hyp:corresponding}{$\dagger$}}}Corresponding author
            }
          }
\date{\today}

\begin{document}
\maketitle
\begin{abstract}
We consider the Potts model on a two-dimensional periodic rectangular lattice with general coupling constants $J_{ij}>0$, where $i,j\in\{1,2,3\}$ are the possible spin values (or colors). The resulting energy landscape is thus significantly more complex than in the original Ising or Potts models. The system evolves according to a Glauber-type spin-flipping dynamics. We focus on a region of the parameter space where there are two symmetric metastable states and a stable state, and the height of a direct path between the metastable states is equal to the height of a direct path between any metastable state and the stable state. We study the metastable transition time in probability and in expectation, the mixing time of the dynamics and the spectral gap of the system  when the inverse temperature $\beta$ tends to infinity. Then, we identify all the critical configurations that are visited with high probability during the metastable transition.

\medskip\noindent
\emph{Acknowledgment:} The authors would like to express their deepest gratitude to Prof. Francesca Romana Nardi, who encouraged the collaboration that led to this manuscript, but sadly passed away shortly after. This work is dedicated to her memory.

SK was supported by NRF-2019-Fostering Core Leaders of the Future Basic Science Program/Global Ph.D. Fellowship Program and the National Research Foundation of Korea (NRF) grant funded by the Korean government (MSIT) (No. 2017R1A5A1015626, 2022R1F1A106366811 and 2022R1A5A6000840). SK would like to thank University of Florence for the warm hospitality during his two-week stay to finalize this work.
\end{abstract}
\tableofcontents

\section{Introduction}\label{intro}

Metastability is a phenomenon that is observed when a thermodynamical system is close to a first-order phase transition. When a physical system lies close to its phase coexistence line, it may remain trapped for a long (random) time in a local minimum of the energy, the so-called \textit{metastable state}, before hitting a stable state in a shorter timescale. The phenomenon of metastability is observed in numerous physical systems, such as supercooled liquids or supersaturated gases, but also plays a role in diverse fields such as biology, computer science, chemistry, etc. Motivated by this, in the last decades a lot of effort has been directed towards building a rigorous mathematical understanding of metastability. In this context, a model typically consists of a space of configurations $\mathcal X$, an energy function $H:\mathcal X\to\mathbb R$ and a suitable Markov dynamics driven by the energy difference between configurations. From a mathematical perspective, the metastable behavior of such a system is summarized by the first hitting time of the collection $\mathcal X^s$ of stable configurations, starting from any metastable configuration $\mathbf m\in\mathcal X^m$, as well as the set of \textit{critical configurations} visited during the transition $\mathbf m\to\mathcal X^s$ with probability close to one. A full characterization of the metastable behavior of a model also involves the determination of the so-called \textit{tube of typical paths} visited during the metastable transition.

The classical Potts model with $q\in\mathbb N$ spins on a lattice is defined as follows. We fix a finite graph $\Lambda=(V,E)$ and consider a spin configuration space $\mathcal{X}:=\{1,2,\dots,q\}^V$. To each configuration $\sigma\in\mathcal{X}$ is assigned the energy $H(\sigma)$, given by 
\begin{align}\label{e_origham}
H(\sigma):=-J\sum_{\{v,w\}\in E} \mathbbm{1}_{\{\sigma(v)=\sigma(w)\}}-h\sum_{v\in V}\mathbbm{1}_{\{\sigma(v)=1\}},
\end{align}
where $J>0$ is the \textit{coupling} or \textit{interaction constant} of the system and $h\in\mathbb R$ is the \textit{external magnetic field}. The model is said to be \textit{ferromagnetic} since those configurations in which neighboring spins have the same value are energetically favored. When $q=2$, this is the well-known Ising model.

In this paper, we study the metastable behavior of a \textit{generalized} three-state Potts model (i.e., $q=3$) with zero external field on a finite two-dimensional discrete torus, in the sense that rather than dealing with a fixed coupling constant $J$ as in \eqref{e_origham}, we assume that the coupling constants differ among all possible pairs of spins:
\begin{align}\label{e_genham}
H(\sigma)=-\sum_{i\in S}J_{ii}\sum_{\{v,w\}\in E} \mathbbm{1}_{\{\sigma(v)=\sigma(w)=i\}}+\sum_{i,j\in S,\,i< j} J_{ij} \sum_{\{v,w\}\in E} \mathbbm{1}_{\{\{\sigma(v),\sigma(w)\}=\{i,j\}\}}.
\end{align}
Here, $S:=\{1,2,\dots,q\}$ and the generalized coupling constants $J_{ii}$ and $J_{ij}$ are assumed to be positive. Because of this, the model is ferromagnetic
and thus it is natural to expect for every monochromatic configuration (all spins equal) to be energetically stable, at least in a local sense. For simplicity, we assumed that the external field $h$ is zero.

If $J_{ii}=J>0$ for all $i\in S$, and $J_{ij}$, $i<j$ are also identical, then we recover the original Potts energy \eqref{e_origham} (translated by a fixed real number) with $h=0$. Thus, this model is an extension of the classical Potts model.
\begin{figure}
\centering
\begin{tikzpicture}[scale=0.8,transform shape]
\draw[white] (1.7,0) circle (1cm);
\draw[white] (-1.7,0) circle (1cm); 

\draw (0,-2.7) node[below] {(a) Ising model, $h>0$};
\draw[thick] (1,0) circle (1cm);
\fill[black!40!white] (1,0) circle (1cm);
\draw[thick] (-1,0) circle (1cm);
\fill[black!20!white] (-1,0) circle (1cm);

\fill (1,0) circle (1.2pt);
\draw (1,0) node[below] {\small{$\mathbf s$}};
\fill (-1,0) circle (1.2pt);
\draw (-1,0) node[below] {\small{$\mathbf m$}};

\draw[thick] (0,0) circle (2.5pt);
\fill[red!20!white] (0,0) circle (2.5pt);
\end{tikzpicture}\ \ \ \ \ \ 
\begin{tikzpicture}[scale=0.8,transform shape]
\fill[red!20!white] (-1.55,2.34) rectangle (1.55,2.59);
\draw[pattern=vertical lines, pattern color=black] (-1.55,2.34) rectangle (1.55,2.59);
\fill[red!20!white,rotate around={-60:(-3.19,2.47)}] (-1.55,2.34) rectangle (1.55,2.59);
\draw[pattern=north east lines, pattern color=black, rotate around={-60:(-3.19,2.47)}] (-1.55,2.34) rectangle (1.55,2.59);
\fill[red!20!white,rotate around={60:(3.19,2.47)}] (-1.55,2.34) rectangle (1.55,2.59);
\draw[pattern=north west lines, pattern color=black, rotate around={60:(3.19,2.47)}] (-1.55,2.34) rectangle (1.55,2.59);

\draw (0,-2.7) node[below] {(b) Potts model, $h=0$};
\draw[thick] (0,-1.5) circle (1cm);
\fill[black!40!white] (0,-1.5) circle (1cm);
\draw[thick] (1.7,1.6) circle (1cm);
\fill[black!40!white] (1.7,1.6) circle (1cm);
\draw[thick] (-1.7,1.6) circle (1cm);
\fill[black!40!white] (-1.7,1.6) circle (1cm);
    
\fill  (0,-1.5) circle (1.2pt);
\draw (0,-1.5) node[below] {\small{$\mathbf{s_3}$}};
\fill (1.7,1.6) circle (1.2pt);
\draw (1.7,1.6) node[below] {\small{$\mathbf{s_2}$}};
\fill (-1.7,1.6) circle (1.2pt);
\draw (-1.7,1.6) node[below] {\small{$\mathbf{s_1}$}};

\end{tikzpicture}\ \ \ \ \ \ 
\begin{tikzpicture}[scale=0.8,transform shape]
\fill[blue!20!white] (-1.55,2.34) rectangle (1.55,2.59);
\draw[pattern=vertical lines, pattern color=black] (-1.55,2.34) rectangle (1.55,2.59);
\fill[blue!20!white,rotate around={-60:(-3.19,2.47)}] (-1.55,2.34) rectangle (1.55,2.59);
\draw[pattern=north east lines, pattern color=black, rotate around={-60:(-3.19,2.47)}] (-1.55,2.34) rectangle (1.55,2.59);
\fill[blue!20!white,rotate around={60:(3.19,2.47)}] (-1.55,2.34) rectangle (1.55,2.59);
\draw[pattern=north west lines, pattern color=black, rotate around={60:(3.19,2.47)}] (-1.55,2.34) rectangle (1.55,2.59);

\draw (0,-2.7) node[below] {(c) Potts model, $h<0$};
\draw[thick] (0,-1.5) circle (1cm);
\fill[black!40!white] (0,-1.5) circle (1cm);
\draw[thick] (1.7,1.6) circle (1cm);
\fill[black!40!white] (1.7,1.6) circle (1cm);
\draw[thick] (-1.7,1.6) circle (1cm);
\fill[black!40!white] (-1.7,1.6) circle (1cm);
\draw[thick] (0,0.5) circle (1cm);
\fill[black!20!white] (0,0.5) circle (1cm);
    
\fill  (0,-1.5) circle (1.2pt);
\draw (0,-1.5) node[below] {\small{$\mathbf{s_3}$}};
\fill (1.7,1.6) circle (1.2pt);
\draw (1.7,1.6) node[below] {\small{$\mathbf{s_2}$}};
\fill (-1.7,1.6) circle (1.2pt);
\draw (-1.7,1.6) node[below] {\small{$\mathbf{s_1}$}};
\fill (0,0.5) circle (1.2pt);
\draw (0,0.5) node[below] {\small{$\mathbf m$}};

\draw[thick] (0.85,1.04) circle (2.5pt);
\fill[red!20!white] (0.85,1.04) circle (2.5pt);
\draw[thick] (-0.85,1.04) circle (2.5pt);
\fill[red!20!white] (-0.85,1.04) circle (2.5pt);
\draw[thick] (0,-0.5) circle (2.5pt);
\fill[red!20!white] (0,-0.5) circle (2.5pt);
\end{tikzpicture}\vspace{3mm}

\begin{tikzpicture}[scale=0.8,transform shape]
\draw[thick,dotted] (-1.55,2.59) -- (1.55,2.59) (-1.55,2.34) -- (1.55,2.34);
\draw[thick,dotted,rotate around={-60:(-3.19,2.47)}] (-1.55,2.59) -- (1.55,2.59) (-1.55,2.34) -- (1.55,2.34);
\draw[thick,dotted,rotate around={60:(3.19,2.47)}] (-1.55,2.59) -- (1.55,2.59) (-1.55,2.34) -- (1.55,2.34);

\draw (0,-2.7) node[below] {(d) Potts model, $h>0$};
\draw[thick] (0,-1.5) circle (1cm);
\fill[black!20!white] (0,-1.5) circle (1cm);
\draw[thick] (1.7,1.6) circle (1cm);
\fill[black!20!white] (1.7,1.6) circle (1cm);
\draw[thick] (-1.7,1.6) circle (1cm);
\fill[black!20!white] (-1.7,1.6) circle (1cm);
\draw[thick] (0,0.5) circle (1cm);
\fill[black!40!white] (0,0.5) circle (1cm);
    
\fill  (0,-1.5) circle (1.2pt);
\draw (0,-1.5) node[below] {\small{$\mathbf{m_3}$}};
\fill (1.7,1.6) circle (1.2pt);
\draw (1.7,1.6) node[below] {\small{$\mathbf{m_2}$}};
\fill (-1.7,1.6) circle (1.2pt);
\draw (-1.7,1.6) node[below] {\small{$\mathbf{m_1}$}};
\fill (0,0.5) circle (1.2pt);
\draw (0,0.5) node[below] {\small{$\mathbf s$}};

\draw[thick] (0.85,1.04) circle (2.5pt);
\fill[red!20!white] (0.85,1.04) circle (2.5pt);
\draw[thick] (-0.85,1.04) circle (2.5pt);
\fill[red!20!white] (-0.85,1.04) circle (2.5pt);
\draw[thick] (0,-0.5) circle (2.5pt);
\fill[red!20!white] (0,-0.5) circle (2.5pt);
\end{tikzpicture}\ \ \ \ \ \ 
\begin{tikzpicture}[scale=0.8,transform shape]
\fill[red!20!white] (-1.55,1.34) rectangle (1.55,1.59);
\draw[pattern=vertical lines, pattern color=black] (-1.55,1.34) rectangle (1.55,1.59);

\draw (0,-2.7) node[below] {\textbf{(e) Symmetric generalization}};
\draw[thick] (1.7,0.6) circle (1cm);
\fill[black!20!white] (1.7,0.6) circle (1cm);
\draw[thick] (-1.7,0.6) circle (1cm);
\fill[black!20!white] (-1.7,0.6) circle (1cm);
\draw[thick] (0,-0.5) circle (1cm);
\fill[black!40!white] (0,-0.5) circle (1cm);
    
\fill (1.7,0.6) circle (1.2pt);
\draw (1.7,0.6) node[below] {\small{$\mathbf{m_2}$}};
\fill (-1.7,0.6) circle (1.2pt);
\draw (-1.7,0.6) node[below] {\small{$\mathbf{m_1}$}};
\fill (0,-0.5) circle (1.2pt);
\draw (0,-0.5) node[below] {\small{$\mathbf s$}};

\draw[thick] (0.85,0.04) circle (2.5pt);
\fill[red!20!white] (0.85,0.04) circle (2.5pt);
\draw[thick] (-0.85,0.04) circle (2.5pt);
\fill[red!20!white] (-0.85,0.04) circle (2.5pt);
\end{tikzpicture}\ \ \ \ \ \ 
\begin{tikzpicture}[scale=0.8,transform shape]
\draw[white] (1.7,0) circle (1cm);
\draw[white] (-1.7,0) circle (1cm); 

\draw (0,-2.7) node[below] {(f) Asymmetric generalization};
\draw[thick] (1,0.86) circle (1cm);
\fill[black!30!white] (1,0.86) circle (1cm);
\draw[thick] (-1,0.86) circle (1cm);
\fill[black!20!white] (-1,0.86) circle (1cm);
\draw[thick] (0,-0.86) circle (1cm);
\fill[black!40!white] (0,-0.86) circle (1cm);

\fill (1,0.86) circle (1.2pt);
\draw (1,0.86) node[below] {\small{$\mathbf{m_2}$}};
\fill (-1,0.86) circle (1.2pt);
\draw (-1,0.86) node[below] {\small{$\mathbf{m_1}$}};
\fill (0,-0.86) circle (1.2pt);
\draw (0,-0.86) node[below] {\small{$\mathbf s$}};

\draw[thick] (0,0.86) circle (2.5pt);
\fill[red!20!white] (0,0.86) circle (2.5pt);
\draw[thick] (0.5,0) circle (2.5pt);
\fill[blue!20!white] (0.5,0) circle (2.5pt);
\draw[thick] (-0.5,0) circle (2.5pt);
\fill[red!20!white] (-0.5,0) circle (2.5pt);
\end{tikzpicture}

\caption{\label{fig1}Energy landscape of the general Potts model for various choices of the coupling constants. In this paper, we focus on the landscape represented in (e).}
\end{figure}
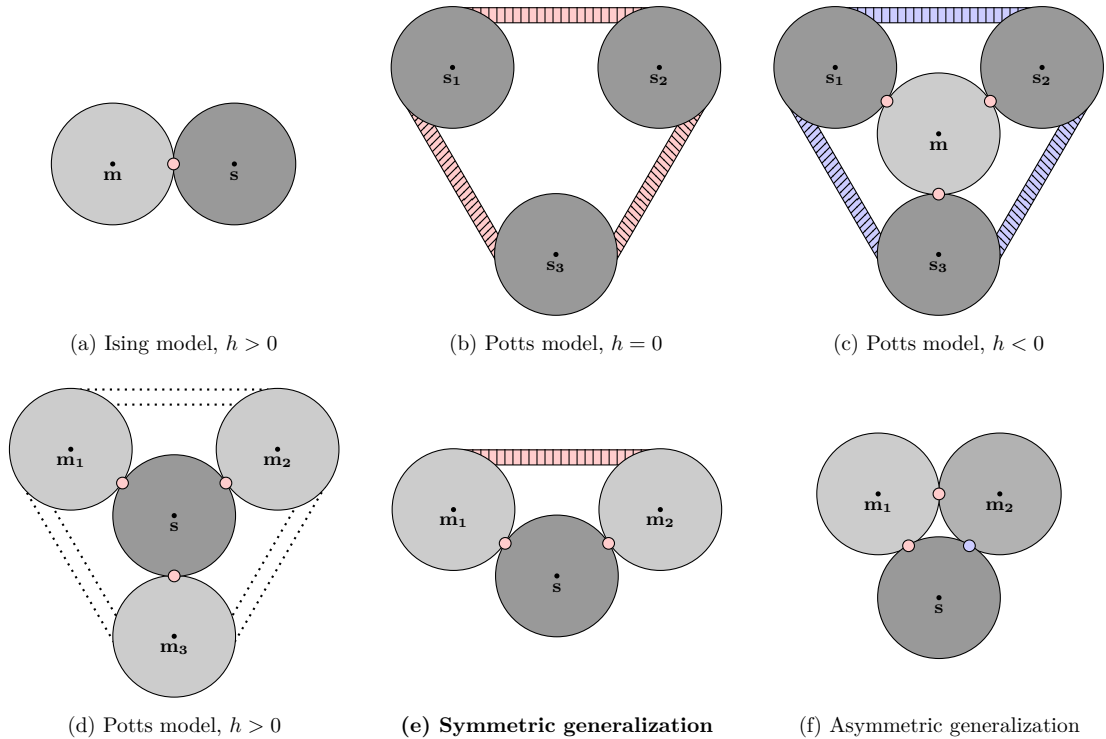
Different choices for the coupling constants $(J_{ij})_{i,j\in S}$ lead to rather different energy landscapes, see Figure \ref{fig1}. In this figure, $\mathbf m$ and $\mathbf s$ indicate respectively the metastable and stable configurations. Moreover, gray (resp.~dark gray) disks indicate the valleys\footnote{These are actually \textit{initial cycles}; see \eqref{e_Cr} for the exact definition.} of the metastable (resp.~stable) configurations. Item (a) illustrates the well-known Ising model with small external field \cite{neves1991critical}, where the red area indicates the collection of so-called critical configurations. Item (b) represents the energy landscape of Potts model with zero external field \cite{bet2021critical,kim2021metastability,nardi2019tunneling}, so that between any $\mathbf{s_i}$ and $\mathbf{s_j}$, there lies a large plateau of saddle configurations depicted in red color. Item (c) corresponds to the degenerate Potts model with negative external field \cite{bet2021metastabilityneg}, where there exists one metastable configuration $\mathbf m$ and multiple stable configurations $\mathbf{s_i}$. If the process starts from the metastable configuration $\mathbf m$, with probability close to one it passes through a critical configuration (colored red) to reach one of the stable configurations $\mathbf{s_i}$. After reaching some $\mathbf{s_i}$, with high probability the process makes transition to another stable configuration $\mathbf{s_j}$ by visiting the large plateau of configurations between $\mathbf{s_i}$ and $\mathbf{s_j}$, colored blue. With probability close to one, the red region is not visited again since the blue region has energy much smaller than the red region. Item (d) represents the degenerate Potts model with positive external field \cite{bet2021metastabilitypos}, where there exist multiple metastable configurations and a unique stable configuration. In this case, starting from any metastable configuration $\mathbf{m_i}$, the process necessarily visits the red region between $\mathbf{m_i}$ and $\mathbf s$ to reach the stable state $\mathbf s$. Indeed, the plateaux between any two metastable configurations (dotted region) have higher energy than the red regions, so that with high probability they are not visited during the metastable transition.

The discussion above allows us to illustrate the main contribution of our work. Indeed, in all the energy landscapes in Figure \ref{fig1}-(a)--(d), each metastable transition involves only one of the following: either a transition through a single critical configuration (which we call \textit{sharp}), or a transition through a plateau of saddle configurations (which we call \textit{complex}). This property simplifies the study of metastable transitions considerably. On the other hand, in our model with general coupling constants, depending on the precise values of the couplings, it may happen that 
\begin{enumerate}
\item[\textit{(D1)}] during a single metastable transition, the system undergoes both sharp and complex transitions, or
\item[\textit{(D2)}] two metastable states do not have the same energy level.
\end{enumerate}
These two situations are represented respectively in Figure \ref{fig1}-(e) and  \ref{fig1}-(f). We begin by discussing the former setting. Here, $\mathbf{m_1}$ and $\mathbf{m_2}$ are symmetric metastable configurations. If the process starts from, say, $\mathbf{m_1}$, then the system either moves to $\mathbf{m_2}$ via the large plateau, or to $\mathbf{s}$ via a critical configuration. Therefore, a full metastable transition from $\mathbf{m_1}$ to $\mathbf{s}$ involves possibly several complex transitions between the two metastable states, and a final sharp transition to $\mathbf{s}$. Note that the coupling constants have been chosen precisely such that the system must overcome the same energy barrier for both types of transitions.  This implies that both types happen with positive probability. Figure \ref{fig1}-(f) is a sketch of the second situation \textit{(D2)}, where $\mathbf{m_1}$ and $\mathbf{m_2}$ are both metastable but $\mathbf{m_2}$ has lower energy than $\mathbf{m_1}$. Starting from $\mathbf{m_1}$, the process either visits the stable state $\mathbf s$ or visits the other metastable state $\mathbf{m_2}$. Moreover, starting from $\mathbf{m_2}$, the process makes a sharp transition to $\mathbf s$. In this case, all the metastble transitions are sharp. \textit{In this paper, we choose the coupling constants such that the resulting energy landscape is the one represented in Figure \ref{fig1}-(e)}, see Section \ref{Sec2} for a more detailed discussion on the coupling constants.

The stochastic evolution is given by a \textit{Glauber-type dynamics}, which is a Markov chain with transition probabilities given by the Metropolis algorithm that only allows single spin flip updates (cf.~\eqref{metropolisTP}). This dynamics is reversible with respect to the Gibbs measure \eqref{e_gibbs}, and hence \eqref{e_gibbs} is the stationary distribution. We study this system in the low-temperature regime, i.e., in the limit $\beta\to\infty$ where $\beta$ is the inverse temperature.

Let us now briefly describe our approach. Without loss of generality, we choose the coupling constants so that the constant configurations $\mathbf2$ and $\mathbf3$ (all spins equal to 2 and 3 respectively) are the metastable configurations and the configuration $\mathbf1$ (all spins equal to 1) is stable. In Section \ref{Sec3.1}, we begin by proving that, indeed, the set of metastable configurations is $\mathcal X^m=\{\mathbf2,\mathbf3\}$ and that the only stable configuration is $\mathcal X^s=\{\mathbf1\}$. We prove that the energy barrier between a metastable state $\mathbf m\in\mathcal{X}^m$ and $\mathcal{X}^s$ is equal to the energy barrier between the two metastable configurations $\mathbf2$ and $\mathbf3$, which is the key feature of our model. We also need to prove that any configuration other than $\mathbf1$, $\mathbf2$ or $\mathbf3$ has a sufficiently lower stability level than $\mathbf2$ or $\mathbf3$. To this end, we carefully analyze the local geometry of configurations to deduce a necessary condition for a configuration to be a local minimum of the energy function. Then, we apply this necessary condition to deduce a non-trivial, but still sufficient upper bound for the stability level, see Proposition \ref{p_stablev}. As a byproduct, we also obtain the \textit{recurrence property} which means that starting from an arbitrary initial configuration, the process visits one of the monochromatic configurations in a much shorter timescale than the metastable transition time, see Theorem \ref{t_rec}. Then, we focus on the transition $\mathbf m\to\mathbf 1$, and we are able to obtain the expected value and distribution of the transition time (cf.~Theorem \ref{t_meta}-(a)). Furthermore, in Theorem \ref{t_meta}-(b)(c), we study the behavior of the \textit{mixing time} and give an estimate of the \textit{spectral gap} in the low-temperature regime, see \eqref{mixingtimedef} and \eqref{spectralgapdef} for the exact definitions. Finally, in Theorems \ref{t_gate1} and \ref{t_gate2} we identify the set of all minimal gates for the same transition, so that then we can characterize in Corollary \ref{c_gate} all the essential saddle configurations visited during the transition.

\paragraph{Mathematical obstacles}
In the Potts model with positive external field analyzed in \cite{bet2021metastabilitypos} there is an arbitrary number of metastable configurations (e.g., two) and one stable configuration. However, the energy level of a canonical path from a metastable configuration to the stable configuration is strictly less than the energy level of a canonical path between two distinct metastable configurations. This indicates that, with high probability, during a metastable transition, no other metastable configurations are visited, see Figure \ref{fig1}-(d) or \cite[Figure 2]{bet2021metastabilitypos}. On the other hand, as mentioned above, in our model the energy barrier between $\mathbf m\in\{\mathbf2,\mathbf3\}$ and $\mathbf1$ equals the energy barrier between $\mathbf2$ and $\mathbf3$. This feature makes it challenging to analyze the model since the system may perform an arbitrary number of excursions between the metastable states before the actual metastable transition takes place. Moreover, the analysis of a transition between two monochromatic configurations (say, from $\mathbf 2$ to $\mathbf 1$) is made significantly harder by possible spin flips involving the third spin (say, 3). Indeed, it is more likely to observe such a spin update in our setting than in the setting of \cite{bet2021metastabilitypos}; see also Remark \ref{r_opt}.

In order to face these obstacles, we define a \textit{projection operator} $\mathscr{P}^{rs}$, $r,s\in S$ on the configuration space $\mathcal X$. This operator replaces all spins $r$ in a configuration with spins $s$, so that there is no spin $r$ in the resulting configuration, which simplifies the analysis greatly. This is made possible because, intuitively, the resulting configuration is more homogeneous than the original one in terms of spin values (i.e., lower energy), since the spin disagreements between $r$ and $s$ disappear. We will define the operator precisely and study it in detail in Section \ref{Sec4}.

Additional obstacle appears when we estimate the stability level of configurations other than $\mathbf1$, $\mathbf2$ and $\mathbf3$. A case of particular difficulty is when a configuration has a small cluster of spin $1$ which is surrounded by both spins $2$ and $3$, see the second subcase of \textbf{(Case 1)} in the proof of Proposition \ref{p_stablev}. To overcome this problem, we introduce a procedure to update the spins $1$ in a distinctive manner which is sufficient to provide an upper bound for the stability level. More precisely, we choose a candidate spin between $2$ and $3$, say $r\in\{2,3\}$, and then update each spin $1$ to $r$ first on the boundary of the $1$-cluster and then in the interior, so that we can explicitly calculate the energy difference of spin updates in the worst-case scenario. We refer to the proof of Proposition \ref{p_stablev} for details.

\paragraph{Literature on the Potts model}
The Potts model, which originates from statistical physics, has been studied extensively both in the mathematics literature and the physics literature. The study of the equilibrium properties of the Potts model and their dependence on $q$ have been investigated on the square lattice $\mathbb Z^d$ in  \cite{baxter1982critical,baxter1973potts}, on the triangular lattice in \cite{baxter1978triangular,enting1982triangular} and on the Bethe lattice in \cite{ananikyan1995phase,de1991metastability,di1987potts}. The mean-field version of the Potts model has been studied in \cite{costeniuc2005complete,ellis1990limit,ellis1992limit,gandolfo2010limit,wang1994solutions}. 
In \cite{bet2021critical,kim2021metastability,nardi2019tunneling} the authors investigate the tunneling behavior of the Potts model with zero external magnetic field. In this energy landscape there are $q$ stable states and no relevant metastable state.
In \cite{nardi2019tunneling}, the authors derive the asymptotic behavior of the first hitting time for the transition between stable configurations, and give results in probability, in expectation and in distribution. They also focus on the behavior of the mixing time and give a lower and an upper bound for the spectral gap. In \cite{bet2021critical}, the authors study the tunneling from a stable state to the other stable configurations and between two stable states. In both cases, they also identify the set of all the critical configurations and the tube of typical trajectories. Finally, in \cite{kim2021metastability}, the authors study the model in dimensions two and three. They give a description of the so-called \textit{gateway configurations} in order to compute the prefactor of the expected metastable transition time. 
The $q$-Potts model with positive and negative external magnetic field has been studied in \cite{bet2021metastabilitypos} and \cite{bet2021metastabilityneg}, respectively. In \cite{bet2021metastabilitypos}, the energy landscape is characterized by $q-1$ multiple degenerate metastable states and a unique stable configuration. On the other hand, in \cite{bet2021metastabilityneg}, there is a unique metastable state and $q-1$ stable configurations. In both scenarios, the authors tackle all the three issues of metastability discussed in the introduction.

\paragraph{Literature on metastability} In this paper we adopt the framework known as \textit{pathwise approach}. 
This was introduced in \cite{cassandro1984metastable} in 1984 by Cassandro, Galves, Olivieri and Vares, and then it was further developed in \cite{alonso1996three,olivieri1995markov,olivieri1996markov,olivieri2005large}. The pathwise approach requires a detailed knowledge of the energy landscape to give precise answers to the three issues of metastability in the form of ad hoc large deviations estimates. In \cite{cirillo2013relaxation,cirillo2015metastability,fernandez2015asymptotically,fernandez2016conditioned,manzo2004essential,nardi2016hitting} the pathwise approach was further built up by separating the study of the first and second issues of metastability from that regarding the tube of typical trajectories since this requires more detailed model-dependant data. The framework of the pathwise approach has been also adopted in \cite{arous1996metastability,cirillo1998metastability,cirillo1996metastability,kotecky1994shapes,nardi1996low,neves1991critical,neves1992behavior,olivieri2005large} to tackle the three issues for Ising-like models with Glauber dynamics. Moreover, it was also used in \cite{apollonio2021metastability,den2003droplet,gaudilliere2005nucleation,hollander2000metastability,nardi2016hitting,zocca2019tunneling} to study the transition time and the gates for Ising-like and hard-core models with Kawasaki and Glauber dynamics. Finally, this approach has been applied in \cite{cirillo2003metastability,cirillo2008competitive,cirillo2008metastability,dai2015fast,procacci2016probabilistic} to probabilistic cellular automata (parallel dynamics).

The so-called \emph{potential-theoretic approach} exploits a suitable Dirichlet form and spectral properties of the transition matrix to give sharp asymptotics for the hitting time. More precisely, this method estimates the leading order of the expected value of the transition time including its \textit{prefactor}, see \cite{bovier2016metastability,bovier2002metastability,bovier2004metastability,cirillo2017sum}. 
In \cite{bashiri2019on,bovier2006sharp,boviermanzo2002metastability,cirillo2017sum,den2012metastability,den2018metastability,jovanovski2017metastability}, the authors adopt the potential-theoretic approach to estimate the prefactor for Ising-like models and the hard-core model for Glauber and Kawasaki dynamics and in \cite{bet2020effect,nardi2012sharp} for parallel dynamics. We refer also to the approaches recently developed in \cite{bianchi2016metastable,gaudillierelandim2014}.

A new technique known as the \textit{martingale approach} has been developed in \cite{beltran2010tunneling,beltran2012tunneling}. This approach is particularly useful to investigate a series of metastable transitions happening at the same timescale, since we can characterize the successive metastable transitions as a Markov chain on a simple representative set; this is called the method of \textit{model reduction}. Applying this methodology, numerous results regarding the metastable systems with complex structures have been obtained \cite{beltran2009zerorange,bianchi2016metastable,kim2021inclusion,kimseo2021inclusion,landim2014totasymzrp,seo2019zerorange}. Moreover, there has been a recent development of technology \cite{landim2021resolvent}, known as the \textit{resolvent approach} to metastability, in which the authors successfully characterize metastability in the sense of \textit{equivalent conditions} on whether a metastable behavior occurs by focusing on a certain class of resolvent equations on the state space. See also \cite{landim2022resolvent} for a recent breakthrough using this resolvent approach to characterize metastability of overdamped Langevin dynamics with highly general potential functions.

\paragraph{Open problems}
As mentioned in the explanation regarding Figure \ref{fig1}, we deal with the case (e) when spins $2$ and $3$ have equal magnitude of both stability and energy, so that the model is symmetric with respect to the action $2\leftrightarrow3$. On the other hand, as in Figure \ref{fig1}-(f), it is also possible to choose coupling constants so that $\mathbf2$ and $\mathbf3$ has the same stability level but with different energy values, say $H(\mathbf2)<H(\mathbf3)$. In this case, the resulting metastable transition on the system will be \textit{totally asymmetric}, in the sense that only the directions $\mathbf3\to\mathbf2\to\mathbf1$ are possible for a metastable transition to occur.

Moreover, for simplicity we only deal with the three-state Potts model, i.e., $q=3$. One can also think of a more complex model, where there are generally $q\ge 3$ spins in the system and both the cases in Figure \ref{fig1}-(e)--(f) emerge within the energy landscape. The authors expect that one can construct energy landscapes which exhibit metastable behavior in a high level of generality. This will be an important research objective in the future.

In addition, one can proceed further to investigate more \textit{quantitative} analysis of metastability, which includes calculating the prefactor of expected transition time (the so-called \textit{Eyring--Kramers formula}), employing the method of model reduction to characterize the law of successive metastable transitions, etc. To this end, one needs to apply the potential-theoretic \cite{bovier2016metastability} and martingale approach \cite{landim2019metastable} to metastability. See also the last remark in Definition \ref{d_refpath2}.

\paragraph{Outline} The outline of the paper is as follows. In Section \ref{Sec2} we introduce the model and its main distinctive features. In Section \ref{Sec3} we give the main results concerning the energy landscape and the first two issues of metastability introduced in the first paragraph. In order to prove these results, in Section \ref{Sec4} we introduce the projection operator and prove some of its interesting properties. In Section \ref{Sec5} we recall some results related to the original Ising model. Finally, Section \ref{Sec6} is devoted to the proof of the main results.

\section{Model description}\label{Sec2}

We consider a generalized three-state Potts model on a finite two-dimensional rectangular lattice graph $\Lambda=(V,E)$, where $V=\{0,\dots,K-1\}\times\{0,\dots,L-1\}$ is the vertex set and $E$ is the edge set. Without loss of generality, we assume throughout the article that
\[
K\le L.
\]
To fix ideas, we assume periodic boundary conditions; more precisely, we also include each pair of vertices lying on opposite sides of the lattice in the edge set, so that we obtain a two-dimensional torus $\mathbb{T}_K\times\mathbb{T}_L$. We say that two vertices $v,w\in V$ are nearest neighbors (or simply neighbors) and denote by $v\sim w$ when they share an edge of $\Lambda$, i.e., when $\{v,w\}\in E$. To each vertex $v\in V$  is associated a spin taking value in $S:=\{1,2,3\}$, and thus $\mathcal X := S^V$ denotes the set of \textit{spin configurations}. We denote by $\textbf{1},\textbf{2},\textbf{3}\in \mathcal X$ those configurations in which all the vertices have spin value $1,2,3$, respectively.

To each configuration $\sigma\in\mathcal X$ we associate the energy $H(\sigma)\in\mathbb{R}$ given by

\begin{align}\label{e_hamiltonian}
H(\sigma):=-\sum_{i\in S}J_{ii}\sum_{\{v,w\}\in E} \mathbbm{1}_{\{\sigma(v)=\sigma(w)=i\}}+\sum_{i,j\in S,\,i< j} J_{ij} \sum_{\{v,w\}\in E} \mathbbm{1}_{\{\{\sigma(v),\sigma(w)\}=\{i,j\}\}}, 
\end{align}
where for any $i,j\in S$, $J_{ij}>0$ are the coupling or interaction constants. The function $H:\mathcal X\to\mathbb R$ is the so-called \textit{Hamiltonian}, or energy function. We remark that our techniques would also work if the Hamiltonian includes an external field, leading to similar results. For simplicity we decided to focus on the case of zero external field.

We assume that
\begin{align}\label{e_Jii}
J_{11}>J_{22}=J_{33}\;\;\;\;\text{and}\;\;\;\;J_{12}=J_{13},
\end{align}
so that, intuitively, spin $1$ is more stable than spins $2$ and $3$, and the Hamiltonian is symmetric with respect to the spin exchange $2\leftrightarrow3$. Then, we write
\begin{align}\label{e_gammadef}
\gamma_1:=J_{11}-J_{22}>0,\;\;\;\;\gamma_{12}:=J_{12}+J_{22}>0\;\;\;\;\text{and}\;\;\;\;\gamma_{23}:=J_{23}+J_{22}>0.
\end{align}
We assume the following conditions throughout this article. Recall the definition of function $f_h(x)$ from \eqref{e_fh}.

\begin{assumption}\label{assump}
The following conditions hold.
\begin{enumerate}
\item[\textup{\textbf{A}}.] $\frac{2\gamma_{12}+\gamma_1}{2\gamma_1}$ is not an integer.
\item[\textup{\textbf{B}}.] $f_{\gamma_1}(\gamma_{12})=2(K+1)\gamma_{23}$.
\item[\textup{\textbf{C}}.] $2\gamma_{12}\ge4\gamma_{23}+\gamma_1$.
\end{enumerate}
\end{assumption}
Intuitively, condition \textbf{A} corresponds to the familiar condition $2/h\notin\mathbb{N}$, where $h>0$ is the external field of the original Ising model, first proposed in \cite[standard case]{neves1991critical}. Condition \textbf{B} implies that the energy barriers of the canonical transitions $\mathbf2\to\mathbf1$ and $\mathbf2\to\mathbf3$ are the same. It is clear that condition \textbf{C} is satisfied if there exists a constant $k>0$ sufficiently large so that $\gamma_{12}\ge k\gamma_1$ and $\gamma_{12}\ge k\gamma_{23}$. This is a natural selection of constants which will be justified in more detail in Section \ref{SecA.3}.
We refer to Section \ref{SecA.3} for more detailed explanation on each of our specific choices. 

\begin{remark}\label{r_opt}
Condition \textbf{C} is an optimal condition on the coefficients and used in the proof of Lemma \ref{l_P12} only. This inequality has the following interpretation. Keeping \eqref{e_gammadef} in mind, we can rewrite condition \textbf{C} as
\begin{align}\label{e_mincond'}
2J_{12}-J_{11}+3J_{22} \ge 4J_{23}+4J_{22}.
\end{align}
By a simple algebraic computation, we can see that the left-hand side of \eqref{e_mincond'} is the energy needed to add a protuberance to a cluster of spins $1$ in the sea of spins $2$, and the right-hand side of \eqref{e_mincond'} is the energy needed to add a new spin $3$ in the sea of spins $2$. Thus, \eqref{e_mincond'} suggests that the dynamics favor a single appearance of an unrelated spin (in this case, $3$) over the enlargement of a cluster of spins $1$. This subtle dynamical behavior constitutes a significant challenge that is not present in the ferromagnetic Ising and Potts models analyzed in, say, \cite{bet2021metastabilitypos}.
\end{remark}

The \textit{Gibbs measure} is defined as a probability distribution on the configuration space $\mathcal X$ given by
\begin{align}\label{e_gibbs}
\mu_\beta(\sigma):=\frac{e^{-\beta H(\sigma)}}{Z_\beta},
\end{align}
where $\beta>0$ is the inverse temperature and where the normalization constant
$$
Z_\beta:=\sum_{\sigma'\in\mathcal X}e^{-\beta H(\sigma')}
$$
is the \textit{partition function}.

The spin system evolves according to a discrete-time Glauber-type Metropolis dynamics, which is described by a single-spin-updating Markov chain $\{X_t^\beta\}_{t\in\mathbb{N}}$ on the space $\mathcal X$ with the following transition probabilities: for $\sigma, \sigma' \in \mathcal X$,
\begin{align}\label{metropolisTP}
P_\beta(\sigma,\sigma'):=
\begin{cases}
Q(\sigma,\sigma')e^{-\beta [H(\sigma')-H(\sigma)]^+}, &\text{if}\ \sigma' \neq \sigma,\\
1-\sum_{\eta:\,\eta \neq \sigma} P_\beta (\sigma, \eta), &\text{if}\ \sigma'=\sigma,
\end{cases}
\end{align}
where $[t]^+:=\max\{0,t\}$ is the positive part of $t$ and 
\begin{align*}
Q(\sigma,\sigma'):=
\begin{cases}
\frac{1}{3|V|}, &\text{if}\ |\{v\in V:\ \sigma(v) \neq \sigma'(v)\}|=1,\\
0, &\text{if}\ |\{v\in V:\ \sigma(v) \neq \sigma'(v)\}|>1.
\end{cases}
\end{align*}
Here, $Q$ is the so-called \textit{connectivity matrix} and it is symmetric and irreducible, i.e., for all $\sigma, \sigma' \in \mathcal X$, there exists a finite sequence of configurations $\omega_0,\dots,\omega_n \in \mathcal X$ such that $\omega_0=\sigma$, $\omega_n=\sigma'$ and $Q(\omega_i,\omega_{i+1})>0$ for $i=0,\dots,n-1$. We call the triplet $(\mathcal X,H,Q)$ \textit{energy landscape}. The resulting stochastic dynamics defined by \eqref{metropolisTP} is reversible with respect to the Gibbs measure defined $\mu_\beta$ in \eqref{e_gibbs}. The law and the corresponding expectation of the dynamics are denoted by $\mathbb{P}$ and $\mathbb{E}$.

\section{Main results}\label{Sec3}
%

\subsection{Stable and metastable states}\label{Sec3.1}

We denote by $\mathcal{X}^s$ the set of global minima of the Hamiltonian \eqref{e_hamiltonian}. A simple algebraic computation implies the following proposition.
\begin{proposition}[Identification of $\mathcal X^s$]\label{propstable}
It holds that $\mathcal X^s=\{\mathbf 1\}$.
\end{proposition}
\begin{proof}
By the definition \eqref{e_hamiltonian} and the assumption \eqref{e_Jii}, it is straightforward that
\[
H(\sigma)\ge-\sum_{i\in S}J_{ii}\sum_{\{v,w\}\in E}\mathbbm{1}_{\{\sigma(v)=\sigma(w)=i\}}\ge-J_{11}\sum_{i\in S}\sum_{\{v,w\}\in E}\mathbbm{1}_{\{\sigma(v)=\sigma(w)=i\}}.
\]
The last double summation is exactly $\sum_{\{v,w\}\in E}\mathbbm{1}_{\{\sigma(v)=\sigma(w)\}}$, which is clearly bounded above by $|E|=2KL$. Therefore, we conclude that $H(\sigma)\ge-2J_{11}KL$. The equality holds if and only if $\sigma(v)=\sigma(w)=1$ for all $\{v,w\}\in E$, which is equivalent to $\sigma=\mathbf1$.
\end{proof}
Next, we identify the metastable configurations. To this end, we need to define some crucial notions. A \textit{path} is a finite sequence $\omega$ of configurations $\omega_0,\dots,\omega_n \in \mathcal X$, $n \in \mathbb{N}$, such that $Q(\omega_i,\omega_{i+1})>0$ for $i=0,\dots,n-1$. We denote by $\Omega_{\sigma,\sigma'}$ the set of all paths between $\sigma$ and $\sigma'$. For convenience of notation, we sometimes write $\omega:\sigma\to \sigma'$ to indicate $\omega\in\Omega_{\sigma,\sigma'}$.
For any path $\omega=(\omega_0,\dots,\omega_n)$, we define the \textit{height} of $\omega$ as
\begin{align}\label{height}
\Phi_\omega:=\max_{i=0,\dots,n} H(\omega_i).
\end{align}
For any pair of configurations $\sigma, \sigma' \in \mathcal X$, the \textit{communication height} $\Phi(\sigma,\sigma')$ between $\sigma$ and $\sigma'$ is defined as the minimal height among all paths $\omega:\sigma\to\sigma'$, i.e., 
\begin{align}\label{e_comheightdef}
\Phi(\sigma,\sigma'):=\min_{\omega:\sigma \to \sigma'} \Phi_\omega = \min_{\omega:\sigma \to \sigma'} \max_{\eta \in \omega} H(\eta).
\end{align}
We define the set of \textit{optimal paths} between $\sigma, \sigma' \in\mathcal X$ as
\begin{align}\label{optpaths}
\Omega_{\sigma,\sigma'}^{opt}:=\{\omega\in\Omega_{\sigma,\sigma'}:\ \Phi_\omega=\Phi(\sigma,\sigma')\}.
\end{align}
Accordingly, for disjoint subsets $\mathcal{A}$ and $\mathcal{B}$ of $\mathcal{X}$, we define
\[
\Phi(\mathcal{A},\mathcal{B}):=\min_{\sigma\in\mathcal{A}}\min_{\sigma'\in\mathcal{B}}\Phi(\sigma,\sigma')
\]
and
\[
\Omega_{\mathcal{A},\mathcal{B}}^{opt}:=\{\omega\in\Omega_{\sigma,\sigma'}:\ \sigma\in\mathcal{A},\;\sigma'\in\mathcal{B},\;\Phi_\omega=\Phi(\mathcal{A},\mathcal{B})\}.
\]

\noindent For any $\sigma \in \mathcal X$, let $\mathcal{I}_\sigma:=\{\eta \in \mathcal X:\ H(\eta)<H(\sigma)\}$ 
be the set of configurations with energy strictly smaller than $H(\sigma)$.
Then, we define the \textit{stability level} of $\sigma$ as
\begin{align}\label{stabilitylevel}
V_\sigma:=\Phi(\sigma,\mathcal{I}_\sigma)-H(\sigma).
\end{align}
If $\mathcal{I}_\sigma = \varnothing$ (i.e. if $\sigma\in\mathcal{X}^s$), we set $V_\sigma:=\infty$.
Finally, we define the collection of \textit{metastable states} as 
\begin{align}\label{metastableset}
\mathcal X^m:=\big\{\eta \in \mathcal X:\ V_\eta=\max_{\sigma\in\mathcal X\setminus\mathcal X^s} V_\sigma\big\}.
\end{align}
Furthermore, for any $\sigma\in\mathcal X$ and any $\varnothing\neq\mathcal A\subset\mathcal X$, we set
\begin{align}
\Gamma(\sigma,\mathcal A):=\Phi(\sigma,\mathcal A)-H(\sigma).
\end{align}
From the definition it immediately follows that $V_\sigma=\Gamma(\sigma,\mathcal{I}_\sigma)$.\medskip{}

First, we investigate the stability level of configurations $\mathbf2$ and $\mathbf3$. We define (cf. Assumption \ref{assump}-\textbf{B})
\begin{align}\label{e_consts}
\ell^\star:=\Big\lceil\frac{2\gamma_{12}+\gamma_1}{2\gamma_1}\Big\rceil\;\;\;\;\text{and}\;\;\;\;\Gamma^\star:=f_{\gamma_1}(\gamma_{12})=2(K+1)\gamma_{23}.
\end{align}
By definition \eqref{e_fh} we have that
\begin{align}\label{Gammametastable}
\Gamma^\star=4\ell^\star\Big(\gamma_{12}+\frac{\gamma_1}2\Big)-2\gamma_1(\ell^{\star2}-\ell^\star+1)=4\ell^\star\gamma_{12}-2\gamma_1(\ell^{\star2}-2\ell^\star+1).
\end{align}
The result below shows that the communication height between stable and metastable states is exactly $\Gamma^\star$ above the energy level of metastable configurations.

\begin{theorem}[Communication height]\label{t_comheight}
It holds that
\[
\Gamma(\mathbf2,\mathbf1)=\Gamma(\mathbf3,\mathbf1)=\Gamma(\mathbf2,\mathbf3)=\Gamma^\star.
\]
\end{theorem}

By definition, the previous theorem is equivalent to $\Phi(\mathbf2,\mathbf1)=\Phi(\mathbf3,\mathbf1)=\Phi(\mathbf2,\mathbf3)=H(\mathbf2)+\Gamma^\star$. This theorem is proved in Section \ref{Sec6.1}.\medskip{}

Next, we claim that stability levels of any other configurations are significantly smaller than $\Gamma^\star$. We present a proof of the following proposition in Section \ref{Sec6.2}.

\begin{proposition}[Stability level of other configurations]\label{p_stablev}
For any configuration $\eta\in\mathcal{X}\setminus\{\mathbf1,\mathbf2,\mathbf3\}$,
\[
V_\eta\le2(\ell^\star-1)(\gamma_{23}+\gamma_1).
\]
\end{proposition}

\begin{remark}
By \eqref{e_consts}, it holds that $\ell^\star< \frac{2\gamma_{12}+\gamma_1}{2\gamma_1} +1 $. Employing this inequality to the formula of $\Gamma^\star$ given in \eqref{Gammametastable}, we obtain
\[
\Gamma^\star >4\ell^\star\gamma_{12}-2\gamma_1\frac{2\gamma_{12}+\gamma_1}{2\gamma_1}(\ell^\star-1)=(2\gamma_{12}-\gamma_1)\ell^\star +(2\gamma_{12}+\gamma_1).
\]
Again using $\ell^\star< \frac{2\gamma_{12}+\gamma_1}{2\gamma_1} +1$ on the last term, we have
\[
\Gamma^\star>(2\gamma_{12}-\gamma_1)\ell^\star+2\gamma_1(\ell^\star-1)=2\gamma_{12}\ell^\star+\gamma_1(\ell^\star-2)\ge 2\gamma_{12}\ell^\star.
\]
On the other hand, the upper bound appearing in Proposition \ref{p_stablev} is estimated via Assumption \ref{assump}-\textbf{C} as
\[
2(\ell^\star-1)(\gamma_{23}+\gamma_1)<2\ell^\star\times\frac{2\gamma_{12}}{k}=\frac{4}{k}\gamma_{12}\ell^\star
\]
where $k$ is sufficiently large. Therefore, we conclude that
\[
\Gamma^\star > \frac{k}{2} \times 2(\ell^\star-1)(\gamma_{23}+\gamma_1),
\]
which implies that the stability levels of configurations other than $\mathbf1$, $\mathbf2$ and $\mathbf3$ are significantly smaller than the stability level of metastable configurations $\mathbf2$ and $\mathbf3$.
\end{remark}

By combining Theorem \ref{t_comheight} and Proposition \ref{p_stablev}, we now identify the set $\mathcal{X}^m$.

\begin{theorem}[Identification of $\mathcal X^m$]\label{t_meta}
We have $V_\mathbf{2}=V_\mathbf{3}=\Gamma^\star$ and $\mathcal X^m=\{\mathbf 2,\mathbf 3\}$.
\end{theorem}

\begin{proof}
To prove the theorem, it suffices to demonstrate that
\begin{align}\label{e_V2Gamma}
V_\mathbf2=\Gamma^\star.
\end{align}
Indeed, then by symmetry we also have $V_\mathbf3=\Gamma^\star$, and combining these with Proposition \ref{p_stablev} we conclude that $\mathcal{X}^m=\{\mathbf2,\mathbf3\}$.

Before proving \eqref{e_V2Gamma}, we claim that
\begin{align}\label{e_etaGamma}
\Gamma(\eta,\mathbf1)<\Gamma^\star\;\;\;\;\text{for all }\eta\in\mathcal{X}\setminus\{\mathbf1,\mathbf2,\mathbf3\}.
\end{align}
To prove the claim, we fix $\eta\in\mathcal{X} \setminus \{\mathbf1, \mathbf2, \mathbf3\}$. Starting from $\eta$, we find another configuration $\eta_1$ with lower energy such that an optimal path from $\eta$ to $\eta_1$ realizes the stability level $V_\eta$. Repeating this algorithm, since $\mathcal{X}$ is finite and $\mathcal{X}^s=\{\mathbf1\}$, we can take a finite sequence $\eta=\eta_0,\eta_1,\dots,\eta_m=\mathbf1$ of configurations such that $H(\eta_i)>H(\eta_{i+1})$ and $V_{\eta_i}=\Gamma(\eta_i,\eta_{i+1})$ for all $0\le i\le m-1$. Then, we estimate
\[
\Gamma(\eta,\mathbf1)=\Phi(\eta,\mathbf1)-H(\eta)\le\max_{0\le i\le m-1}\Phi(\eta_i,\eta_{i+1})-H(\eta),
\]
where the inequality holds by concatenating the $m-1$ optimal paths from $\eta_i$ to $\eta_{i+1}$. By construction, the last term equals
\[
\max_{0\le i\le m-1}[V_{\eta_i}+H(\eta_i)]-H(\eta).
\]
By Theorem \ref{t_comheight} and Proposition \ref{p_stablev}, $V_\sigma\le\Gamma^\star$ for all $\sigma\ne\mathbf1$. Thus, the last display is bounded by
\[
\Gamma^\star+H(\eta_0)-H(\eta)=\Gamma^\star.
\]
This concludes the proof of \eqref{e_etaGamma}.\medskip{}

Finally, we prove $V_\mathbf2=\Gamma^\star$. Theorem \ref{t_comheight} readily implies that $V_\mathbf2\le\Gamma^\star$. If, on the contrary, $V_\mathbf2<\Gamma^\star$ then by definition, there exists $\sigma\in\mathcal{X}$ with $H(\sigma)<H(\mathbf2)$ such that $\Phi(\mathbf2,\sigma)<\Gamma^\star+H(\mathbf2)$, where clearly $\sigma\ne\mathbf1,\mathbf2,\mathbf3$. Then by the claim \eqref{e_etaGamma}, we have $\Phi(\sigma,\mathbf1)<\Gamma^\star+H(\sigma)$. Thus,
\[
\Phi(\mathbf2,\mathbf1)\le\max\{\Phi(\mathbf2,\sigma),\Phi(\sigma,\mathbf1)\}<\Gamma^\star+H(\mathbf2),
\]
which contradicts Theorem \ref{t_comheight}. This conclude the proof of Theorem \ref{t_meta}.
\end{proof}

We have a following another important consequence of Proposition \ref{p_stablev}, which is called the \textit{recurrence property}. With probability tending to one in the limit $\beta\to\infty$, starting from any configuration in $\mathcal X$, the process visits $\mathcal X^s\cup\mathcal X^m$ within a time of order $e^{[2(\ell^\star-1)(\gamma_{23}+\gamma_1)]\beta}$ which is much smaller than the metastable timescale $e^{\Gamma^\star\beta}$.

Given a non-empty subset $\mathcal A \subseteq \mathcal X$ and a configuration $\sigma \in \mathcal X$, we define
\begin{align}\label{firsthittingtime}
\tau_\mathcal{A}^\sigma:= \inf\{t>0:\ X_t^\beta \in \mathcal A\}
\end{align}
which is the \textit{first hitting time} of the subset $\mathcal A$ for the Markov chain $\{X_t^\beta\}_{t \in \mathbb{N}}$ starting from $\sigma$.

\begin{theorem}[Recurrence property]\label{t_rec}
For any $\sigma\in\mathcal X$ and for any $\epsilon>0$, there exists $\kappa>0$ such that for $\beta$ sufficiently large, we have
\[
\mathbb{P}\big[\tau_{\{\mathbf 1,\mathbf 2,\mathbf 3\}}^\sigma>e^{\beta[2(\ell^\star-1)(\gamma_{23}+\gamma_1)+\epsilon]}\big]\le e^{-e^{\kappa\beta}}.
\]
\end{theorem}

\begin{proof}
We apply \cite[Theorem 3.1]{manzo2004essential} for the level set with respect to $2(\ell^\star-1)(\gamma_{23}+\gamma_1)$. This concludes the proof since $\mathcal{X}^s=\{\mathbf1\}$ by Proposition \ref{propstable}, $V_{\mathbf2}=V_{\mathbf3}=\Gamma^\star>2(\ell^\star-1)(\gamma_{23}+\gamma_1)$ by Theorem \ref{t_meta} and $V_\eta\le 2(\ell^\star-1)(\gamma_{23}+\gamma_1)$ for all $\eta\in\mathcal{X}\setminus\{\mathbf1,\mathbf2,\mathbf3\}$ by Proposition \ref{p_stablev}.
\end{proof}

\subsection{Metastable transition time, mixing time and spectral gap}\label{Sec3.2}

Our next goal is to study the first hitting time of the stable configuration $\mathbf1$ starting from a metastable configuration $\mathbf{m}\in\mathcal{X}^m=\{\mathbf2,\mathbf3\}$, as well as the mixing time and spectral gap of the stochastic dynamics. For every $\epsilon\in(0,1)$ we define the \textit{mixing time} $t^{mix}_\beta(\epsilon)$ by
\begin{align}\label{mixingtimedef}
t^{mix}_\beta(\epsilon):=\min\big\{n\ge 0:\ \max_{\sigma\in\mathcal X}\|P_\beta^n(\sigma,\cdot)-\mu_\beta\|_{\text{TV}}\le\epsilon\big\},
\end{align}
where $\|\nu-\nu'\|_{\text{TV}}:=\frac 1 2 \sum_{\sigma\in\mathcal X}|\nu(\sigma)-\nu'(\sigma)|$ is the total variation distance between two probability distributions $\nu,\nu'$ on $\mathcal X$. Furthermore, we define the \textit{spectral gap} of the dynamics as
\begin{align}\label{spectralgapdef}
\rho_\beta:=1-\lambda_\beta^{(2)},
\end{align}
where $1=\lambda_\beta^{(1)}>\lambda_\beta^{(2)}\ge\dots\ge\lambda_\beta^{(|\mathcal X|)}\ge-1$ are the eigenvalues of the matrix $P_\beta$.

\begin{theorem}[Metastable transition time, mixing time and spectral gap]\label{t_trans}
For any $\mathbf m\in\{\mathbf2,\mathbf3\}$ the following statements hold.

\begin{enumerate}
\item[\emph{(a)}] For every $\epsilon>0$, $\lim_{\beta\to\infty}\mathbb P(e^{\beta(\Gamma^\star-\epsilon)}<\tau^\mathbf m_{\mathbf1}<e^{\beta(\Gamma^\star+\epsilon)})=1$.
\item[\emph{(b)}] $\lim_{\beta\to\infty}\frac{1}{\beta}\log\mathbb E[\tau^\mathbf m_{\mathbf1}]=\Gamma^\star$.
\item[\emph{(c)}] For every $\epsilon\in(0,1)$, $\lim_{\beta\to\infty}\frac{1}{\beta} \log t^{mix}_\beta(\epsilon)=\Gamma^\star$ and there exist two constants $0<c_1\le c_2<\infty$ independent of $\beta$ such that for any $\beta>0$, $c_1e^{-\beta\Gamma^\star}\le\rho_\beta\le c_2e^{-\beta\Gamma^\star}$.
\end{enumerate}
\end{theorem}

\begin{proof}
Item (a) follows by \cite[Theorem 4.1]{manzo2004essential}, while item (b) follows by \cite[Theorem 4.9]{manzo2004essential}. In both cases we applied the model-independent results with $\eta_0=\mathbf m$ and $\Gamma=\Gamma^*$. To prove item (c), by \cite[Proposition 3.24]{nardi2016hitting}, it suffices to demonstrate in our model that $\widetilde\Gamma(\mathcal{X}\setminus\{\mathbf1\})=\Gamma^\star$ (see equation (21) in \cite{nardi2016hitting} for the definition of $\widetilde\Gamma$). Indeed, by \cite[Lemma 3.6]{nardi2016hitting} we know that
\[
\widetilde\Gamma(\mathcal{X}\setminus\{\mathbf1\})=\max_{\eta\in\mathcal{X}\setminus\{\mathbf1\}}\Gamma(\eta,\mathbf1).
\]
By Theorem \ref{t_comheight}, $\Gamma(\mathbf2,\mathbf1)=\Gamma(\mathbf3,\mathbf1)=\Gamma^\star$. Moreover, by the claim \eqref{e_etaGamma} we have that $\Gamma(\eta,\mathbf1)<\Gamma^\star$ for all $\eta\ne\mathbf1,\mathbf2,\mathbf3$. This concludes the proof.
\end{proof}

\subsection{Minimal gates}\label{Sec3.3}

Next we are interested in identifying the set of minimal gates for the metastable transitions. First, we need a few more model-independent definitions.\medskip{}

The set
\begin{align}\label{saddles}
\mathcal S(\mathcal A,\mathcal B):=\big\{\xi\in\mathcal X:\ \exists\omega\in\Omega_{\mathcal A,\mathcal B}^{opt},\; \xi\in\mathrm{argmax}_\omega H\big\}.
\end{align}
is known as the set of \textit{minimal saddles} between $\mathcal A,\mathcal B\subseteq\mathcal X$. In particular, any $\xi\in\mathcal S(\mathcal A,\mathcal B)$ is called an \textit{essential saddle} if there exists $\omega\in\Omega_{\mathcal A,\mathcal B}^{opt}$ such that $\xi\in\mathrm{argmax}_\omega H$ and
\[
\mathrm{argmax}_{\omega'} H\not\subseteq \mathrm{argmax}_\omega H \setminus \{\xi\}\;\;\;\;\text{for all }\omega'\in\Omega_{\mathcal A,\mathcal B}^{opt}\setminus\{\omega\}
\]
A saddle $\xi\in\mathcal S(\mathcal A,\mathcal B)$ which does not satisfy the condition is said to be \textit{unessential}. One can easily check that this definition coincides with the classical one \cite{manzo2004essential} but is simpler.\medskip{}

A collection $\mathcal W$ of configurations is a \textit{gate} for the transition between $\mathcal{A},\mathcal{B}\in\mathcal X$ if $\mathcal W\subseteq\mathcal S(\mathcal{A},\mathcal{B})$ and $\omega\cap\mathcal W\neq\varnothing$ for all $\omega\in\Omega_{\mathcal{A},\mathcal{B}}^{opt}$. Moreover, $\mathcal{W}$ is said to be a \textit{minimal gate} for the transition $\mathcal A\to\mathcal B$ if it is a gate and for any $\mathcal W'\subset\mathcal W$ there exists $\omega'\in\Omega_{\mathcal A,\mathcal B}^{opt}$ such that $\omega'\cap\mathcal W'=\varnothing$. The set $\mathcal{G}=\mathcal{G}(\mathcal A,\mathcal B)$ denotes the union of all minimal gates for the transition $\mathcal A\to\mathcal B$.\medskip{}

Let us now focus on some model-dependent definitions concerning our setting. We refer to Figures \ref{fig2} and \ref{fig3} for illustrations.

\begin{figure}
\centering
\begin{tikzpicture}[scale=0.65,transform shape]

\fill[black!28!white]
(1.8,0) rectangle (2.1,2.7)(2.1,0.6) rectangle (2.4,1.8);
\draw[step=0.3cm,color=black] (0,0) grid (3.6,2.7);
\draw (1.8,-0.1) node[below] {{(a)}};
\end{tikzpicture}\ \ \ \ \ \ 
\begin{tikzpicture}[scale=0.65,transform shape]
\fill[black!28!white] 
(0,0) rectangle (0.6,2.4);
\draw[step=0.3cm,color=black] (0,0) grid (3.6,2.7);
\draw (1.8,-0.1) node[below] {{(b)}};
\end{tikzpicture}\ \ \ \ \ \ 
\begin{tikzpicture}[scale=0.65,transform shape]
\fill[black!28!white] 
(0.3,0) rectangle (0.6,2.7)(0.6,0) rectangle (0.9,2.4);
\draw[step=0.3cm,color=black] (0,0) grid (3.6,2.7);
\draw (1.8,-0.1) node[below] {{(c)}};
\end{tikzpicture}\ \ \ \ \ \ 
\begin{tikzpicture}[scale=0.65,transform shape]
\fill[black!28!white] 
(0.3,0) rectangle (2.7,2.7)(2.7,0.9) rectangle (3,2.4);
\draw[step=0.3cm,color=black] (0,0) grid (3.6,2.7);
\draw (1.8,-0.1) node[below] {{(d)}};
\end{tikzpicture}

\caption{\label{fig2}Examples of configurations on a grid graph $12\times9$ belonging to (a) $B_{1,K}^4(r,s)\subset\mathscr H_4(\mathbf r,\mathbf s)$, (b) $R_{2,K-1}(r,s)\subset\mathscr Q(\mathbf r,\mathbf s)$, (c) $B_{1,K}^{K-1}(r,s)\subset\mathscr P(\mathbf r,\mathbf s)$ and (d) $B^5_{8,K}(r,s)\subset\mathscr W^5_8(\mathbf r,\mathbf s)$. Spins $r$ and $s$ are represented by colors white and gray, respectively.}
\end{figure}
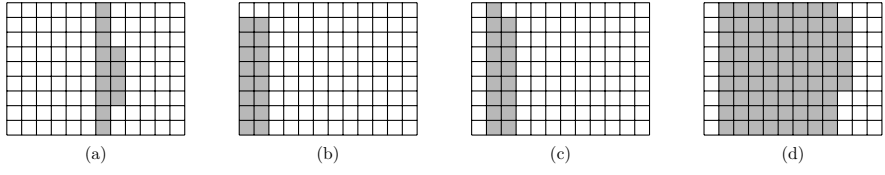

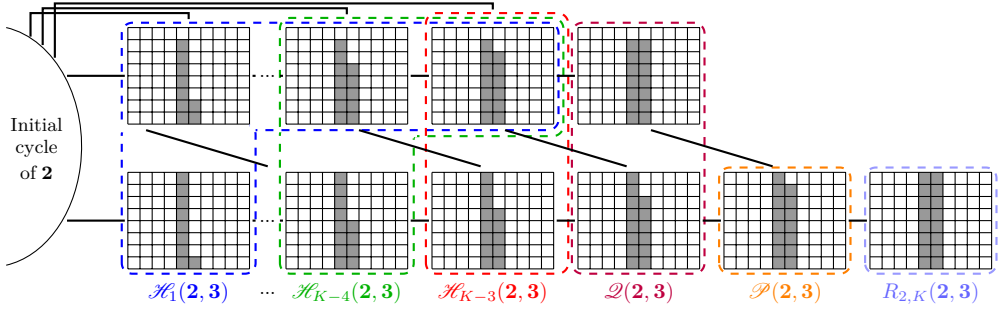
\begin{figure}
\centering
\begin{tikzpicture}[scale = 0.8, transform shape]
\draw[rounded corners,dashed,green!70!black,thick] (2.5,-0.08)rectangle(4.7,2.4);
\fill[white] (4.6,2.4)rectangle(4.8,2.3)(2.4,2.3)rectangle(2.6,2.4)(2.55,1.8)rectangle(4.6,2.5);
\draw[white,ultra thick](2.4,2.3)--(4.6,2.3)(-0.1,2.4)--(2,2.4)(2.4,2.5)--(4.6,2.5)(2.1,2.3)--(2.1,4.2)(2.55,2.2)--(4.6,2.2)(2.55,2.2)--(4.7,2.2);
\draw[rounded corners,dashed,green!70!black,thick](2.5,2.2)rectangle(7.16,4.16);
\draw[white,ultra thick](-0.1,2.3)--(2,2.3)(2.55,2.2)--(4.7,2.2);
\fill[white](2.55,2.2)rectangle(4.75,2.45);

\draw[rounded corners,dashed,blue,thick] (-0.1,-0.08)rectangle(2.1,2.4);
\fill[white] (2,2.4)rectangle(2.2,2.3)(0,2.4)rectangle(-0.2,2.3);
\draw[white,ultra thick](-0.1,2.3)--(2,2.3)(-0.1,2.4)--(2,2.4)(2.1,2.3)--(2.1,4.2);
\draw[rounded corners,dashed,blue,thick](-0.1,2.3)rectangle(7.08,4.08);
\draw[white,ultra thick](-0.15,2.3)--(2.1,2.3);
\fill[white] (-0.1,2.3)rectangle(0,2.4);

\draw[rounded corners,dashed,red,thick] (4.9,-0.08)rectangle(7.24,4.24);

\draw[rounded corners,dashed,purple,thick] (7.3,-0.08)rectangle(9.48,4.08);

\draw[rounded corners,dashed,orange,thick] (9.72,-0.08)rectangle(11.88,1.68);

\draw[rounded corners,dashed,blue!40!white,thick] (12.12,-0.08)rectangle(14.28,1.68);
\fill[black!40!white] (0.8,0)rectangle(1,1.6)(1,0)rectangle(1.2,0.2);
\draw[step=0.2cm,color=black] (0,0) grid (2,1.6);
\draw[black!100!white,thick] (-1,0.8)--(-0.05,0.8)(2.05,0.8)--(2.1,0.8)(2.5,0.8)--(2.55,0.8);
\draw (2.3,0.8) node {\small{$...$}};
\draw (1,-0.1) node[below] {\textcolor{blue}{$\mathscr H_1(\mathbf 2,\mathbf 3)$}};

\fill[black!40!white] (3.4,0)rectangle(3.6,1.6)(3.6,0)rectangle(3.8,0.8);
\draw[step=0.2cm,color=black] (2.6,0) grid (4.6,1.6);
\draw[black!100!white,thick] (4.65,0.8)--(4.95,0.8);
\draw (2.6,0)--(2.6,1.6);
\draw (3.6,-0.1) node[below] {\textcolor{green!70!black}{$\mathscr H_{K-4}(\mathbf 2,\mathbf 3)$}};

\fill[black!40!white] (5.8,0)rectangle(6,1.6)(6,0)rectangle(6.2,1);
\draw[step=0.2cm,color=black] (5,0) grid (7,1.6);
\draw[black!100!white,thick] (7.05,0.8)--(7.35,0.8);
\draw (5,0)--(5,1.6);
\draw (6,-0.1) node[below] {\textcolor{red}{$\mathscr H_{K-3}(\mathbf 2,\mathbf 3)$}};

\fill[black!40!white] (8.2,0)rectangle(8.4,1.6)(8.4,0)rectangle(8.6,1.2);
\draw[step=0.2cm,color=black] (7.4,0) grid (9.4,1.6);
\draw[black!100!white,thick] (9.45,0.8)--(9.75,0.8);
\draw (7.4,0)--(7.4,1.6);
\draw (8.4,-0.1) node[below] {\textcolor{purple}{$\mathscr Q(\mathbf 2,\mathbf 3)$}};

\fill[black!40!white] (10.6,0)rectangle(10.8,1.6)(10.8,0)rectangle(11,1.4);
\draw[step=0.2cm,color=black] (9.8,0) grid (11.8,1.6);
\draw[black!100!white,thick] (11.85,0.8)--(12.15,0.8);
\draw (9.8,0)--(9.8,1.6);
\draw (10.8,-0.1) node[below] {\textcolor{orange}{$\mathscr P(\mathbf 2,\mathbf 3)$}};

\fill[black!40!white] (13,0)rectangle(13.2,1.6)(13.2,0)rectangle(13.4,1.6);
\draw[step=0.2cm,color=black] (12.2,0) grid (14.2,1.6);
\draw (12.2,0)--(12.2,1.6);
\draw (13.2,-0.1) node[below] {\textcolor{blue!60!white}{$R_{2,K}(\mathbf 2,\mathbf 3)$}};

\fill[black!40!white](0.8,2.4)rectangle(1,3.8)(1,2.4)rectangle(1.2,2.8);
\draw[step=0.2cm,color=black] (0,2.4) grid (2,4);
\draw[black!100!white,thick] (-1,3.2)--(-0.05,3.2)(2.05,3.2)--(2.1,3.2)(2.5,3.2)--(2.55,3.2);
\draw (0,2.4)--(2,2.4);
\draw (2.3,3.2) node {{\small{$...$}}};
\draw (2.3,-0.25) node[below] {{\small{$...$}}};

\fill[black!40!white] (3.4,2.4)rectangle(3.6,3.8)(3.6,2.4)rectangle(3.8,3.4);
\draw[step=0.2cm,color=black] (2.6,2.4) grid (4.6,4);
\draw[black!100!white,thick] (4.65,3.2)--(4.95,3.2);
\draw (2.6,2.4)--(2.6,4)(2.6,2.4)--(4.6,2.4);

\fill[black!40!white] (5.8,2.4)rectangle(6,3.8)(6,2.4)rectangle(6.2,3.6);
\draw[step=0.2cm,color=black] (5,2.4) grid (7,4);
\draw[black!100!white,thick] (7.05,3.2)--(7.35,3.2);
\draw (5,2.4)--(7,2.4)(5,2.4)--(5,4);

\fill[black!40!white] (8.2,2.4)rectangle(8.4,3.8)(8.4,2.4)rectangle(8.6,3.8);
\draw[step=0.2cm,color=black] (7.4,2.4) grid (9.4,4);
\draw (7.4,2.4)--(7.4,4)(7.4,2.4)--(9.4,2.4);

\draw[black!100!white,thick] (0.3,2.3)--(2.3,1.7)(3.8,2.3)--(5.8,1.7)(6.2,2.3)--(8.2,1.7)(8.6,2.3)--(10.6,1.7);

\draw (-2,4) arc [start angle=80, end angle=-80, x radius=1.5cm, y radius=2cm];
\draw (-1.5,2.4) node {\small{Initial}};
\draw (-1.5,2) node {\small{cycle}};
\draw (-1.5,1.6) node {\small{of\hspace{2.5pt}$\mathbf 2$}};

\draw[black!100!white,thick] (6,4.3)--(6,4.4)--(-1.2,4.4)--(-1.2,3.5);
\draw[black!100!white,thick] (3.6,4.22)--(3.6,4.32)--(-1.4,4.32)--(-1.4,3.71);
\draw[black!100!white,thick] (1,4.14)--(1,4.24)--(-1.6,4.24)--(-1.6,3.85);

\end{tikzpicture}
\caption{\label{fig3}Local geometry of the configurations belonging to $\mathscr{P}(\mathbf 2,\mathbf 3)$, $\mathscr{Q}(\mathbf 2,\mathbf 3)$ and $\mathscr{H}_i(\mathbf 2,\mathbf 3)$, $1\le i\le K-3$, where $K=8$ and $L=10$. We refer to \eqref{e_Cr} for the definition of initial cycles.}
\end{figure}

\begin{itemize}
\item We say that $R\subseteq V$ is a \textit{rectangle of shape $a\times b$} if the sites in $R$ form a rectangle with $a$ columns and $b$ rows. It is a \textit{strip} if it wraps around $\Lambda$, i.e., if $a=L$ or $b=K$.
\item For $r,s\in S$, we denote by $R_{a,b}(r,s)$ the collection of configurations in which all vertices have spins $r$, except for those in a rectangle $a\times b$, which have spins $s$. Note that $R_{a,b}(r,s)\ne R_{b,a}(r,s)$ if $a\ne b$. Moreover, we write $B_{a,b}^l(r,s)$ (resp. $\hat{B}_{a,b}^l(r,s)$) the collection of configurations in which all vertices have spins $r$, except for those which have spins $s$, in a rectangle $a\times b$ with a bar of length $l$ adjacent to one of the sides of length $b$ (resp. $a$), with $1\le l\le b-1$ (resp. $1\le l\le a-1$).
\item We set
\[
\mathscr{P}(\mathbf r,\mathbf s):=\begin{cases}
B_{1,K}^{K-1}(r,s),&\text{if }K<L,\\
B_{1,K}^{K-1}(r,s)\cup \hat{B}_{K,1}^{K-1}(r,s),&\text{if }K=L.
\end{cases}
\]
\item We define
\[
\mathscr{Q}(\mathbf r,\mathbf s):=\begin{cases}
R_{2,K-1}(r,s)\cup B_{1,K}^{K-2}(r,s),&\text{if }K<L,\\
R_{2,K-1}(r,s)\cup B_{1,K}^{K-2}(r,s)\cup R_{K-1,2}(r,s)\cup \hat{B}_{K,1}^{K-2}(r,s),&\text{if }K=L.
\end{cases}
\]
\item For $1\le i\le K-3$, we define $\mathscr{H}_i(\mathbf r,\mathbf s)$ as
\[
\begin{cases}
B_{1,K}^i(r,s)\cup\bigcup_{j=i+1}^{K-2}B_{1,K-1}^j(r,s),&\text{if }K<L,\\
B_{1,K}^i(r,s)\cup\bigcup_{j=i+1}^{K-2}B_{1,K-1}^j(r,s)\cup \hat{B}_{K,1}^i(r,s)\cup \bigcup_{j=i+1}^{K-2}\hat{B}_{K-1,1}^j(r,s),&\text{if }K=L.
\end{cases}
\]
\item Finally, for $2\le j\le L-3$ and $1\le h\le K-1$ we set
\[
\mathscr{W}_j^h(\mathbf r,\mathbf s):=\begin{cases}
B_{j,K}^h(r,s),&\text{if }K<L,\\
B_{j,K}^h(r,s)\cup \hat{B}_{K,j}^h(r,s),&\text{if }K=L.
\end{cases}
\]
\end{itemize}

Using the sets defined above, we now formulate all the possible minimal gates for the metastable transitions. For $m\in\{2,3\}$, we write
\begin{align}\label{e_Wm1}
\mathcal{W}(\mathbf{m},\mathbf1):=
B_{\ell^\star-1,\ell^\star}^1(m,1)\cup \hat{B}_{\ell^\star,\ell^\star-1}^1(m,1).
\end{align}
Moreover, we abbreviate
\begin{align}\label{e_Wmm'}
\mathcal{W}(\mathbf2,\mathbf3):=&\;\bigcup_{i=1}^{K-3}\mathscr{H}_i(\mathbf2,\mathbf3)\cup\mathscr{Q}(\mathbf2,\mathbf3)\cup\mathscr{P}(\mathbf2,\mathbf3)\nonumber\\
&\;\cup\bigcup_{j=2}^{L-3}\bigcup_{h=1}^{K-1}\mathscr{W}_j^h(\mathbf2,\mathbf3)\cup\mathscr{P}(\mathbf3,\mathbf2)\cup\mathscr{Q}(\mathbf3,\mathbf2)\cup\bigcup_{i=1}^{K-3}\mathscr{H}_i(\mathbf3,\mathbf2).
\end{align}

First, we address the metastable transition from $\mathbf{m}\in\{\mathbf2,\mathbf3\}$ to $\mathbf{1}$. We refer to Figure \ref{fig4} for a viewpoint from above the energy landscape.

\begin{figure}
\centering
\begin{tikzpicture}[scale=1.2,transform shape]
\fill[red!20!white] (-1.55,2.34) rectangle (1.55,2.59);
\draw[pattern=vertical lines, pattern color=black] (-1.55,2.34) rectangle (1.55,2.59);
\draw[decorate,decoration=brace]  (-1.6,3.6) -- (1.6,3.6);
\draw (0,3.65) node[above] {\small{$\mathcal W(\mathbf 2,\mathbf 3)$}};
\draw (-1.5,2.7)--(-1.5,2.8)(-1.2,2.7)--(-1.2,2.8)(-1.50,2.8) -- (-1.20,2.8)(-1.35,2.8)--(-1.35,2.85);
\draw[<-] (-1.7,2.85) -- (-1.35,2.85);
\draw (-1.6,2.8) node[above,left] {\tiny{$\bigcup_i\mathscr H_i(\mathbf 2,\mathbf 3)$}};

\draw (-1.17,2.7)--(-1.17,2.8)(-1.03,2.7)--(-1.03,2.8)(-1.17,2.8)--(-1.03,2.8);
\draw(-1.10,2.8)--(-1.1,3.05);
\draw[<-](-1.7,3.05)--(-1.1,3.05);
\draw (-1.6,3.05) node[above,left] {\tiny{$\mathscr Q(\mathbf 2,\mathbf 3)$}};

\draw (-1,2.7)--(-1,2.8)(-0.86,2.7)--(-0.86,2.8)(-1,2.8)--(-0.86,2.8);
\draw[->](-0.93,2.8)--(-0.93,3.3);
\draw (-0.93,3.2) node[above] {\tiny{$\mathscr P(\mathbf 2,\mathbf 3)$}};

\draw (-0.83,2.7)--(-0.83,2.8)(0.83,2.7)--(0.83,2.8)(-0.83,2.8)--(0.83,2.8);
\draw (0,2.75) node[above] {\tiny{$\bigcup_{j,h}\mathcal W_j^h(\mathbf 2,\mathbf 3)$}};

\draw (1.5,2.7)--(1.5,2.8)(1.2,2.7)--(1.2,2.8)(1.50,2.8) -- (1.20,2.8)(1.35,2.8)--(1.35,2.85);
\draw[<-] (1.7,2.85) -- (1.35,2.85);
\draw (1.6,2.8) node[above,right] {\tiny{$\bigcup_i\mathscr H_i(\mathbf 3,\mathbf 2)$}};

\draw (1.17,2.7)--(1.17,2.8)(1.03,2.7)--(1.03,2.8)(1.17,2.8)--(1.03,2.8);
\draw(1.10,2.8)--(1.1,3.05);
\draw[->](1.10,3.05)--(1.7,3.05);
\draw (1.6,3.05) node[above,right] {\tiny{$\mathscr Q(\mathbf 3,\mathbf 2)$}};

\draw (1,2.7)--(1,2.8)(0.86,2.7)--(0.86,2.8)(1,2.8)--(0.86,2.8);
\draw[->](0.93,2.8)--(0.93,3.3);
\draw (0.93,3.2) node[above] {\tiny{$\mathscr P(\mathbf 3,\mathbf 2)$}};

\draw[thick] (0,0.5) circle (1cm);
\fill[black!40!white] (0,0.5) circle (1cm);
\draw[thick] (1.7,1.6) circle (1cm);
\fill[black!20!white] (1.7,1.6) circle (1cm);
\draw[thick] (-1.7,1.6) circle (1cm);
\fill[black!20!white] (-1.7,1.6) circle (1cm);

\fill  (0,0.5) circle (1.2pt);
\draw (0,0.5) node[below] {\small{$\mathbf 1$}};
\draw (0.2,0.6) node[right] {\large{$\mathcal{C}_1$}};
\fill (1.7,1.6) circle (1.2pt);
\draw (1.7,1.6) node[below] {\small{$\mathbf 3$}};
\draw (1.9,1.7) node[right] {\large{$\mathcal{C}_3$}};
\fill (-1.7,1.6) circle (1.2pt);
\draw (-1.7,1.6) node[below] {\small{$\mathbf 2$}};
\draw (-1.9,1.7) node[left] {\large{$\mathcal{C}_2$}};

\draw[->,thin] (0.85,1.04) --(1.5,-0.1);
\draw (1.45,-0.1) node[right] {\small{$\mathcal W(\mathbf 3,\mathbf 1)$}};
\draw[->,thin] (-0.85,1.04) --(-1.5,-0.1);
\draw (-1.45,-0.1) node[left] {\small{$\mathcal W(\mathbf 2,\mathbf 1)$}};
\draw[thick] (0.85,1.04) circle (2.5pt);
\fill[red!20!white] (0.85,1.04) circle (2.5pt);
\draw[thick] (-0.85,1.04) circle (2.5pt);
\fill[red!20!white] (-0.85,1.04) circle (2.5pt);
\end{tikzpicture}
\caption{\label{fig4}Viewpoint from above of the energy landscape cut at the energy level $\Phi(\mathbf 2,\mathbf 1)=\Phi(\mathbf 3,\mathbf 1)=\Phi(\mathbf 2,\mathbf 3)=\Gamma^\star+H(\mathbf2)$, which explains Figure \ref{fig1}(e) in more detail. Here, $\mathcal{C}_r$ denotes the initial cycle of $\mathbf{r}$; see \eqref{e_Cr} for the exact definition. The gates $\mathcal{W}(\mathbf m,\mathbf1)$ between $\mathbf m\in\{\mathbf2,\mathbf3\}$ and $\mathbf1$ consist of singletons, whereas the gate $\mathcal{W}(\mathbf 2,\mathbf3)$ between $\mathbf2$ and $\mathbf3$ is a complex union of essential saddles.}
\end{figure}
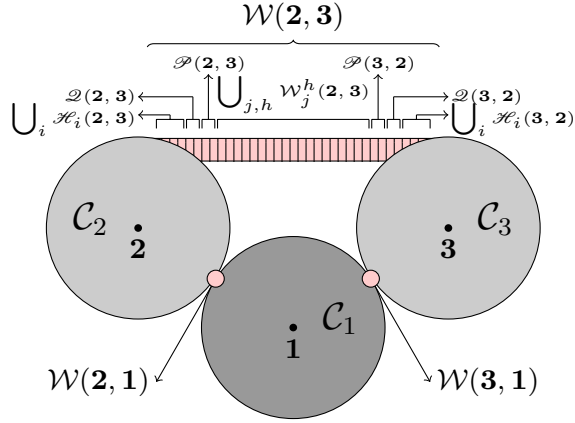

\begin{theorem}[Minimal gates for $\mathbf2\to\mathbf1$ and $\mathbf3\to\mathbf1$]\label{t_gate1}
Fix $\mathbf{m}\in\{\mathbf2,\mathbf3\}$ and consider the metastable transition from $\mathbf{m}$ to $\mathbf1$. Take any set $\mathscr{A}$ such that
\[
\mathscr{A}\in\big\{\mathscr{W}_j^h(\mathbf2,\mathbf3)\big\}_{j,h}\cup\big\{\mathscr{Q}(\mathbf2,\mathbf3),\mathscr{P}(\mathbf2,\mathbf3),\mathscr{P}(\mathbf3,\mathbf2),\mathscr{Q}(\mathbf3,\mathbf2)\big\}\cup\big\{\mathscr{H}_i(\mathbf2,\mathbf3)\big\}_i\cup\big\{\mathscr{H}_i(\mathbf3,\mathbf2)\big\}_i,
\]
where the collections are over all $2\le j\le L-3$, $1\le h\le K-1$ and $1\le i\le K-3$. Then,
\begin{enumerate}
\item[\emph{(a)}] $\mathcal{W}(\mathbf2,\mathbf1)\cup\mathcal{W}(\mathbf3,\mathbf1)$ is a minimal gate;
\item[\emph{(b)}] $\mathcal{W}(\mathbf{m},\mathbf1)\cup\mathscr{A}$ is a minimal gate;
\item[\emph{(c)}] moreover, these are all the minimal gates in the sense that
\begin{align*}
\mathcal{G}(\mathbf2,\mathbf1)=\mathcal{G}(\mathbf3,\mathbf1)=\mathcal{W}(\mathbf 2,\mathbf1)\cup\mathcal{W}(\mathbf 3,\mathbf1)\cup\mathcal{W}(\mathbf{2},\mathbf{3}).
\end{align*}
\end{enumerate}
\end{theorem}

Next, we state a theorem regarding the minimal gates for the transition between $\mathbf2$ and $\mathbf3$.

\begin{theorem}[Minimal gates for $\mathbf2\to\mathbf3$]\label{t_gate2}
Consider the transition from $\mathbf2$ to $\mathbf3$. Take any set
\[
\mathscr{A}\in\big\{\mathscr{W}_j^h(\mathbf2,\mathbf3)\big\}_{j,h}\cup\big\{\mathscr{Q}(\mathbf2,\mathbf3),\mathscr{P}(\mathbf2,\mathbf3),\mathscr{P}(\mathbf3,\mathbf2),\mathscr{Q}(\mathbf3,\mathbf2)\big\}\cup\big\{\mathscr{H}_i(\mathbf2,\mathbf3)\big\}_i\cup\big\{\mathscr{H}_i(\mathbf3,\mathbf2)\big\}_i,
\]
as in Theorem \ref{t_gate1}. Then,
\begin{enumerate}
\item[\emph{(a)}] $\mathcal{W}(\mathbf2,\mathbf1)\cup\mathscr{A}$ is a minimal gate;
\item[\emph{(b)}] $\mathcal{W}(\mathbf3,\mathbf1)\cup\mathscr{A}$ is a minimal gate;
\item[\emph{(c)}] these are all the minimal gates in the sense that
\begin{align*}
\mathcal{G}(\mathbf2,\mathbf3)=\mathcal{W}(\mathbf 2,\mathbf1)\cup\mathcal{W}(\mathbf 3,\mathbf1)\cup\mathcal{W}(\mathbf{2},\mathbf{3}).
\end{align*}
\end{enumerate}
\end{theorem}

We prove Theorems \ref{t_gate1} and \ref{t_gate2} in Section \ref{Sec6.4}.\medskip{}

Finally, in the last result of this section we show that during the metastable transition, the process typically visits the corresponding gates identified in Theorems \ref{t_gate1} and \ref{t_gate2}.

\begin{corollary}\label{c_gate}
Take a set $\mathscr{A}$ in the collection
\[
\big\{\mathscr{W}_j^h(\mathbf2,\mathbf3)\big\}_{j,h}\cup\big\{\mathscr{Q}(\mathbf2,\mathbf3),\mathscr{P}(\mathbf2,\mathbf3),\mathscr{P}(\mathbf3,\mathbf2),\mathscr{Q}(\mathbf3,\mathbf2)\big\}\cup\big\{\mathscr{H}_i(\mathbf2,\mathbf3)\big\}_i\cup\big\{\mathscr{H}_i(\mathbf3,\mathbf2)\big\}_i
\]
as we did in Theorems \ref{t_gate1} and \ref{t_gate2}.
\begin{enumerate}
\item[\emph{(a)}] For the transition from $\mathbf{m}\in\{\mathbf2,\mathbf3\}$ to $\mathbf1$,
\[
\lim_{\beta\to\infty}\mathbb{P}_\beta[\tau_{\mathcal{W}(\mathbf2,\mathbf1)\cup\mathcal{W}(\mathbf3,\mathbf1)}^{\mathbf{m}}<\tau_\mathbf1^{\mathbf{m}}]=1\;\;\;\;\text{and}\;\;\;\;\lim_{\beta\to\infty}\mathbb{P}_\beta[\tau_{\mathcal{W}(\mathbf{m},\mathbf1)\cup\mathscr{A}}^{\mathbf{m}}<\tau_{\mathbf1}^{\mathbf{m}}]=1.
\]
\item[\emph{(b)}] For the transition $\mathbf2\to\mathbf3$,
\[
\lim_{\beta\to\infty}\mathbb{P}_\beta[\tau_{\mathcal{W}(\mathbf2,\mathbf1)\cup\mathscr{A}}^{\mathbf2}<\tau_\mathbf3^{\mathbf2}]=1\;\;\;\;\text{and}\;\;\;\;\lim_{\beta\to\infty}\mathbb{P}_\beta[\tau_{\mathcal{W}(\mathbf3,\mathbf1)\cup\mathscr{A}}^{\mathbf2}<\tau_{\mathbf3}^{\mathbf2}]=1.
\]
\end{enumerate}
\end{corollary}

\begin{proof}
By Theorems \ref{t_gate1} and \ref{t_gate2}, the four sets in the subscripts are gates for the corresponding transitions. Thus, \cite[Theorem 5.4]{manzo2004essential} implies the desired equations.
\end{proof}

\section{Projection operator}\label{Sec4}

In this section, we introduce the notion of projection operators which act on the configuration space $\mathcal X$. These operators are extremely useful for analyzing the energy landscape, especially when we want to focus only on  two given spin values. First, we introduce some notation which will be useful in this section and following ones. For each $\sigma\in\mathcal{X}$ and $i,j\in S$, we define the number of spins $i$ in $\sigma$ as
\begin{align}\label{e_Ni}
\mathfrak{N}_i(\sigma):=\big|\{v\in V:\ \sigma(v)=i\}\big|.
\end{align}
Then we define, for $n\ge0$, the set of configurations which have exactly $n$ spins $i$:
\begin{align}\label{e_Vni}
\mathscr{V}_n^i:=\{\eta\in\mathcal{X}:\ \mathfrak{N}_i(\eta)=n\}.
\end{align}
Moreover, for an edge $e\in E$, we say that $e$ is an \textit{$ij$-edge} of $\sigma$ if the corresponding two spins are $i$ and $j$ in $\sigma$. We write
\begin{align}\label{e_nij}
\mathfrak{n}_{ij}(\sigma):=\big|\{e\in E: \ e\text{ is an }ij\text{-edge of }\sigma\}\big|.
\end{align}
Finally, for spin values $r,s\in S$, we define the \textit{projection operator} $\mathscr{P}^{rs}:\mathcal{X}\rightarrow\mathcal{X}$ as
\begin{align}\label{e_proj}
(\mathscr{P}^{rs}\sigma)(x)=\begin{cases}
s, & \text{if }\sigma(x)=r,\\
\sigma(x), & \text{if }\sigma(x)\ne r.
\end{cases}
\end{align}
The operator $\mathscr{P}^{rs}$ projects all spins $r$ to $s$ and preserves all the other spins. Intuitively, one would expect the projected configuration to have lower energy than the original configuration, since all disagreeing edges between $r$ and $s$ disappear. This is in fact the case, unless the spin value $r$ is more stable than $s$ (for example, if $r=1$ and $s=2$), in which case the projected configuration may still have higher energy than the original configuration.

Two projections which are important for us are $\mathscr{P}^{32}$ and $\mathscr{P}^{12}$. We begin by analyzing $\mathscr{P}^{32}$. The analysis is simpler because spins $2$ and $3$ have the same level of stability. 

\begin{lemma}[Projection $3\to2$]\label{l_P32}
For any $\sigma\in\mathcal{X}$, we have
\begin{align}
H(\mathscr{P}^{32}\sigma)\le H(\sigma).
\end{align}
Moreover, equality holds if and only if $\mathfrak{n}_{23}(\sigma)=0$.
\end{lemma}

\begin{proof}
By the definition of $\mathscr{P}^{32}$, it is easy to check that $\mathfrak{n}_{11}(\mathscr{P}^{32}\sigma)=\mathfrak{n}_{11}(\sigma)$,  $\mathfrak{n}_{12}(\mathscr{P}^{32}\sigma)=\mathfrak{n}_{12}(\sigma)+\mathfrak{n}_{13}(\sigma)$, $\mathfrak{n}_{13}(\mathscr{P}^{32}\sigma)=0$, $\mathfrak{n}_{22}(\mathscr{P}^{32}\sigma)=\mathfrak{n}_{22}(\sigma)+\mathfrak{n}_{23}(\sigma)+\mathfrak{n}_{33}(\sigma)$ and $\mathfrak{n}_{23}(\mathscr{P}^{32}\sigma)=\mathfrak{n}_{33}(\mathscr{P}^{32}\sigma)=0$. Thus, using the interpretation \eqref{e_ham'} we may write
\[
H(\mathscr{P}^{32}\sigma)=H(\mathbf2)-\gamma_1\mathfrak{n}_{11}(\sigma)+\gamma_{12}[\mathfrak{n}_{12}(\sigma)+\mathfrak{n}_{13}(\sigma)]
\]
and
\[
H(\sigma)=H(\mathbf2)-\gamma_1\mathfrak{n}_{11}(\sigma)+\gamma_{12}\mathfrak{n}_{12}(\sigma)+\gamma_{13}\mathfrak{n}_{13}(\sigma)+\gamma_{23}\mathfrak{n}_{23}(\sigma).
\]
Recalling that $\gamma_{12}=\gamma_{13}$ from \eqref{e_Jii} and \eqref{e_gammadef}, we deduce
\[
H(\mathscr{P}^{32}\sigma)-H(\sigma)=-\gamma_{23}\mathfrak{n}_{23}(\sigma)\le0.
\]
This proves the first statement. Moreover, from \eqref{e_gammadef} it follows that the equality holds if and only if $\mathfrak{n}_{23}(\sigma)=0$. This concludes the proof.
\end{proof}

For a configuration $\sigma\in\mathcal{X}$, we say that a row (resp.~column) in $\Lambda$ is a \textit{horizontal bridge} (resp.~\textit{vertical bridge}) of $\sigma$ if all spins on it have the same value. If there exist both a horizontal bridge and a vertical bridge simultaneously (in which case the spin value must be the same), we call the union a \textit{cross}.

Given a configuration $\sigma\in\mathcal{X}$ and a spin $r\in S$, we say that $A\subseteq V$ is an \textit{$r$-cluster} of $\sigma$ if it is a maximal connected subset of $V$ on which all spins are $r$; i.e., if $A$ is connected, $\sigma(v)=r$ for all $v\in A$ and $\sigma(v)\ne r$ for all $v\in \partial A$, where $\partial A$ is the \textit{outer boundary} of $A$:
\begin{align}\label{e_bdry}
\partial A:=\big\{w\in V\setminus A:\ \exists w'\in A,\;\{w,w'\}\in E\big\}.
\end{align}

Next we deal with the projection $\mathscr{P}^{12}$. In this case, as we mentioned in the beginning of this section, the statement becomes much more restrictive because spin $1$ is more stable than spin $2$. Indeed, the next lemma shows that the projection $\mathscr{P}^{12}$ lowers the energy of a configuration only when this has a low number of spins $1$.

\begin{lemma}[Projection $1\to2$]\label{l_P12}
Suppose that $\sigma\in\mathcal{X}$ satisfies $H(\sigma)-H(\mathbf2)\le\Gamma^\star$ and $\mathfrak{N}_1(\sigma)\le\ell^{\star2}$. Then, we have
\begin{align}
H(\mathscr{P}^{12}\sigma)\le H(\sigma).
\end{align}
Moreover, equality holds if and only if $\mathfrak{N}_1(\sigma)=0$.
\end{lemma}

\begin{proof}
Abbreviate $\tilde\sigma:=\mathscr{P}^{12}\sigma$ and $\bar\sigma:=\mathscr{P}^{32}\sigma$. Clearly, we have $\mathfrak{N}_1(\bar\sigma)=\mathfrak{N}_1(\sigma)$. We divide into three cases according to the type of $1$-bridges of $\sigma$.
\begin{itemize}
\item If $\sigma$ has an $1$-cross, then we can regard $\bar\sigma$ as a configuration of spins $2$ in the sea of spins $1$. Applying the well-known isoperimetric inequality (e.g. \cite[Corollary 2.5]{alonso1996three}) to the $2$-clusters, the perimeter is at least
\[
4\sqrt{\mathfrak{N}_2(\bar\sigma)}=4\sqrt{KL-\mathfrak{N}_1(\bar\sigma)}\ge4\sqrt{KL-\ell^{\star2}}.
\]
Since the perimeter is exactly $\mathfrak{n}_{12}(\sigma)$, we have
\begin{align*}
\mathfrak{n}_{12}(\bar\sigma)\ge4\sqrt{KL-\ell^{\star2}}.
\end{align*}
Next, since $\mathfrak{N}_1(\bar\sigma)\le\ell^{\star2}$ and each vertex with spin $1$ is contained in at most four $11$-edges, we have that
\begin{align}\label{e_21LBn11}
\mathfrak{n}_{11}(\bar\sigma)\le\frac12\times4\times \mathfrak{N}_1(\bar\sigma)\le 2\ell^{\star2},
\end{align}
where the factor $\frac12$ appears because there are two spins $1$ in a single $11$-edge. Therefore, by \eqref{e_ham'} we deduce that
\begin{align*}
H(\bar\sigma)-H(\mathbf2)=-\gamma_1\mathfrak{n}_{11}(\bar\sigma)+\gamma_{12}\mathfrak{n}_{12}(\bar\sigma)\ge4\sqrt{KL-\ell^{\star2}}\gamma_{12}-2\ell^{\star2}\gamma_1.
\end{align*}
This is strictly bigger than $\Gamma^{\star}=4\ell^{\star}\gamma_{12}-(2\ell^{\star2}-4\ell^{\star}+2)\gamma_1$ (cf. \eqref{Gammametastable}).
Indeed, we have to verify that
\begin{align*}
[2\sqrt{KL-\ell^{\star2}}-2\ell^{\star}]\gamma_{12}>(2\ell^{\star}-1)\gamma_1.
\end{align*}
This holds since by Assumption \ref{assump}-\textbf{C} we have $\gamma_{12}>\gamma_1$, and since the lattice size $K$ and $L$ are assumed to be sufficiently larger than the critical length $\ell^\star$ we have that $2\sqrt{KL-\ell^{\star2}}-2\ell^\star\ge2\ell^\star-1$. Therefore, in this case, by Lemma \ref{l_P32} we always have $H(\sigma)-H(\mathbf2)\ge H(\bar\sigma)-H(\mathbf2)>\Gamma^\star$, which contradicts the assumption.
\item If $\sigma$ has an $1$-bridge but no $1$-cross, then there are
at least $K$ rows or $L$ columns in $\bar\sigma$ which are not bridges. In each non-bridge of $\bar\sigma$, there are at least two $12$-edges. Moreover, by the same logic as in \eqref{e_21LBn11} there are at most $2\ell^{\star2}$ $11$-edges. Therefore, we have
\begin{align*}
H(\bar\sigma)-H(\mathbf2)\ge2\min\{K,L\}\gamma_{12}-2\ell^{\star2}\gamma_1 = 2K \gamma_{12}-2\ell^{\star2}\gamma_1.
\end{align*}
Similarly, this is strictly bigger than $\Gamma^{\star}$ due to Assumption \ref{assump}-\textbf{C}, so that we get a contradiction.
\item If $\sigma$ does not have an $1$-bridge, then all $1$-clusters of $\sigma$ do not wrap around the periodic lattice, and thus we may apply the isoperimetric inequality. By the projection $\sigma\mapsto\tilde\sigma=\mathscr{P}^{12}\sigma$, only the $11$-, $12$- and $13$-edges of $\sigma$ are affected. Thus, by \eqref{e_ham'} we may write
\begin{align*}
H(\tilde\sigma)-H(\sigma)=\gamma_1\mathfrak{n}_{11}(\sigma) -\gamma_{12}\mathfrak{n}_{12}(\sigma) +(\gamma_{23}-\gamma_{13})\mathfrak{n}_{13}(\sigma).
\end{align*}
Since $\gamma_{12}=\gamma_{13}>0$ and $\gamma_{23}>0$ by \eqref{e_gammadef},
\begin{align*}
H(\tilde\sigma)-H(\sigma)\le \gamma_1\mathfrak{n}_{11}(\sigma)-[\mathfrak{n}_{12}(\sigma)+\mathfrak{n}_{13}(\sigma)]\times(\gamma_{12}-\gamma_{23}).
\end{align*}
By \eqref{e_nijNi}, it holds that $2\mathfrak{n}_{11}(\sigma)+\mathfrak{n}_{12}(\sigma)+\mathfrak{n}_{13}(\sigma)=4\mathfrak{N}_1(\sigma)$. Again by the isoperimetric inequality, $\mathfrak{n}_{12}(\sigma)+\mathfrak{n}_{13}(\sigma)\ge4\sqrt{\mathfrak{N}_1(\sigma)}$ and
thus the last-displayed formula is bounded above by
\begin{align*}
2\big[\mathfrak{N}_1(\sigma)-\sqrt{\mathfrak{N}_1(\sigma)}\big]\gamma_1-4\sqrt{\mathfrak{N}_1(\sigma)}(\gamma_{12}-\gamma_{23}).
\end{align*}
Thus to prove $H(\tilde\sigma)\le H(\sigma)$, it suffices to verify that
\[
2(\gamma_{12}-\gamma_{23})>(\sqrt{\mathfrak{N}_1(\sigma)}-1)\gamma_1.
\]
Since $\mathfrak{N}_1(\sigma)\le\ell^{\star2}$ and $\ell^\star=\lceil\frac{2\gamma_{12}+\gamma_1}{2\gamma_1}\rceil<\frac{2\gamma_{12}+\gamma_1}{2\gamma_1}+1$ (cf. \eqref{e_consts}), we have
\begin{align}\label{e_opt}
(\sqrt{\mathfrak{N}_1(\sigma)}-1)\gamma_1<\frac{2\gamma_{12}+\gamma_1}{2\gamma_1}\times\gamma_1=\gamma_{12}+\frac{\gamma_1}2\le2(\gamma_{12}-\gamma_{23}).
\end{align}
The last inequality follows from Assumption \ref{assump}-\textbf{C}. Therefore, we conclude the proof of the first statement.
\end{itemize}
Finally, we investigate the equality condition. Carefully inspecting the proof above, the equality holds if and only if $\sigma$ does not have an $1$-bridge, $\mathfrak{n}_{12}(\sigma)=0$ and $\mathfrak{N}_1(\sigma)=0$. This is equivalent to saying that $\mathfrak{N}_1(\sigma)=0$, in which case $\tilde\sigma=\sigma$ and thus the equality is obvious.
\end{proof}

\begin{remark}
The last inequality in \eqref{e_opt} is where the exact condition \textbf{C} is required. Indeed,
\[
\gamma_{12}+\frac{\gamma_1}2\le2(\gamma_{12}-\gamma_{23})\;\;\;\;\text{if and only if}\;\;\;\;2\gamma_{12}\ge4\gamma_{23}+\gamma_1.
\]
\end{remark}

\section{Comparison with the original Ising path}\label{Sec5}

In this section, we prove that if a path is restricted to only two spins, then it is equivalent, in the sense of communication height and gates, to the original Ising path with positive external field (if the spins are $\{1,2\}$ or $\{1,3\}$) or the original Ising path with zero external field (if the spins are $\{2,3\}$).

We say that a path $\omega$ is an \textit{$ij$-path} if it involves spins $i$ and $j$ only, i.e., $\omega_n(v)\in\{i,j\}$ for all $n$ and $v\in V$. According to the notation \eqref{e_Xij}, this is equivalent to $\omega\subseteq\mathcal{X}^{ij}$.

\begin{proposition}[Gate property 1]\label{p_gatepos}
Suppose that $\omega=(\omega_n)_{n=0}^N$ is an $1r$-path from $\mathbf{r}$ to $\mathbf1$ for some $r\in\{2,3\}$.
\begin{enumerate}
\item[\emph{(a)}] $\Phi_\omega\ge H(\mathbf2)+\Gamma^\star$.
\item[\emph{(b)}] If $\Phi_\omega=H(\mathbf2)+\Gamma^\star$, then $
\omega\cap\mathcal{W}(\mathbf{r},\mathbf1)\ne\varnothing$ (cf. \eqref{e_Wm1}).
\end{enumerate}
\end{proposition}

\begin{remark}
Item (a) in Proposition \ref{p_gatepos}, along with the reference path constructed in Section \ref{SecA.4}, corresponds to the communication height. Moreover, item (b) implies that $\mathcal{W}(\mathbf{r},\mathbf1)$ works as a gate.
\end{remark}

\begin{proof}[Proof of Proposition \ref{p_gatepos}]
For $\sigma\in\mathcal{X}^{1r}$, we have by \eqref{e_ham'} and \eqref{e_nijNi} that
\[
H(\sigma)=H(\mathbf2)-\gamma_1\mathfrak{n}_{11}(\sigma)+\gamma_{1r}\mathfrak{n}_{1r}(\sigma)=H(\mathbf2)-2\gamma_1\mathfrak{N}_1(\sigma)+\Big(\frac12\gamma_1+\gamma_{1r}\Big)\mathfrak{n}_{1r}(\sigma).
\]
Therefore, $\mathcal{X}^{1r}$ is isomorphic to the original Ising configuration space $\{+1,-1\}^V$ via correspondence of spins $1\leftrightarrow +1$ and $r\leftrightarrow -1$, and moreover the Hamiltonian is the same as the original Ising Hamiltonian with coupling constant  $J:=\frac12\gamma_1+\gamma_{1r}$ and external field $h:=2\gamma_1$, translated by a fixed real number $H(\mathbf2)$. Therefore, provided that the Glauber transitions happen inside the restricted set $\mathcal{X}^{1r}$, we may refer to the previous well-known results. Item (a) is equivalent to the lower bound of the communication height, which is provided in \cite[Theorem 3]{neves1991critical}. Moreover, item (b) is equivalent to saying that the collection of critical configurations in the Ising model is a gate for the metastable transition. This is also a very classic result which was first proved in \cite[Theorem 5.10]{manzo2004essential}.
\end{proof}

Similarly, we argue that a $23$-path behaves in the same way as the original Ising path with zero external field. Recall the sets defined in Section \ref{Sec3.3}.

\begin{proposition}[Gate property 2]\label{p_gatezero}
Suppose that $\omega=(\omega_n)_{n=0}^N$ is a $23$-path from $\mathbf2$ to $\mathbf3$.
\begin{enumerate}
\item[\emph{(a)}] $\Phi_\omega\ge H(\mathbf2)+\Gamma^\star$.
\item[\emph{(b)}] Suppose that $\Phi_\omega=H(\mathbf2)+\Gamma^\star$. Take
\[
\mathscr{A}\in\big\{\mathscr{W}_j^h(\mathbf2,\mathbf3)\big\}_{j,h}\cup\big\{\mathscr{Q}(\mathbf2,\mathbf3),\mathscr{P}(\mathbf2,\mathbf3),\mathscr{P}(\mathbf3,\mathbf2),\mathscr{Q}(\mathbf3,\mathbf2)\big\}\cup\big\{\mathscr{H}_i(\mathbf2,\mathbf3)\big\}_i\cup\big\{\mathscr{H}_i(\mathbf3,\mathbf2)\big\}_i,
\]
where the collections are over all $2\le j\le L-3$, $1\le h\le K-1$ and $1\le i\le K-3$. Then,
\[
\omega\cap\mathscr{A}\ne\varnothing.
\]
\end{enumerate}
\end{proposition}

\begin{proof}
We may use the same argument as in the proof of Proposition \ref{p_gatepos}, with the modification that here we identify $\mathcal{X}^{23}$ with the Ising model with zero external field \cite{bet2021critical,kim2021metastability,nardi2019tunneling}. This is possible since we have an alternative expression of energy for $\sigma\in\mathcal{X}^{23}$, which is
\[
H(\sigma)=H(\mathbf2)+\gamma_{23}\mathfrak{n}_{23}(\sigma).
\]
Thus, item (a) is equivalent to \cite[Proposition 2.6]{nardi2019tunneling} and item (b) is proved in \cite[Theorem 3.3]{bet2021critical}.
\end{proof}

\section{Proofs}\label{Sec6}
%

\subsection{Communication height}\label{Sec6.1}

In this subsection we prove Theorem \ref{t_comheight}. We divide the proof into three propositions. Specifically, Proposition \ref{p_123UB} establishes the upper bound $\Gamma^\star+H(\mathbf2)$ for the communication heights, and Propositions \ref{p_21LB} and \ref{p_23LB} establish the matching lower bound.

\begin{proposition}\label{p_123UB}
It holds that $\Gamma(\mathbf{2},\mathbf{1})\le\Gamma^{\star}$, $\Gamma(\mathbf{3},\mathbf{1})\le\Gamma^{\star}$ and $\Gamma(\mathbf2,\mathbf3)\le\Gamma^{\star}$.
\end{proposition}

\begin{proof}
The reference paths constructed in Section \ref{SecA.4} have height $\Gamma^\star+H(\mathbf2)$. From this, the upper bounds immediately follow.
\end{proof}

\begin{proposition}\label{p_21LB}
We have $\Gamma(\mathbf{2},\mathbf{1})=\Gamma(\mathbf{3},\mathbf{1})\ge\Gamma^{\star}$.
\end{proposition}

\begin{proof}
By the model symmetry between spins $2$ and $3$, it suffices to prove that $\Gamma(\mathbf{2},\mathbf{1})\ge\Gamma^{\star}$. To this end, take an arbitrary path $\omega=(\omega_n)_{n=0}^N$ so that $\omega_0=\mathbf2$ and $\omega_N=\mathbf1$. Then, define (cf. \eqref{e_proj})
\[
\bar\omega_n:=\mathscr{P}^{32}\omega_n\;\;\;\;\text{for each }0\le n\le N.
\]
It holds automatically that $\bar\omega_0=\mathbf2$, $\bar\omega_N=\mathbf1$ and $\bar\omega_n\in\mathcal{X}^{12}$ (cf. \eqref{e_Xij}). Since $\omega_n$ and $\omega_{n+1}$ differ in exactly one site, $\bar\omega_n$ and $\bar\omega_{n+1}$ differ in at most one site. These observations imply that $\bar\omega=(\bar\omega_n)_{n=0}^N$ is an $12$-path (possibly with non-updating instances) from $\mathbf2$ to $\mathbf1$. Therefore, Proposition \ref{p_gatepos}-(a) implies that
\[
\Phi_{\bar\omega}\ge H(\mathbf2)+\Gamma^\star.
\]
Then, Lemma \ref{l_P32} implies that $H(\omega_n)\ge H(\bar\omega_n)$ for all $n$ and thus
\[
\Phi_\omega \ge \Phi_{\bar\omega} \ge H(\mathbf2)+\Gamma^\star.
\]
Since our choice of $\omega$ was arbitrary, we deduce that $\Gamma(\mathbf2,\mathbf1)\ge\Gamma^\star$.
\end{proof}

All it remains is to provide a lower bound for $\Gamma(\mathbf2,\mathbf3)$. Before this, we state a lemma. Recall the notion of $\mathscr{V}_n^i$ defined in \eqref{e_Vni}. Moreover, for $m\in\{2,3\}$ we define
\[
\mathcal{W}'(\mathbf{m},\mathbf1):=\hat{B}_{\ell^\star-1,\ell^\star}^1(m,1)\cup B_{\ell^\star,\ell^\star-1}^1(m,1).
\]
Intuitively, $\mathcal{W}'(\mathbf{m},\mathbf1)$ is the collection of configurations that have a protuberance of spin $1$ on a \textit{wrong} side of the rectangular $1$-cluster, in the sense that we cannot proceed further to reach $\mathbf1$ without returning to a protocritical configuration in $R_{\ell^\star-1,\ell^\star}(m,1)\cup R_{\ell^\star,\ell^\star-1}(m,1)$.

\begin{lemma}\label{l_critical}
Suppose that $\eta\in\mathscr{V}_{\ell^{\star}(\ell^{\star}-1)+1}^1$.
\begin{enumerate}
\item[\emph{(a)}] It holds that $H(\eta)-H(\mathbf2)\ge\Gamma^{\star}$.
\item[\emph{(b)}] Equality in \emph{(a)} holds if and only if $\eta\in \bigcup_{m=2}^3[\mathcal{W}(\mathbf{m},\mathbf1)\cup \mathcal{W}'(\mathbf{m},\mathbf1)]$.
\end{enumerate}
\end{lemma}

\begin{proof}
(a) We abbreviate $\bar\eta:=\mathscr{P}^{32}\eta\in\mathcal{X}^{12}$, where clearly $\bar\eta\in\mathscr{V}_{\ell^\star(\ell^\star-1)+1}^1$. We divide into three cases as we did in the proof of Lemma \ref{l_P12}.
\begin{itemize}
\item If $\bar\eta$ has an $1$-cross, then we may regard $\bar\eta$ as a configuration of spins $2$ in the sea of spins $1$. Recalling \eqref{e_Ni} and \eqref{e_nij} and applying the isoperimetric inequality (cf. \cite[Corollary 2.5]{alonso1996three}), we have
\begin{align}\label{e_21LBn12}
\mathfrak{n}_{12}(\bar\eta)\ge4\sqrt{\mathfrak{N}_2(\bar\eta)}=4\sqrt{KL-\mathfrak{N}_1(\bar\eta)}=4\sqrt{KL-\ell^{\star}(\ell^{\star}-1)-1}.
\end{align}
Using the same argument as in \eqref{e_21LBn11}, it holds that $\mathfrak{n}_{11}(\bar\eta)\le 2\mathfrak{N}_1(\bar\eta)=2\ell^\star(\ell^\star-1)+2$. Thus, by \eqref{e_ham'} we have
\begin{align*}
H(\bar\eta)-H(\mathbf2)=-\gamma_1\mathfrak{n}_{11}(\bar\eta)+\gamma_{12}\mathfrak{n}_{12}(\bar\eta)\ge4\sqrt{KL-\ell^{\star2}+\ell^{\star}-1}\gamma_{12}-(2\ell^{\star2}-2\ell^{\star}+2)\gamma_1.
\end{align*}
By Assumption \ref{assump}-\textbf{C}, it can be shown that this is strictly larger than $\Gamma^{\star}=4\ell^{\star}\gamma_{12}-(2\ell^{\star2}-4\ell^{\star}+2)\gamma_1$ via a similar argument as in the first case in the proof of Lemma \ref{l_P12}. Therefore, by Lemma \ref{l_P32} we obtain
\[
H(\eta)\ge H(\bar\eta)>H(\mathbf2)+\Gamma^\star.
\]
\item If $\bar\eta$ has an $1$-bridge but no $1$-cross, then proceeding similarly to the second case in proof of Lemma \ref{l_P12}, we obtain
\begin{align*}
H(\eta)-H(\mathbf2)\ge H(\bar\eta)-H(\mathbf2)\ge2K\gamma_{12}-(2\ell^{\star2}-2\ell^{\star}+2)\gamma_1>\Gamma^\star.
\end{align*}
\item If $\bar\eta$ does not have an $1$-bridge, then all $1$-clusters of $\bar\eta$ are in the sea of spins $2$. By \eqref{e_nijNi} and since $\mathfrak{n}_{13}(\bar\eta)=0$ and $\mathfrak{N}_1(\bar\eta)=\ell^\star(\ell^\star-1)+1$, we have
\[
2\mathfrak{n}_{11}(\bar\eta)+\mathfrak{n}_{12}(\bar\eta)=4\ell^\star(\ell^\star-1)+4.
\]
Then by \eqref{e_ham'},
\begin{align*}
H(\bar\eta)-H(\mathbf2)=\gamma_{12}\mathfrak{n}_{12}(\bar\eta)-\gamma_1\mathfrak{n}_{11}(\bar\eta)= (4\ell^{\star2}-4\ell^\star+4)\gamma_{12}-\mathfrak{n}_{11}(\bar\eta)(\gamma_1+2\gamma_{12}).
\end{align*}
Since $\mathfrak{N}_1(\bar\eta)=\ell^\star(\ell^\star-1)+1$, by the isoperimetric inequality, the perimeter of the $1$-clusters is at least $4\ell^{\star}$.
This is equivalent to $\mathfrak{n}_{12}(\bar\eta)\ge4\ell^{\star}$ and thus equivalent to $\mathfrak{n}_{11}(\bar\eta)\le 2\ell^{\star2}-4\ell^\star+2$. Hence,
\begin{align*}
H(\bar\eta)-H(\mathbf2)\ge(4\ell^{\star2}-4\ell^\star+4)\gamma_{12}-(2\ell^{\star2}-4\ell^\star+2)(\gamma_1+2\gamma_{12})=\Gamma^{\star},
\end{align*}
which concludes the proof of item (a) since, by Lemma \ref{l_P32}, we have $H(\eta)\ge H(\bar\eta)$.
\end{itemize}
(b) By the proof above, the equality holds if and only if $\bar\eta$ does not have an $1$-bridge, the equality in the isoperimetric inequality holds and $H(\eta)=H(\bar\eta)$. By Lemma \ref{l_P32}, this is equivalent to
\[
\eta=\bar\eta\in \bigcup_{m=2}^3 [\mathcal{W}(\mathbf{m},\mathbf1)\cup \mathcal{W}'(\mathbf{m},\mathbf1)],
\]
which proves item (b).
\end{proof}

Now we are ready to prove the lower bound of $\Gamma(\mathbf2, \mathbf3)$.

\begin{proposition}\label{p_23LB}
It holds that $\Gamma(\mathbf{2},\mathbf{3})\ge\Gamma^{\star}$.
\end{proposition}

\begin{proof}
Assume by contradiction that there exists a path $(\omega_{n})_{n=0}^{N}$
from $\mathbf{2}$ to $\mathbf{3}$ such that $H(\omega_{n})-H(\mathbf2)<\Gamma^{\star}$
for all $n$. Recall $\mathscr{P}^{12}$ from \eqref{e_proj} and define
\[
\tilde\omega_n:=\mathscr{P}^{12}\omega_n.
\]
It is clear that $\tilde\omega_0=\mathbf{2}$ and $\tilde\omega_N=\mathbf{3}$. Thus, $(\tilde\omega_{n})_{n=0}^{N}$ is a $23$-path (possibly with some non-updating instances) from $\mathbf{2}$ to $\mathbf{3}$. Then, Proposition \ref{p_gatezero}-(a) indicates that
\[
\max_{0\le n\le N}H(\tilde\omega_n)\ge H(\mathbf2)+\Gamma^\star.
\]
To conclude the proof, it is enough to show that $\mathfrak{N}_1(\omega_n)\le\ell^{\star2}$ for all $n$. Indeed, if the claim holds then we may apply Lemma \ref{l_P12} to each $\omega_n$, so that along with the last display we have
\[
\max_{0\le n\le N}H(\omega_n)\ge \max_{0\le n\le N}H(\tilde\omega_n)\ge H(\mathbf2)+\Gamma^\star.
\]
This contradicts the original assumption.\medskip{}

It remains to prove the claim. We prove that $\mathfrak{N}_1(\omega_n)\le\ell^\star(\ell^\star-1)$ for all $n$, which is clearly sufficient. If not, $\mathfrak{N}_1(\omega_{M})\ge\ell^\star(\ell^\star-1)+1$ for some $M$. Since $\mathfrak{N}_1(\omega_0)=0$ and $|\mathfrak{N}_1(\omega_{n+1})-\mathfrak{N}_1(\omega_n)|\le1$ for any $n$, there exists $n_0$ such that $\mathfrak{N}_1(\omega_{n_0})=\ell^\star(\ell^\star-1)+1$. Then by Lemma \ref{l_critical}-(a), $H(\omega_{n_0})\ge H(\mathbf 2)+\Gamma^\star$ which contradicts our assumption. Thus, we conclude the proof of the claim.
\end{proof}

\begin{proof}[Proof of Theorem \ref{t_comheight}]
The statement now follows directly from Propositions \ref{p_123UB}, \ref{p_21LB} and \ref{p_23LB}.
\end{proof}

\subsection{Stability level}\label{Sec6.2}

In this subsection, we prove Proposition \ref{p_stablev}, i.e., we estimate the stability level of configurations other than $\mathbf1$, $\mathbf2$ and $\mathbf3$. First we introduce some notation.

\begin{itemize}
\item For $\sigma\in\mathcal{X}$ and $v\in V$, we denote by \textit{tile centered at $v$}, or \textit{$v$-tile}, the collection of five vertices consisting of $v$ and its four nearest neighbors.
\item A $v$-tile is \textit{stable} for $\sigma$ if any spin flip on $v$ from $\sigma(v)$ to any other spin does not decrease the energy.
\item Moreover, we say that a stable $v$-tile is \textit{strictly stable} if any spin flip on $v$ from $\sigma(v)$ to any other spin strictly increases the energy.
\end{itemize}

First, we can characterize all the stable and strictly stable tiles. Since the following lemma can be proved by simple algebra, we omit the proof.

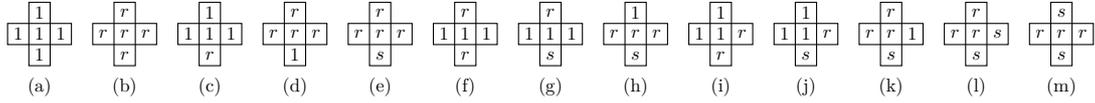
\begin{figure}[ht!]
\centering
\begin{tikzpicture}[scale=0.7,transform shape]
\foreach \i in {0,1.6,3.2,4.8,6.4,8,9.6,11.2,12.8,14.4,16,17.6,19.2}
\draw (\i,0)rectangle(\i+1.2,0.4);
\foreach \i in {0,1.6,3.2,4.8,6.4,8,9.6,11.2,12.8,14.4,16,17.6,19.2}
\draw (\i+0.4,-0.4)rectangle(\i+0.8,0.8);

\draw (0.6,-0.5) node[below] {(a)}(2.2,-0.5) node[below] {(b)}(3.8,-0.5) node[below] {(c)}(5.4,-0.5) node[below] {(d)}(7,-0.5) node[below] {(e)}(8.6,-0.5) node[below] {(f)} (10.2,-0.5) node[below] {(g)} (11.8,-0.5) node[below] {(h)} (13.4,-0.5) node[below] {(i)} (15,-0.5) node[below] {(j)} (16.6,-0.5) node[below] {(k)} (18.2,-0.5) node[below] {(l)} (19.8,-0.5) node[below] {(m)};

\foreach \i in {0.2,0.6,1,3.4,3.8,4.2,8.2,8.6,9,9.8,10.2,10.6,13,13.4,14.6,15,17}\draw (\i,0.2) node {$1$};
\foreach \i in {0.6,3.8,11.8,13.4,15}\draw (\i,0.6) node {$1$};
\foreach \i in {0.6,5.4} \draw (\i,-0.2) node {$1$};
\foreach \i in {1.8,2.2,2.6,5,5.4,5.8,6.6,7,7.4,11.4,11.8,12.2,13.8,15.4,16.2,16.6,17.8,18.2,19.4,19.8,20.2}\draw (\i,0.2) node {$r$};
\foreach \i in {2.2,5.4,7,8.6,10.2,16.6,18.2}\draw (\i,0.6) node {$r$};
\foreach \i in {2.2,3.8,8.6,13.4} \draw (\i,-0.2) node {$r$};
\foreach \i in {18.6}\draw (\i,0.2) node {$s$};
\foreach \i in {19.8}\draw (\i,0.6) node {$s$};
\foreach \i in {7,10.2,11.8,15,16.6,18.2,19.8} \draw (\i,-0.2) node {$s$};
\end{tikzpicture}
\caption{\label{fig5}All possible stable tiles for $r,s\in\{2,3\}$, $r\neq s$. The tiles are depicted up to rotations and reflections.}
\end{figure}\FloatBarrier

\begin{lemma}[Stable and strictly stable tiles]\label{l_stable}
Let $\sigma\in\mathcal{X}$ and $v\in V$. Then, $v$-tile is strictly stable for $\sigma$ if and only if the following statements hold.
\begin{itemize}
\item If $\sigma(v)=1$, then $v$ has at least two neighbors with spin $1$, as in Figure \ref{fig5}-\emph{(a)(c)(f)(g)(i)(j)}.
\item If $\sigma(v)=r\in\{2,3\}$, then $v$ has either at least three neighbors with spin $r$, or exactly two neighbors with spin $r$ and one neighbor with spin $1$, as in Figure \ref{fig5}-\emph{(b)(d)(e)(h)(k)}.
\end{itemize}
Moreover, $v$-tile is stable but not strictly stable for $\sigma$ if and only
if $\sigma(v)=r\in\{2,3\}$ and $v$ has exactly two neighbors with
spin $r$ and no neighbor with spin $1$, as in Figure \ref{fig5}-\emph{(l)(m)}.
\end{lemma}

Given a spin configuration $\sigma$, an edge $e=\{x,y\}\in E$ is called an \textit{interface} of $\sigma$ if $\sigma(x)\ne\sigma(y)$. Moreover, two different clusters $C_1$ and $C_2$ of $\sigma$ are said to \textit{interact with each other} if there exists $v\notin C_{1}\cup C_{2}$
such that $v$ is connected to both $C_{1}$ and $C_{2}$, i.e.,
\begin{align}\label{e_interact}
\exists w_1\in C_1\text{ and }\exists w_2\in C_2\;\;\;\;\text{such that}\;\;\;\;\{v,w_1\}\in E\text{ and }\{v,w_2\}\in E.
\end{align}
Moreover, we define the set of local minima $\mathscr{M}$ as
\begin{align}\label{e_M}
\mathscr{M}:=\big\{\sigma\in\mathcal{X}:\ H(\sigma)<H(\sigma')\text{ for all }\sigma'\ne\sigma\text{ with }P_\beta(\sigma,\sigma')>0\big\},
\end{align}
and the set of plateaux $\bar{\mathscr{M}}$ as 
\begin{align}\label{e_Mbar}
\bar{\mathscr{M}}:=\bigcup_{D\text{ stable plateaux}} D.
\end{align}
Here, a subset $D\subset\mathcal{X}$ is a \textit{stable pleateau} if it is a maximal connected subset of equal energy, so that for any $\sigma\in D$ and $\sigma'\in\partial D$ it holds that $H(\sigma)<H(\sigma')$. It is obvious by definition that
\begin{align}\label{e_Mcond}
\sigma\in\mathscr{M}\;\;\;\;\Leftrightarrow\;\;\;\;\text{every tile is strictly stable for }\sigma.
\end{align}
and
\begin{align}\label{e_MMbarcond}
\sigma\in\mathscr{M}\cup\bar{\mathscr{M}}\;\;\;\;\Rightarrow\;\;\;\;\text{every tile is stable for }\sigma.
\end{align}

First, we prove that all the $1$-clusters of the configurations in $\mathscr{M}\cup\bar{\mathscr{M}}$ must be rectangles.

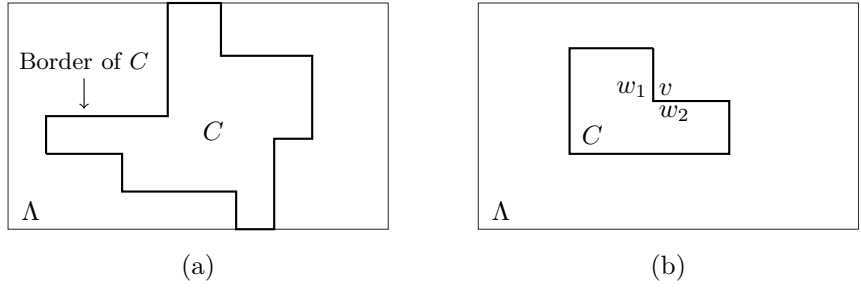
\begin{figure}
\centering
\begin{tikzpicture}
\draw[black!80!white] (0,0) rectangle (5,3);
\draw (0.3,0.25) node {$\Lambda$};
\draw[black,thick] (0.5,1)--(1.5,1)--(1.5,0.5)--(3,0.5)--(3,0)--(3.5,0)--(3.5,1.2)--(4,1.2)--(4,2.3)--(2.8,2.3)--(2.8,3)--(2.1,3)--(2.1,1.5)--(0.5,1.5)--(0.5,1);
\draw (2.7,1.3) node {$C$};
\draw[<-] (1,1.6)--(1,2) node[above] {\small{Border of $C$}};
\draw (2.5,-0.5) node {{(a)}};
\end{tikzpicture}\ \ \ \ \ \ \ \ \ \ 
\begin{tikzpicture}
\draw[black!80!white] (0,0) rectangle (5,3);
\draw[black,thick] (2.3,2.4)--(1.2,2.4) -- (1.2,1)--(3.3,1)--(3.3,1.7)--(2.3,1.7)--(2.3,2.4);
\draw (2.35,1.85) node[left] {$w_1$};
\draw (2.25,1.54) node[right] {$w_2$};
\draw (2.25,1.85) node[right] {$v$};
\draw (1.5,1.25) node {$C$};
\draw (0.3,0.25) node {$\Lambda$};
\draw (2.5,-0.5) node {{(b)}};
\end{tikzpicture}
\caption{\label{fig6} Illustrations regarding the proof of Lemma \ref{l_1cluster}.}
\end{figure}

\begin{lemma}
\label{l_1cluster}Suppose that $\sigma\in\mathscr{M}\cup\bar{\mathscr{M}}$. Then, each $1$-cluster of $\sigma$ is a rectangle.
\end{lemma}

\begin{proof}
We fix $\sigma\in\mathscr{M}\cup\bar{\mathscr{M}}$ so that by \eqref{e_MMbarcond}, all tiles of $\sigma$ must be of the form in Figure \ref{fig5}. We fix an $1$-cluster $C$ of $\sigma$. Consider the \textit{border} of $C$, which is defined as
\[
\big\{\{v,w\}\in E:\ v\in C\text{ and }w\in \partial C\big\}.
\]
We refer to Figure \ref{fig6}-(a) for an illustration. If $C$ is not a rectangle, then there exists at least one internal angle of $\frac{3}{2}\pi$ in the border, as one can see in Figure \ref{fig6}-(b) where $w_1,w_2\in C$ and $v\notin C$. This implies that $\sigma(w_1)=\sigma(w_2)=1$ whereas $\sigma(v)\ne1$. Then by Lemma \ref{l_stable}, $v$-tile is not stable for $\sigma$. This contradicts the fact that $\sigma\in\mathscr{M}\cup\bar{\mathscr{M}}$ due to \eqref{e_MMbarcond}. Therefore, we conclude that $C$ is a rectangle.
\end{proof}

\begin{remark}
It would possible to investigate further the equivalent conditions on the $1$-clusters for a configuration to belong to $\mathscr{M}\cup\bar{\mathscr{M}}$. However, we do not pursue this further because it is not required for proving main results.
\end{remark}

Now, we are ready to prove Proposition \ref{p_stablev}.

\begin{proof}[Proof of Proposition \ref{p_stablev}]
To calculate the stability level of $\eta\in\mathcal{X} \setminus \{\mathbf1,\mathbf2,\mathbf3\}$, suppose first that $\eta\notin\mathscr{M}\cup\bar{\mathscr{M}}$. Then by definition (cf. \eqref{e_M} and \eqref{e_Mbar}), there exists $\eta'\in\mathcal{X}$ such that $H(\eta')<H(\eta)$ and $P_\beta(\eta,\eta')>0$. Thus, clearly $V_\eta=0$. Therefore, we may assume that $\eta\in(\mathscr{M}\cup\bar{\mathscr{M}})\setminus\{\mathbf1,\mathbf2,\mathbf3\}$. It suffices to prove that
\[
V_\eta\le\max\big\{2\gamma_{12}-\gamma_1,2(\ell^\star-1)(\gamma_{23}+\gamma_1),2\gamma_{23}\big\}
\]
since clearly the maximum among the constants in the right-hand side is $2(\ell^\star-1)(\gamma_{23}+\gamma_1)$. We divide into two cases.\medskip{}

\noindent \textbf{(Case 1: $\eta$ has an $1$-cluster)} In this
case, by Lemma \ref{l_1cluster} $\eta$ has an $1$-cluster
$C$ which is a rectangle. Since $\eta\ne\mathbf1$, we have $C\subsetneq V$.
\begin{itemize}
\item \textbf{(If $C$ has a side of length $\ell\ge\ell^\star$)} Since $C\subsetneq V$, we can define a path $\omega=(\omega_0,\dots,\omega_{\ell})$, where $\omega_0=\eta$ and $\omega_{\ell}=\bar\eta$, that flips consecutively those spins adjacent to a side of $C$ of length $\ell$ to spin $1$. Then, a simple computation shows that
\begin{align}\label{e_reccase21}
H(\omega_1)-H(\eta)\le2\gamma_{12}-\gamma_1\;\;\;\;\text{and}\;\;\;\;H(\omega_i)-H(\omega_{i-1})\le-2\gamma_1\text{ for any }2\le i\le\ell.
\end{align}
Thanks to \eqref{e_reccase21}, we have that 
\[
H(\bar\eta)-H(\eta)\le2\gamma_{12}-(2\ell^\star-1)\gamma_1<0.
\]
The last inequality holds since $\ell^\star=\lceil\frac{2\gamma_{12}+\gamma_1}{2\gamma_1}\rceil>\frac{2\gamma_{12}+\gamma_1}{2\gamma_1}$ by Assumption \ref{assump}-\textbf{A}. Moreover, the height of $\omega$ is attained by either $\omega_0=\eta$ or $\omega_1$, which has energy at most $H(\eta)+(2\gamma_{12}-\gamma_1)$. Therefore, we deduce in this case that $V_\eta\le2\gamma_{12}-\gamma_1$.
\item \textbf{(If all sides of $C$ have lengths smaller than $\ell^\star$)} Suppose that $C$ is a rectangle $p\times q$. Since $|\partial C|=2(p+q)$ and all spins on $\partial C$ are either $2$ or $3$, $\mathfrak{N}_2(\partial C)+\mathfrak{N}_3(\partial C)=2(p+q)$. Without loss of generality, we assume that
\begin{align}\label{e_N2partialC}
\mathfrak{N}_2(\partial C)\ge p+q\;\;\;\;\text{and}\;\;\;\;\mathfrak{N}_3(\partial C)\le p+q.
\end{align}
We define a path $\omega=(\omega_n)_{n=0}^{pq}$ from $\omega_0=\eta$ to $\omega_{pq}=:\eta'$ as follows: starting from the upper-left corner of $C$, we update each spin $1$ in $C$ to spin $2$ consecutively in a clockwise manner, until all the spins $1$ in $C$ are updated to spins $2$, see Figure \ref{fig7}.

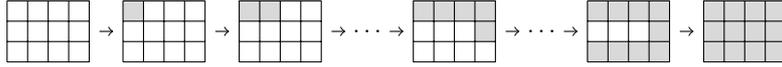
\begin{figure}[ht]
    \centering
        \begin{tikzpicture}[scale=0.9,transform shape]
    \draw[step=0.3cm,color=black] (0,0) grid (1.2,0.9);
    \draw[->] (1.35,0.45) -- (1.55,0.45);
    \end{tikzpicture}
    \begin{tikzpicture}[scale=0.9,transform shape]
    \fill[black!15!white] (0,0.6) rectangle (0.3,0.9);
    \draw[step=0.3cm,color=black] (0,0) grid (1.2,0.9);
    \draw[->] (1.35,0.45) -- (1.55,0.45);
    \end{tikzpicture}
    \begin{tikzpicture}[scale=0.9,transform shape]
    \fill[black!15!white] (0,0.6) rectangle (0.6,0.9);
    \draw[step=0.3cm,color=black] (0,0) grid (1.2,0.9);
    \draw[->] (1.35,0.45) -- (1.55,0.45);
    \draw (1.9,0.45) node {$\dots$};
    \draw[->] (2.2,0.45) -- (2.4,0.45);
    \end{tikzpicture}
    \begin{tikzpicture}[scale=0.9,transform shape]
    \fill[black!15!white] (0,0.6) rectangle (1.2,0.9) (0.9,0.3) rectangle (1.2,0.6);
    \draw[step=0.3cm,color=black] (0,0) grid (1.2,0.9);
    \draw[->] (1.35,0.45) -- (1.55,0.45);
    \draw (1.9,0.45) node {$\dots$};
    \draw[->] (2.2,0.45) -- (2.4,0.45);
    \end{tikzpicture}
    \begin{tikzpicture}[scale=0.9,transform shape]
    \fill[black!15!white] (0,0.6) rectangle (1.2,0.9)(0.9,0) rectangle (1.2,0.6)(0,0) rectangle (0.9,0.3);
    \draw[step=0.3cm,color=black] (0,0) grid (1.2,0.9);
    \draw[->] (1.35,0.45) -- (1.55,0.45);
    \end{tikzpicture}
    \begin{tikzpicture}[scale=0.9,transform shape]
    \fill[black!15!white] (0,0) rectangle (1.2,0.9);
    \draw[step=0.3cm,color=black] (0,0) grid (1.2,0.9);
    \end{tikzpicture}
\caption{\label{fig7} The $4\times3$ rectangle represents the $1$-cluster whose spins are sequentially flipped to $2$.}
\end{figure}

First, we calculate $H(\eta')-H(\eta)$. To this end, note that the edges contained in $V\setminus C$ are not affected by the $pq$ spin updates. Thus,
\[
H(\eta')-H(\eta)=(2pq-p-q)\gamma_1\big[-\mathfrak{N}_2(\partial C)\gamma_{12}+\mathfrak{N}_3(\partial C)(\gamma_{23}-\gamma_{13})\big].
\]
Here, $2pq-p-q$ is the number of internal edges in $C$. Hence, along with \eqref{e_N2partialC},
\begin{align*}
H(\eta')-H(\eta)&\le(2pq-p-q)\gamma_1+(p+q)\gamma_{23}-2(p+q)\gamma_{12}\\
&=2pq\gamma_1-(p+q)(2\gamma_{12}-\gamma_{23}+\gamma_1).
\end{align*}
Subjected to the condition $1\le p,q<\ell^\star$, a simple algebraic computation reveals that the maximum of the last term is attained on $(p,q)=(1,1)$ or $(\ell^\star-1,\ell^\star-1)$. If $(p,q)=(1,1)$ then the value equals $-4\gamma_{12}+2\gamma_{23}<0$ by Assumption \ref{assump}-\textbf{C}, whereas if $(p,q)=(\ell^\star-1,\ell^\star-1)$ then the value becomes
\[
2(\ell^\star-1)[(\ell^\star-1)\gamma_1-2\gamma_{12}+\gamma_{23}-\gamma_1]<2(\ell^\star-1)\Big(-\gamma_{12}+\gamma_{23}-\frac{\gamma_1}{2}\Big)<0,
\]
where we used $\ell^\star<\frac{2\gamma_{12}+\gamma_1}{2\gamma_1}+1$ in the first inequality and Assumption \ref{assump}-\textbf{C} in the second inequality. Thus, we deduce that
\[
H(\eta')-H(\eta)<0.
\]
Therefore, to estimate the stability level of $\eta$ we may focus on the maximal energy attained along the path $\omega:\eta\to\eta'$. By \eqref{e_ham'}, for each $1\le n\le pq$ it holds that
\[
H(\omega_n)-H(\eta)=-\gamma_1[\mathfrak{n}_{11}(\omega_n)-\mathfrak{n}_{11}(\eta)]+\sum_{i<j}\gamma_{ij}[\mathfrak{n}_{ij}(\omega_n)-\mathfrak{n}_{ij}(\eta)].
\]
Note that the edges contained in $V\setminus C$ are unchanged along the path. Thus, we may write
\[
H(\omega_n)-H(\eta)=g_1(n)+g_2(n),
\]
where
\begin{align}\label{e_g1}
g_1(n):=&-\gamma_1\Big[\sum_{\{x,y\}\in E:\,\{x,y\}\cap C\ne\varnothing} \mathbbm{1}_{\{\omega_n(x)=\omega_n(y)=1\}}-(2pq-p-q)
\Big]\nonumber\\
&+\gamma_{12}\Big[\sum_{j=2}^3\sum_{\{x,y\}\in E:\,\{x,y\}\cap C\ne\varnothing}\mathbbm{1}_{\{\{\omega_n(x),\omega_n(y)\}=\{1,j\}\}}-2(p+q)\Big]
\end{align}
and
\begin{align}\label{e_g2}
g_2(n):=\gamma_{23}\sum_{\{x,y\}\in E:\,\{x,y\}\cap C\ne\varnothing}\mathbbm{1}_{\{\{\omega_n(x),\omega_n(y)\}=\{2,3\}\}}.
\end{align}
First, note that $g_1$ records precisely the energy of the (reversed) reference path from $\mathbf2$ to a configuration with a single $1$-cluster $C$, translated by a fixed real number which is $(2pq-p-q)\gamma_1-2(p+q)\gamma_{12}+H(\mathbf2)$. Since $p,q\le\ell^\star-1$, it can be inferred from Definition \ref{d_refpath1} that
\[
\max_{0\le n\le pq}g_1(n)=g_1(p-1)=2(p-1)\gamma_1.
\]
Next, note that $g_2$ is monotone increasing. Hence, we apply a crude bound via assumption \eqref{e_N2partialC} by
\[
\max_{0\le n\le pq}g_2(n)=g_2(pq)=\gamma_{23}\times\mathfrak{N}_3(\partial C)\le (p+q)\gamma_{23}.
\]
Therefore, we estimate
\[
\max_n\{H(\omega_n)-H(\eta)\}\le\max_ng_1(n)+\max_ng_2(n)\le 2(p-1)\gamma_1+(p+q)\gamma_{23}.
\]
Since $p,q\le\ell^\star-1$, this is bounded above by
\[
2(\ell^\star-2)\gamma_1+2(\ell^\star-1)\gamma_{23}<2(\ell^\star-1)(\gamma_{23}+\gamma_1).
\]
Hence, we deduce that
\[
V_\eta\le\max_{0\le n\le pq}\{H(\omega_n)-H(\eta)\}<2(\ell^\star-1)(\gamma_{23}+\gamma_1)
\]
and this concludes the proof of \textbf{Case 1}.
\end{itemize}\medskip{}

\noindent \textbf{(Case 2: $\eta$ does not have an $1$-cluster)}
In this case, $\eta\in\mathcal{X}^{23}$ (cf. \eqref{e_Xij}). We claim that
\[
V_{\eta}\le2\gamma_{23}.
\]
To this end, we take a $2$-cluster $C'$ of $\eta$ which is possible since $\eta\ne\mathbf3$. If there is an internal angle of $\frac{3}{2}\pi$ in the border of $C'$, as in the proof of Lemma \ref{l_1cluster} we can find $w_1,w_2\in C'$ and $v\in \partial C'$ such that $\eta(v)=3$, $\eta(w_1)=\eta(w_2)=2$ and $w_1\sim v\sim w_2$. Thus, by updating the spin $3$ at site $v$ to spin $2$, the energy difference is at most $2\gamma_{23}-2\gamma_{23}=0$, which means that the energy does not increase. Moreover, the number of spins $2$ increases. If we repeat this procedure as long as there remains an internal angle of $\frac32\pi$ in the border of some $2$-cluster, we either obtain $\mathbf{2}$ (then there is nothing to prove since in this case $V_\eta=0$), or obtain a configuration $\hat\eta$ where all internal angles of $2$-clusters are at most $\pi$. We may assume the latter case.\medskip{}

All that remains to be proved is that $V_{\hat\eta}\le2\gamma_{23}$. There are two subcases depending on the internal angles of the $2$-clusters of $\hat\eta$.
\begin{itemize}
\item \textbf{(If the internal angles are all $\pi$)} In this subcase, all the $2$-clusters are strips and in turn all the $3$-clusters are also strips. Thus, taking any $3$-bridge in $\hat\eta$ which is adjacent to a $2$-cluster, we can update consecutively the spins $3$ to spins $2$, so that the maximal energy along the updates are $H(\hat\eta)+2\gamma_{23}$, which can be attained only at the first step. Repeating this, we eventually obtain $\mathbf2$. Therefore, we have proved that $V_{\hat\eta}\le2\gamma_{23}$.
\item \textbf{(If an internal angle is $\frac12\pi$)} By simple inspection, there always exists a collection of spin dispositions which have the form as in Figure \ref{fig8}.

\begin{figure}[ht]
\centering
\begin{tikzpicture}[scale=1.2,transform shape]
\fill[black!15!white] (0.3,0.3)rectangle(2.7,0.6);
\fill[black!60!white](0,0.3)rectangle(0.3,0.6)(2.7,0.3)rectangle(3,0.6)(0.3,0)rectangle(2.7,0.3);
\draw[step=0.3cm,color=black,thick] (0,0.3) grid (3,0.6)(0.3,0) grid (2.7,0.3);
\end{tikzpicture}
\caption{\label{fig8}Light-gray and dark-gray colors represent spins $2$ and $3$, respectively.}
\end{figure}
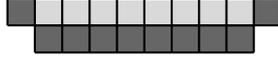\FloatBarrier

Thus, by updating consecutively the indicated light-gray (spin $2$) colors to dark-gray (spin $3$) colors, the energy does not increase along the path and at the last step the energy decreases by at least $2\gamma_{23}$. Thus, we deduce that $V_{\hat\eta}=0$ in this case. Therefore, we are done.
\end{itemize}
\end{proof}

\subsection{Initial cycle and restricted gate}

In this subsection, we define the initial cycle for each metastable and stable configurations. Then, we prove some key lemmas regarding restricted gates, which are crucial to prove the main results in Section \ref{Sec3.3} regarding the minimal gates.\medskip{}

For each $r\in S$, we define the \textit{initial cycle} of $\mathbf{r}$ as
\begin{align}\label{e_Cr}
\mathcal{C}_r:=\{\sigma\in\mathcal{X}:\ \Phi(\mathbf{r},\sigma) <H(\mathbf2)+\Gamma^\star\}.
\end{align}
By Theorem \ref{t_comheight}, $\mathbf1$, $\mathbf2$ and $\mathbf3$ are mutually separated by the energy barrier $H(\mathbf2)+\Gamma^\star$. Therefore, it follows that $\mathcal{C}_1$, $\mathcal{C}_2$ and $\mathcal{C}_3$ are mutually disjoint.

First, we show that the domains of attraction are stable under projections.

\begin{lemma}\label{l_inistab}
It holds that
\begin{enumerate}
\item[\emph{(a)}] $\mathscr{P}^{32}\mathcal{C}_1\subset\mathcal{C}_1$ and $\mathscr{P}^{32}\mathcal{C}_2\subset\mathcal{C}_2$.
\item[\emph{(b)}] $\mathscr{P}^{12}\mathcal{C}_2\subset\mathcal{C}_2$ and $\mathscr{P}^{12}\mathcal{C}_3\subset\mathcal{C}_3$.
\end{enumerate}
\end{lemma}

\begin{proof}
(a) Suppose that $\sigma\in\mathcal{C}_r$ for some $r\in\{1,2\}$. Then, there exists a path $\omega:\sigma\to\mathbf{r}$ whose height is strictly less than $H(\mathbf2)+\Gamma^\star$. Then by Lemma \ref{l_P32}, the projected path $\mathscr{P}^{32}\omega:\mathscr{P}^{32}\sigma\to\mathbf{r}$ has height also strictly less than $H(\mathbf2)+\Gamma^\star$. This proves that $\mathscr{P}^{32}\sigma\in\mathcal{C}_r$.\medskip{}

\noindent (b) First, we claim that for any $\eta\in\mathcal{X}$,
\begin{align}\label{e_claim}
\eta\in\mathcal{C}_2\cup\mathcal{C}_3\;\;\;\;\text{implies}\;\;\;\;\mathfrak{N}_1(\eta)\le\ell^\star(\ell^\star-1).
\end{align}
To this end, assume by contradiction that $\mathfrak{N}_1(\eta)>\ell^\star(\ell^\star-1)$ and, without loss of generality, assume that $\eta\in\mathcal{C}_2$. Take a path $\omega_0:\eta\to\mathbf2$ whose height is strictly less than $H(\mathbf2)+\Gamma^\star$. Then since $\mathfrak{N}_1(\mathbf2)=0$, there exists $\zeta\in\omega_0$ such that $\mathfrak{N}_1(\zeta)=\ell^\star(\ell^\star-1)+1$. Then by Lemma \ref{l_critical}, $H(\zeta)\ge H(\mathbf2)+\Gamma^\star$ and we obtain a contradiction.\medskip{}

Now, assume $\sigma\in\mathcal{C}_s$ for some $s\in\{2,3\}$. Then, there exists a path $\omega:\sigma\to\mathbf{s}$ whose height is less than $H(\mathbf2)+\Gamma^\star$. By the claim above, the number of spins $1$ of every configuration in $\omega$ does not exceed $\ell^\star(\ell^\star-1)$. Thus by Lemma \ref{l_P12}, the projected path $\mathscr{P}^{12}\omega:\mathscr{P}^{12}\sigma\to\mathbf{s}$ also has height less than $H(\mathbf2)+\Gamma^\star$ and we conclude that $\mathscr{P}^{12}\sigma\in\mathcal{C}_s$.
\end{proof}

Now, the first result in this subsection analyzes the optimal paths from $\mathbf m\in\{\mathbf2,\mathbf3\}$ to $\mathbf1$ which do not visit $\mathcal{C}_3$. Such paths are called \textit{restricted-paths} in \cite{bet2021critical}.

\begin{proposition}\label{p_213}
Suppose that $\omega=(\omega_n)_{n=0}^N$ is an optimal path from $\mathbf m$ to $\mathbf1$ which does not visit $\mathcal{C}_{m'}$ where $\{m,m'\}=\{2,3\}$. Then, we have
\[
\omega\cap\mathcal{W}(\mathbf m,\mathbf1)\ne\varnothing.
\]
\end{proposition}

\begin{proof}
Without loss of generality, assume that $\mathbf m=\mathbf2$. As $\omega_0=\mathbf2\notin\mathcal{C}_1$ and $\omega_N=\mathbf1\in\mathcal{C}_1$, we can define
\[
N':=\min\{1\le n\le N:\ \omega_n\in\mathcal{C}_1\}.
\]
Then, $\omega_{N'}\in\mathcal{C}_1$ and
\begin{align}\label{e_condi}
\omega_0,\dots,\omega_{N'-1}\notin\mathcal{C}_1\cup\mathcal{C}_3.
\end{align}
Define $\bar\omega_n:=\mathscr{P}^{32}\omega_n$ for each $0\le n\le N'$. Then, $\bar\omega_0=\mathbf2$ and by Lemma \ref{l_inistab}-(a), $\bar\omega_{N'}\in\mathcal{C}_1$, so that $\bar\omega=(\bar\omega_n)_{n=0}^{N'}$ is an $12$-path from $\mathbf2$ to $\mathcal{C}_1$. As the height of $(\omega_n)_{n=0}^{N'}$ is at most $H(\mathbf2)+\Gamma^\star$, Lemma \ref{l_P32} implies that the height of $\bar\omega$ is at most $H(\mathbf2)+\Gamma^\star$. Thus by Proposition \ref{p_gatepos}-(a), the height of $\bar\omega$ must be exactly $H(\mathbf2)+\Gamma^\star$ and in turn by Proposition \ref{p_gatepos}-(b), 
\[
\bar\omega\cap\mathcal{W}(\mathbf2,\mathbf1)\ne\varnothing.
\]
Hence, there exists $0\le m\le N'$ so that $\bar\omega_m\in\mathcal{W}(\mathbf2,\mathbf1)$. It readily holds that $H(\bar\omega_m)=H(\mathbf2)+\Gamma^\star$. Moreover, by Lemma \ref{l_P32}, $H(\omega_m)\ge H(\bar\omega_m)$. Thus,
\[
H(\mathbf2)+\Gamma^\star\ge H(\omega_m)\ge H(\bar\omega_m) = H(\mathbf2)+\Gamma^\star,
\]
so that equalities must hold in all places. Then, Lemma \ref{l_P32} again implies that
\[
\omega_m\in\mathcal{W}(\mathbf2,\mathbf1)\;\;\;\;\text{or}\;\;\;\;\omega_m\in\mathcal{W}(\mathbf3,\mathbf1).
\]
We conclude the proof of this lemma by showing that the latter is impossible. To this end, suppose on the contrary that $\omega_m\in\mathcal{W}(\mathbf3,\mathbf1)$. Then since $P_\beta(\omega_m,\omega_{m-1})>0$ and $H(\omega_m)=H(\mathbf2)+\Gamma^\star$, we readily deduce that
\begin{align}\label{e_poss1}
\omega_{m-1}\in R_{\ell^\star-1,\ell^\star}(3,1)\cup R_{\ell^\star,\ell^\star-1}(3,1)
\end{align}
or
\begin{align}\label{e_poss2}
\omega_{m-1}\in B_{\ell^\star-1,\ell^\star}^2(3,1)\cup \hat{B}_{\ell^\star,\ell^\star-1}^2(3,1),
\end{align}
since any other possibility implies that $H(\omega_{m-1})>H(\omega_m)=H(\mathbf2)+\Gamma^\star$, which is impossible. If \eqref{e_poss1} holds, then a part of the reference path $\mathbf3\to\mathbf1$ with respect to $\omega_m$ guarantees that $\omega_{m-1}\in\mathcal{C}_3$ which contradicts the condition \eqref{e_condi}. If \eqref{e_poss2} holds then the other side of the reference path guarantees that $\omega_{m-1}\in\mathcal{C}_1$, and we again obtain a contradiction with \eqref{e_condi}.
\end{proof}

Next, we deal with the optimal paths from $\mathbf2$ to $\mathbf3$ which do not visit $\mathcal{C}_1$.

\begin{proposition}\label{p_231}
Suppose that $\omega=(\omega_n)_{n=0}^N$ is an optimal path from $\mathbf2$ to $\mathbf3$ which does not visit $\mathcal{C}_1$. Choose any set
\[
\mathscr{A}\in\big\{\mathscr{W}_j^h(\mathbf2,\mathbf3)\big\}_{j,h}\cup\big\{\mathscr{Q}(\mathbf2,\mathbf3),\mathscr{P}(\mathbf2,\mathbf3),\mathscr{P}(\mathbf3,\mathbf2),\mathscr{Q}(\mathbf3,\mathbf2)\big\}\cup\big\{\mathscr{H}_i(\mathbf2,\mathbf3)\big\}_i\cup\big\{\mathscr{H}_i(\mathbf3,\mathbf2)\big\}_i.
\]
Then, we have
\[
\omega\cap\mathscr{A}\ne\varnothing.
\]
\end{proposition}

\begin{proof}
First, we claim that
\[
\mathfrak{N}_1(\omega_n)\le\ell^{\star2}-1\;\;\;\;\text{for all }0\le n\le N.
\]
To prove the claim, suppose the contrary that there exists $M$ such that $\mathfrak{N}_1(\omega_M)\ge\ell^{\star2}$. Then since $\mathfrak{N}_1(\omega_0)=\mathfrak{N}_1(\mathbf2)=0$, we can take the largest $n_0\in[0,M]$ such that $\mathfrak{N}_1(\omega_{n_0})=\ell^{\star2}-\ell^\star+1$. Then, $\mathfrak{N}_1(\omega_n)\ge\ell^{\star2}-\ell^\star+2$ for $n_0<n\le M$. Then by Lemma \ref{l_critical}, $H(\omega_{n_0})=H(\mathbf2)+\Gamma^\star$ and
\[
\omega_{n_0}\in\bigcup_{m=2}^3[\mathcal{W}(\mathbf m,\mathbf1)\cup\mathcal{W}'(\mathbf m,\mathbf1)].
\]
If $\omega_{n_0}\in\mathcal{W}(\mathbf m,\mathbf1)$ for some $m\in\{2,3\}$  then since $\mathfrak{N}_1(\omega_{n_0+1})=\mathfrak{N}_1(\omega_{n_0})+1$, the only possible option (so that $H(\omega_{n_0+1})\le H(\mathbf2)+\Gamma^\star$) is
\[
\omega_{n_0+1}\in B_{\ell^\star-1,\ell^\star}^2(m,1)\cup \hat{B}_{\ell^\star,\ell^\star-1}^2(m,1).
\]
Then, a suitable reference path from $\omega_{n_0+1}$ to $\mathbf1$ guarantees that $\omega_{n_0+1}\in\mathcal{C}_1$, which contradicts the assumption of the proposition. Hence, we may assume that
\[
\omega_{n_0}\in\bigcup_{m=2}^3\mathcal{W}'(\mathbf m,\mathbf1).
\]
We denote by $R_0$ the rectangle of side lengths $\ell^\star-1$ and $\ell^\star+1$, which is the smallest rectangle containing the $1$-cluster of $\omega_{n_0}$. We will demonstrate that
\[
\{v\in V:\ \omega_n(v) =1\} \subseteq R_0\;\;\;\;\text{for all }n_0\le n\le M,
\]
which contradicts $\mathfrak{N}_1(\omega_M)\ge\ell^{\star2}$ and thus proves the first claim. Let us suppose the contrary. Then, there exists $m_0\in(n_0,M]$ such that $\{v\in V:\ \omega_n(v)=1\}\subseteq R_0$ for all $n_0\le n<m_0$ and $\{v\in V:\ \omega_{m_0}(v)=1\}\nsubseteq R_0$. As usual, we write $\bar\omega_n:=\mathscr{P}^{32}\omega_n$.\medskip{}

\begin{itemize}
\item \textbf{(Step 1)} $H(\bar\omega_n)-H(\mathbf2)\ge 4\ell^\star\gamma_{12}-2\gamma_1(\ell^{\star2}-\ell^\star-1)$ for all $n_0\le n<m_0$.\medskip{}

Since $\mathfrak{N}_1(\bar\omega_n)\ge\ell^{\star2}-\ell^\star+1$ for $n_0\le n<m_0$, every row and column of $R_0$ must have at least one spin $1$. Indeed, if not then we would have
\[
\mathfrak{N}_1(\bar\omega_n)\le \max\{(\ell^\star-1)(\ell^\star+1-1),(\ell^\star+1)(\ell^\star-1-1)\}=\ell^{\star2}-\ell^\star,
\]
which contradicts $\mathfrak{N}_1(\bar\omega_n)\ge\ell^\star(\ell^\star-1)+1$. Thus, in each corresponding row and column there exists at least two $12$-edges in $\bar\omega_n$, so that we have
\begin{align}\label{eq1}
\mathfrak{n}_{12}(\bar\omega_n)\ge2(\ell^\star-1)+2(\ell^\star+1)=4\ell^\star.
\end{align}
Moreover, clearly the maximum of $\mathfrak{n}_{11}(\bar\omega_n)$ is attained when the box $R_0$ is full of spins $1$:
\begin{align}\label{eq2}
\mathfrak{n}_{11}(\bar\omega_n)\le(\ell^\star-1)\ell^\star+(\ell^\star+1)(\ell^\star-2)=2\ell^{\star2}-2\ell^\star-2.
\end{align}
Thus, by \eqref{eq1}, \eqref{eq2} and \eqref{e_ham'} we deduce
\[
H(\bar\omega_n)-H(\mathbf2)=\mathfrak{n}_{12}(\bar\omega_n)\gamma_{12}-\mathfrak{n}_{11}(\bar\omega_n)\gamma_1\ge4\ell^\star\gamma_{12}-2\gamma_1(\ell^{\star2}-\ell^\star-1).
\]
\item \textbf{(Step 2)} $H(\omega_{m_0})>H(\mathbf2)+\Gamma^\star$, which yields a contradiction.\medskip{}

Consider the spin flip from $\bar\omega_{m_0-1}$ to $\bar\omega_{m_0}$. Since this spin flip happens outside $R_0$, the energy difference is at least $2\gamma_{12}-\gamma_1$ and at most $4\gamma_{12}$. Therefore, by \textbf{Step 1}, we deduce that
\[
H(\bar\omega_{m_0})-H(\mathbf2)\ge H(\bar\omega_{m_0-1})-H(\mathbf2)+(2\gamma_{12}-\gamma_1)\ge (4\ell^\star+2)\gamma_{12}-\gamma_1(2\ell^{\star2}-2\ell^\star-1)>\Gamma^\star.
\]
Indeed, the last inequality is equivalent to (cf. \eqref{Gammametastable}) $2\gamma_{12}>(2\ell^\star-3)\gamma_1$, which can be rearranged as
\[
\frac{2\gamma_{12}+\gamma_1}{2\gamma_1}>\ell^\star-1.
\]
This is obvious by \eqref{e_consts}. Hence, by Lemma \ref{l_P32} we have $H(\omega_{m_0})\ge H(\bar\omega_{m_0}))> H(\mathbf2)+\Gamma^\star$.
\end{itemize}

Now, we return to the proof of Proposition \ref{p_231}. The idea is similar to the proof of Proposition \ref{p_213}. As $\omega_0=\mathbf2\notin\mathcal{C}_3$ and $\omega_N=\mathbf3\in\mathcal{C}_3$, we may define
\[
N':=\min\{1\le n\le N:\ \omega_n\in\mathcal{C}_3\},
\]
so that $\omega_{N'}\in\mathcal{C}_3$ and
\begin{align}\label{e_condi'}
\omega_0,\dots,\omega_{N'-1}\notin\mathcal{C}_1\cup\mathcal{C}_3.
\end{align}
Define $\tilde\omega_n:=\mathscr{P}^{12}\omega_n$ for each $0\le n\le N'$. Then, $\tilde\omega_0=\mathbf2$ and by Lemma \ref{l_inistab}-(b) $\tilde\omega_{N'}\in\mathcal{C}_3$, so that $\tilde\omega=(\tilde\omega_n)_{n=0}^{N'}$ is an $23$-path from $\mathbf2$ to $\mathcal{C}_3$. As the height of $(\omega_n)_{n=0}^{N'}$ is at most $H(\mathbf2)+\Gamma^\star$, Lemma \ref{l_P12}, which is applicable by the first claim, implies that the height of $\tilde\omega$ is also at most $H(\mathbf2)+\Gamma^\star$. Thus by Proposition \ref{p_gatezero}-(a), the height of $\tilde\omega$ must be exactly $H(\mathbf2)+\Gamma^\star$ and thus by Proposition \ref{p_gatezero}-(b), 
\[
\tilde\omega\cap \mathscr{A}\ne\varnothing.
\]
Hence, there exists $0\le m\le N'$ so that $\tilde\omega_m\in\mathscr{A}$. It holds that $H(\tilde\omega_m)=H(\mathbf2)+\Gamma^\star$. Moreover, by Lemma \ref{l_P12}, $H(\omega_m)\ge H(\tilde\omega_m)$. Thus,
\[
H(\mathbf2)+\Gamma^\star\ge H(\omega_m)\ge H(\tilde\omega_m) = H(\mathbf2)+\Gamma^\star,
\]
so that equalities must hold in all places. Then Lemma \ref{l_P12} again implies that
\[
\omega_m=\tilde\omega_m\in\mathscr{A},
\]
which concludes the proof.
\end{proof}

\subsection{Minimal gates for the metastable transition}\label{Sec6.4}

In this final subsection, we prove the results stated in Section \ref{Sec3.3}. Referring to the landscape given in Figure \ref{fig4} shall be helpful to understand the gist of the ideas given here.\medskip{}

First, we focus on Theorem \ref{t_gate1}. Since the situation is totally symmetric between spins $2$ and $3$, we prove Theorem \ref{t_gate1} for $\mathbf{m}=\mathbf2$.

\begin{proof}[Proof of Theorem \ref{t_gate1}-\emph{(a)}]
We prove that $\mathcal{W}(\mathbf2,\mathbf1)\cup\mathcal{W}(\mathbf3,\mathbf1)$ is a minimal gate for the transition $\mathbf2\to\mathbf1$. First, suppose that $\omega=(\omega_n)_{n=0}^N$ is an optimal path $\mathbf2\to\mathbf1$. Consider the last visit of $\omega$ to $\mathcal{C}_2\cup \mathcal{C}_3$, which is possible because $\omega$ starts at $\mathbf2\in\mathcal{C}_2\cup\mathcal{C}_3$. If the last visit is to $\mathcal{C}_2$, so that after then it does not visit $\mathcal{C}_3$, then by Proposition \ref{p_213} we deduce that $\omega\cap\mathcal{W}(\mathbf2,\mathbf1)\ne\varnothing$. If the last visit is to $\mathcal{C}_3$, then similarly by Proposition \ref{p_213}, we have $\omega\cap\mathcal{W}(\mathbf3,\mathbf1)\ne\varnothing$. This proves that $\mathcal{W}(\mathbf2,\mathbf1)\cup\mathcal{W}(\mathbf3,\mathbf1)$ is a gate.\medskip{}

To show that it is minimal, take any $\sigma\in\mathcal{W}(\mathbf2,\mathbf1)\cup\mathcal{W}(\mathbf3,\mathbf1)$. It suffices to show that $[\mathcal{W}(\mathbf2,\mathbf1)\cup\mathcal{W}(\mathbf3,\mathbf1)]\setminus\{\sigma\}$ is not a gate. If $\sigma\in\mathcal{W}(\mathbf2,\mathbf1)$, then the reference path $\mathbf2\to\mathbf1$ with respect to $\sigma$ defined in Definition \ref{d_refpath1} does not visit $\mathcal{W}(\mathbf2,\mathbf1)\cup\mathcal{W}(\mathbf3,\mathbf1)$ except at $\sigma$, and thus we are done. If $\sigma\in\mathcal{W}(\mathbf3,\mathbf1)$, we can concatenate any reference path from $\mathbf2$ to $\mathbf3$ (given in Definition \ref{d_refpath2}) and the reference path $\mathbf3\to\mathbf1$ with respect to $\sigma$ (given in Definition \ref{d_refpath1}) to obtain an optimal path from $\mathbf2$ to $\mathbf1$. This path does not visit $\mathcal{W}(\mathbf2,\mathbf1)\cup\mathcal{W}(\mathbf3,\mathbf1)$ except at $\sigma$. In these two cases we proved that $\mathcal{W}(\mathbf2,\mathbf1)\cup\mathcal{W}(\mathbf3,\mathbf1)$ is indeed a minimal gate.
\end{proof}

\begin{proof}[Proof of Theorem \ref{t_gate1}-\emph{(b)}]
As stated in the theorem, we fix a set $\mathscr{A}$ in
\[
\big\{\mathscr{W}_j^h(\mathbf2,\mathbf3)\big\}_{j,h}\cup\big\{\mathscr{Q}(\mathbf2,\mathbf3),\mathscr{P}(\mathbf2,\mathbf3),\mathscr{P}(\mathbf3,\mathbf2),\mathscr{Q}(\mathbf3,\mathbf2)\big\}\cup\big\{\mathscr{H}_i(\mathbf2,\mathbf3)\big\}_i\cup\big\{\mathscr{H}_i(\mathbf3,\mathbf2)\big\}_i.
\]
We prove that $\mathcal{W}(\mathbf2,\mathbf1)\cup\mathscr{A}$ is a minimal gate. First, we demonstrate that it is a gate. Take an arbitrary optimal path $\omega=(\omega_n)_{n=0}^N$ from $\mathbf2$ to $\mathbf1$. As we did in the proof of part (a), we divide into two cases, but in this case we consider the first visit to $\mathcal{C}_1\cup\mathcal{C}_3$ which is possible since $\mathbf1\in\mathcal{C}_1\cup\mathcal{C}_3$. If the first visit to $\mathcal{C}_1\cup\mathcal{C}_3$ is to $\mathcal{C}_1$, then by Proposition \ref{p_213} it holds that $\omega\cap\mathcal{W}(\mathbf2,\mathbf1)\ne\varnothing$. If the first visit to $\mathcal{C}_1\cup\mathcal{C}_3$ is to $\mathcal{C}_3$, then by Proposition \ref{p_231} it holds that $\omega\cap\mathscr{A}\ne\varnothing$.\medskip{}

Finally, we prove that $\mathcal{W}(\mathbf2,\mathbf1)\cup\mathscr{A}$ is minimal. Take any $\sigma\in\mathcal{W}(\mathbf2,\mathbf1)\cup\mathscr{A}$. If $\sigma\in\mathcal{W}(\mathbf2,\mathbf1)$, then the reference path $\mathbf2\to\mathbf1$ with respect to $\sigma$ defined in Definition \ref{d_refpath1} does not visit $\mathcal{W}(\mathbf2,\mathbf1)\cup\mathscr{A}$ except at $\sigma$, and thus we are done. If $\sigma\in\mathscr{A}$, we can concatenate the reference path from $\mathbf2$ to $\mathbf3$ with respect to $\sigma$ (given in Definition \ref{d_refpath2}) and any reference path from $\mathbf3$ to $\mathbf1$ (given in Definition \ref{d_refpath1}) to obtain an optimal path from $\mathbf2$ to $\mathbf1$. This path does not visit $\mathcal{W}(\mathbf2,\mathbf1)\cup\mathscr{A}$ except at $\sigma$. Therefore, we conclude that $\mathcal{W}(\mathbf2,\mathbf1)\cup\mathscr{A}$ is minimal.
\end{proof}

Next, we prove that there are no configurations, other than the ones characterized above, that form another minimal gate. This is exactly the content of item (c).

\begin{proof}[Proof of Theorem \ref{t_gate1}-\emph{(c)}]
By the equivalent characterization of unessential saddles given in \cite[Theorem 5.1]{manzo2004essential}, it suffices to prove that every $\sigma\notin\mathcal{W}(\mathbf2,\mathbf1)\cup\mathcal{W}(\mathbf3,\mathbf1)\cup\mathcal{W}(\mathbf2,\mathbf3)$ is unessential, i.e., for any $\omega\in\Omega^{opt}_{\mathbf2,\mathbf1}$ with $\sigma\in\omega$, there exists another $\omega'\in\Omega^{opt}_{\mathbf2,\mathbf1}$ such that $\mathrm{argmax}_{\omega'} H\subseteq \mathrm{argmax}_\omega H \setminus \{\sigma\}$.\medskip{}

The idea is nearly the same as the one presented in \cite[Proof of Theorem 3.2]{bet2021critical}, so we will briefly sketch the proof. We fix such $\sigma$ and $\omega\in\Omega_{\mathbf2,\mathbf1}^{opt}$. Then, as we demonstrated above, $\mathcal{W}(\mathbf2,\mathbf1)\cup\mathcal{W}(\mathbf3,\mathbf1)$ is a gate for $\mathbf2\to\mathbf1$, and thus
\[
\omega\cap[\mathcal{W}(\mathbf2,\mathbf1)\cup\mathcal{W}(\mathbf3,\mathbf1)]\ne\varnothing.
\]
If there exists $\zeta\in\omega\cap\mathcal{W}(\mathbf2,\mathbf1)$, then we may construct $\omega'$ as the reference path $\mathbf2\to\mathbf1$ with respect to $\zeta$, so that
\[
\mathrm{argmax}_{\omega'} H =\{\zeta\}\subseteq \mathrm{argmax}_{\omega} H \setminus \{\sigma\}.
\]
If $\omega\cap\mathcal{W}(\mathbf2,\mathbf1)=\varnothing$, then $\omega$ must visit $\mathcal{C}_3$ before $\mathcal{C}_1$; otherwise, there is a subpath of $\omega$ from $\mathbf2$ to $\mathcal{C}_1$ which does not visit $\mathcal{C}_3$, and then Proposition \ref{p_213} implies that $\omega\cap\mathcal{W}(\mathbf2,\mathbf1)\ne\varnothing$ which yields a contradiction. Thus by Proposition \ref{p_231}, $\omega\cap\mathscr{A}\ne\varnothing$ where $\mathscr{A}$ is any collection chosen as in Proposition \ref{p_231}. Therefore, we can apply the argument given in \cite[Proof of Theorem 3.2]{bet2021critical}, where we record the last visit to $\mathscr{H}_1(\mathbf2,\mathbf3)$ and then the first visit to $\mathscr{H}_1(\mathbf3,\mathbf2)$. Then, we can record all the gate configurations visited by $\omega$ during this period, and then glue them together to construct a reference path $\omega_1'$ from $\mathbf2$ to $\mathbf3$. Next, since $\omega\cap\mathcal{W}(\mathbf2,\mathbf1)=\varnothing$ it holds that $\omega\cap\mathcal{W}(\mathbf3,\mathbf1)\ne\varnothing$. Thus, there exists a reference path $\omega_2': \mathbf3\to\mathbf1$ which visits the configuration belonging to $\omega\cap\mathcal{W}(\mathbf3,\mathbf1)$. Concatenating $\omega_1'$ and $\omega_2'$, we obtain a new optimal path $\omega'$ so that
\[
\mathrm{argmax}_{\omega'} H \subseteq \mathrm{argmax}_\omega H.
\]
Moreover, $\mathrm{argmax}_{\omega'} H$ is clearly a subset of $\mathcal{W}(\mathbf2,\mathbf3)\cup\mathcal{W}(\mathbf3,\mathbf1)$, so that $\sigma\notin \mathrm{argmax}_{\omega'} H$. Therefore, we conclude that
\[
\mathrm{argmax}_{\omega'} H \subseteq \mathrm{argmax}_\omega H \setminus \{\sigma\}.
\]
This concludes the proof. For a more detailed explanation of the construction of $\omega'_1$, we refer to \cite[Proof of Theorem 3.2]{bet2021critical}.
\end{proof}

\begin{proof}[Proof of Theorem \ref{t_gate2}]
The proof follows the same steps as the previous one; the main ingredients are Propositions \ref{p_213} and \ref{p_231}. We omit the details of the proof to avoid unnecessary repetitions of the technical details.
\end{proof}

\appendix

\section{Appendix}
%

\subsection{Alternative form of the Hamiltonian}

Recall the definition \eqref{e_hamiltonian} of the Hamiltonian function $H:\mathcal{X}\rightarrow\mathbb{R}$, which is
\begin{align}\label{e_ham}
H(\sigma)=-\sum_{i\in S}J_{ii}\sum_{\{v,w\}\in E} \mathbbm{1}_{\{\sigma(v)=\sigma(w)=i\}}+\sum_{i,j\in S:\, i< j} J_{ij} \sum_{\{v,w\}\in E} \mathbbm{1}_{\{\{\sigma(v),\sigma(w)\}=\{i,j\}\}}.
\end{align}
Recall the definitions  \eqref{e_nij}, \eqref{e_Ni}. Since the total number of edges is $2KL$, it is clear that
\[
\sum_{i\in S}\mathfrak{n}_{ii}(\sigma)+\sum_{i<j}\mathfrak{n}_{ij}(\sigma)=2KL.
\]
Then, we may rewrite \eqref{e_ham} as
\[
H(\sigma)=-\sum_{i\in S}J_{ii}\mathfrak{n}_{ii}(\sigma)+\sum_{i<j}J_{ij}\mathfrak{n}_{ij}(\sigma)=-2KLJ_{22}-\sum_{i\in S}(J_{ii}-J_{22})\mathfrak{n}_{ii}(\sigma)+\sum_{i<j}(J_{ij}+J_{22})\mathfrak{n}_{ij}(\sigma).
\]
Note that by \eqref{e_Jii}, $J_{22}=J_{33}$. By \eqref{e_gammadef} and since $H(\mathbf2)=-2KLJ_{22}$, we deduce that
\begin{align}\label{e_ham'}
H(\sigma)=H(\mathbf2)-\gamma_1\mathfrak{n}_{11}(\sigma)+\sum_{i< j} \gamma_{ij}\mathfrak{n}_{ij}(\sigma).
\end{align}

For a final remark, fix a configuration $\sigma\in\mathcal{X}$ and a spin $i\in S$. Consider the number of $ij$-edges in $\sigma$ for all $j\in S$. If we count according to the fixed spin $i$, since each spin has exactly four neighboring spins,
this is simply four times $\mathfrak{N}_i(\sigma)$. Alternatively, using the definition \eqref{e_nij}, this equals $2\mathfrak{n}_{ii}(\sigma)+\sum_{j\ne i}\mathfrak{n}_{ij}(\sigma)$ where the factor $2$ appears in front of $\mathfrak{n}_{ii}(\sigma)$ since each $ii$-edge must be counted twice. Therefore, we conclude that
\begin{align}\label{e_nijNi}
4\mathfrak{N}_i(\sigma)=2\mathfrak{n}_{ii}(\sigma)+\sum_{j:\, j\ne i}\mathfrak{n}_{ij}(\sigma)\;\;\;\;\text{for all }\sigma\in\mathcal{X}\text{ and }i\in S.
\end{align}

\subsection{Auxiliary function}

In this subsection, we provide an estimate on an auxiliary function which is used in Assumption \ref{assump}. For every real number $h>0$, we define a function $f_h:(0,\infty)\rightarrow\mathbb{R}$ by
\begin{align}\label{e_fh}
f_h(x):= 4\Big(x+\frac{h}2\Big)\Big\lceil\frac{x+\frac{h}2}{h}\Big\rceil -2h\Big(\Big\lceil\frac{x+\frac{h}2}{h}\Big\rceil^2-\Big\lceil\frac{x+\frac{h}2}{h}\Big\rceil+1\Big).
\end{align}
Here, $\lceil\alpha\rceil$ is the least integer not smaller than $\alpha$.

\begin{lemma}\label{l_fh}
The following statements hold for $h>0$.
\begin{enumerate}
\item[\emph{(a)}] The function $f_{h}$ is continuous, piece-wise linear and strictly
increasing on $(0,\infty)$.
\item[\emph{(b)}] We have $f_{h}(\frac{h}2)=2h$ and $\lim_{x\rightarrow\infty}f_{h}(x)=\infty$.
\item[\emph{(c)}] For all $x\in(0,\infty)$,
\[
0\le f_{h}(x)-\Big(\frac{2x^{2}}{h}+4x-\frac{h}2\Big)\le\frac{h}{2}.
\]
The left (resp. right) equality holds if and only if $(x+\frac{h}2)/h\in\mathbb{N}$ (resp. $x/h\in\mathbb{N}$).
\end{enumerate}
\end{lemma}

\begin{proof}
For each integer $m\ge0$, if $x\in((m-\frac12)h,(m+\frac12)h]$ then $\lceil\frac{x+\frac{h}2}{h}\rceil=m+1$ and thus
\[
f_{h}(x)=4(m+1)\Big(x+\frac{h}2\Big)-2h(m^2+m+1)=4(m+1)x-2hm^2.
\]
In turn, we have for each $m\ge1$ that
\[
\lim_{x\rightarrow (m-\frac12)h+}f_{h}(x)=4(m+1)\Big(m-\frac12\Big)h-2hm^2=f_{h}\Big(\Big(m-\frac12\Big)h\Big).
\]
These formulas verify both (a) and (b) of the lemma. Finally, to prove (c) notice that if $x\in((m-\frac12)h,(m+\frac12)h]$, we have
\begin{align*}
f_h(x)-\Big(\frac{2x^2}h+4x-\frac{h}2\Big)=-\frac2h(x-mh)^2+\frac{h}2.
\end{align*}
This concludes the proof of (c).
\end{proof}

\subsection{Heuristics behind Assumption \ref{assump}}\label{SecA.3}

Here, we provide a brief explanation which justifies each condition given in Assumption \ref{assump}. We introduce a notation for convenience. For spins $i,j\in S$,
\begin{align}\label{e_Xij}
\mathcal{X}^{ij}:=\big\{\eta\in\mathcal{X}:\ \eta(v)\in\{i,j\}\text{ for all }v\in V\big\}.
\end{align}
In other words, $\mathcal{X}^{ij}$ is the collection of configurations in which all spins are either $i$ or $j$.\medskip{}

\begin{itemize}
\item[\textbf{A}.] First, we focus on the (potentially metastable) transition from $\mathbf{2}$ to $\mathbf{1}$. If $\sigma(v)\in\{1,2\}$ for all $v\in V$, formula \eqref{e_ham'} can be rewritten as
\[
H(\sigma)=H(\mathbf2)-\gamma_1\mathfrak{n}_{11}(\sigma)+\gamma_{12}\mathfrak{n}_{12}(\sigma).
\]
By \eqref{e_nijNi} and since $\mathfrak{n}_{13}(\sigma)=0$, we have $2\mathfrak{n}_{11}(\sigma)+\mathfrak{n}_{12}(\sigma)=4\mathfrak{N}_1(\sigma)$. Thus,
\[
H(\sigma)=H(\mathbf2)+\Big(\frac12\gamma_1+\gamma_{12}\Big)\mathfrak{n}_{12}(\sigma)-2\gamma_1\mathfrak{N}_1(\sigma).
\]
The right-hand side is, modulo translation by a real number, equivalent to the Hamiltonian of the original Ising model (where spin $1$ corresponds to $+1$ and spin $2$ corresponds to $-1$), with interaction constant $J:=\frac12\gamma_1+\gamma_{12}$ and external field $h:=2\gamma_1$. To avoid technical difficulties it is standard to assume that $\frac{2J}{h}$ is not an integer (e.g., \cite[standard case]{neves1991critical}). In our context, this is equivalent to
\[
\frac{2\gamma_{12}+\gamma_1}{2\gamma_1}\text{ is not an integer,}
\]
which is exactly condition \textbf{A}.
\item[\textbf{B}.] In the original Ising metastable transition, the energy barrier is known to be $4J\lceil\frac{2J}{h}\rceil-h(\lceil\frac{2J}{h}\rceil^2-\lceil\frac{2J}{h}\rceil+1)$ (cf. \cite{neves1991critical}). In our context, this becomes
\begin{align}\label{e_Ebarrier21}
(4\gamma_{12}+2\gamma_1)\Big\lceil\frac{2\gamma_{12}+\gamma_1}{2\gamma_1}\Big\rceil-2\gamma_1\Big( \Big\lceil\frac{2\gamma_{12}+\gamma_1}{2\gamma_1}\Big\rceil^2-\Big\lceil\frac{2\gamma_{12}+\gamma_1}{2\gamma_1}\Big\rceil+1\Big)=f_{\gamma_1}(\gamma_{12}).
\end{align}
The last equality follows from the definition \eqref{e_fh}.\medskip{}

Next, we consider the (potentially metastable) transition from $\mathbf{2}$ to $\mathbf{3}$. If $\sigma(v)\in\{2,3\}$ for all $v\in V$, the representation \eqref{e_ham'} becomes
\[
H(\sigma)=H(\mathbf2)+\gamma_{23}\mathfrak{n}_{23}(\sigma).
\]
The right-hand side is, modulo translation by a real number, equivalent to the Hamiltonian of the original Ising model with interaction constant $J:=\gamma_{23}$ and zero external field. The energy barrier in this setting is known to be $2J(K+1)$ (cf. \cite{nardi2019tunneling}). In our context, this equals
\begin{align}\label{e_Ebarrier23}
2(K+1)\gamma_{23}.
\end{align}
Gathering \eqref{e_Ebarrier21} and \eqref{e_Ebarrier23}, to have the same energy barrier between $\mathbf2\to\mathbf1$ and $\mathbf2\to\mathbf3$, we obtain the condition
\begin{align*}
f_{\gamma_1}(\gamma_{12})=2(K+1)\gamma_{23},
\end{align*}
which is condition \textbf{B}.
\item[\textbf{C}.] First, the size of the protocritical droplet $\lceil\frac{2\gamma_{12}+\gamma_1}{2\gamma_1}\rceil$ (cf. \eqref{e_Ebarrier21}) is assumed to be large enough:
\begin{align}\label{e_C1}
\frac{\gamma_{12}}{\gamma_1}\text{ is sufficiently large.}
\end{align}
Next, we start from condition \textbf{B}. By Lemma \ref{l_fh}-(c), we obtain that
\begin{align*}
2(K+1)\gamma_{23}=f_{\gamma_1}(\gamma_{12})\le \frac{2\gamma_{12}^2}{\gamma_1}+4\gamma_{12}<\frac{2(\gamma_{12}+\gamma_1)^2}{\gamma_1}.
\end{align*}
The above display implies
\begin{align*}
\frac{(K+1)\gamma_1}{\gamma_{12}+\gamma_1} < \frac{\gamma_{12}+\gamma_1}{\gamma_{23}}.
\end{align*}
The lattice size is large enough compared to the fixed coefficients, and thus $\frac{(K+1)\gamma_1}{\gamma_{12}+\gamma_1}$ can be assumed sufficiently large. This also implies that
\begin{align}\label{e_C2}
\frac{\gamma_{12}+\gamma_1}{\gamma_{23}}\text{ is sufficiently large.}
\end{align}
Conditions \eqref{e_C1} and \eqref{e_C2} imply the desired condition \textbf{C}.
\end{itemize}

\subsection{Reference paths}\label{SecA.4}

Reference paths in our new model are defined in the same manner as in the original Ising model. Since the reference paths for Ising/Potts models with both non-zero \cite[Definition 5.1]{bet2021metastabilityneg} or zero \cite[Proposition 2.4]{nardi2019tunneling} external fields are very well known, we will give our definitions in a concise manner.

\begin{definition}[Reference path between $\mathbf{m}\in\{\mathbf2,\mathbf3\}$ and $\mathbf1$]\label{d_refpath1} Recall from \eqref{e_Wm1} that
\[
\mathcal{W}(\mathbf{m},\mathbf1)=B_{\ell^\star-1,\ell^\star}^1(m,1)\cup \hat{B}_{\ell^\star,\ell^\star-1}^1(m,1).
\]
For any $\eta\in\mathcal{W}(\mathbf m,\mathbf1)$, we construct a reference path $\omega:\mathbf m\to\mathbf1$ satisfying $\mathrm{argmax}_\omega H=\{\eta\}$ as follows. Denote by $R_\eta$ the rectangle of side lengths $\ell^\star$ and $\ell^\star-1$ contained in the $1$-cluster of $\eta$. Starting from $\omega_0=\mathbf m$, we consecutively update spins $m$ in $R_\eta$ to spins $1$ to form $\omega_{\ell^\star(\ell^\star-1)}=\eta_0$, where
\[
\eta_0(v)=\begin{cases} 1,& \text{if }v\in R_\eta,\\
m,& \text{if }v\notin R_\eta.
\end{cases}
\]
Next, from $\eta_0$ we create the corresponding protuberance to obtain $\omega_{\ell^\star(\ell^\star-1)+1}=\eta$. Then, we resume to enlarge the $1$-cluster in the usual consecutive manner to obtain $\omega_{KL}=1$. By the isomorphism argument given in Section \ref{SecA.3}-\textbf{A}, it is standard to observe that the height of $\omega$ is obtained uniquely at $\omega_{\ell^\star(\ell^\star-1)+1}=\eta$, and that the corresponding height is
\[
H(\eta)=H(\mathbf2)+4\ell^\star\Big(\gamma_{12}+\frac{\gamma_1}2\Big)-2\gamma_1(\ell^{\star2}-\ell^\star+1)=H(\mathbf2)+\Gamma^\star.
\]
\end{definition}

\begin{definition}[Reference path between $\mathbf2$ and $\mathbf3$]\label{d_refpath2}
First, we choose an arbitrary column $c$. Starting from $\mathbf{2}$, we update spins $2$ in $c$ to $3$ in a consecutive manner. Then, we choose one of its neighboring column and repeat the process. Iterating this procedure, we obtain $\mathbf{3}$. Similarly, by the isomorphism argument given in Section \ref{SecA.3}-\textbf{B}, we observe that the height of this path is obtained multiple times and that the height is
\[
H(\mathbf2)+(2K+2)\gamma_{23}=H(\mathbf2)+\Gamma^\star.
\]
For the reference path, we are able to select any order of columns, as long as the consecutive ones are neighboring, and also we may select any order of updates in each column, as long as the updates are consecutive. Thus, one can see that there are a huge number of possible reference paths from $\mathbf2$ to $\mathbf3$.

Consider any selection $\mathscr{A}$ from the collection
\[
\big\{\mathscr{W}_j^h(\mathbf2,\mathbf3)\big\}_{j,h}\cup\big\{\mathscr{Q}(\mathbf2,\mathbf3),\mathscr{P}(\mathbf2,\mathbf3),\mathscr{P}(\mathbf3,\mathbf2),\mathscr{Q}(\mathbf3,\mathbf2)\big\}\cup\big\{\mathscr{H}_i(\mathbf2,\mathbf3)\big\}_i\cup\big\{\mathscr{H}_i(\mathbf3,\mathbf2)\big\}_i,
\]
as in Theorem \ref{t_gate1}. By the diagram illustrated in Figure \ref{fig3} and by the freedom to choose an arbitrary order of spin updates, it is clear that for any $\sigma\in\mathscr{A}$ we can construct a reference path $\mathbf2\to\mathbf3$ so that it visits $\mathscr{A}$ only at $\sigma$.

Finally, we remark that in the case $K=L$, we may also choose an arbitrary row and proceed as described above, which also gives the height $(2L+2)\gamma_{23}=(2K+2)\gamma_{23}=\Gamma^\star$. Thus in this case, there are exactly two times more reference paths compared to the case $K<L$. This fact is not taken into account in the qualitative analysis of metastability via pathwise approach done in this paper; however, this will be crucial in the \textit{quantitative} analysis, when one intends to investigate the exact prefactor of the mean metastable transition time \cite{bet2021metastabilityneg,boviermanzo2002metastability,kim2021metastability}. This serves as a fruitful future research topic.
\end{definition}


\begin{thebibliography}{10}

\bibitem{alonso1996three}
L.~Alonso and R.~Cerf.
\newblock The three dimensional polyominoes of minimal area.
\newblock {\em The Electronic Journal of Combinatorics}, 3(1):R27, 1996.

\bibitem{ananikyan1995phase}
N.~Ananikyan and A.~Akheyan.
\newblock Phase transition mechanisms in the {P}otts model on a {B}ethe
  lattice.
\newblock {\em Journal of Experimental and Theoretical Physics},
  80(1):105--111, 1995.

\bibitem{apollonio2021metastability}
V.~Apollonio, V.~Jacquier, F.~R. Nardi, and A.~Troiani.
\newblock Metastability for the {I}sing model on the hexagonal lattice.
\newblock {\em arXiv:2101.11894}, 2021.

\bibitem{arous1996metastability}
G.~B. Arous and R.~Cerf.
\newblock Metastability of the three dimensional {I}sing model on a torus at
  very low temperatures.
\newblock {\em Electronic Journal of Probability}, 1, 1996.

\bibitem{bashiri2019on}
K.~Bashiri.
\newblock On the metastability in three modifications of the {I}sing model.
\newblock {\em View Journal Impact}, 25(3):483--532, 2019.

\bibitem{baxter1982critical}
R.~Baxter.
\newblock Critical antiferromagnetic square-lattice {P}otts model.
\newblock {\em Proceedings of the Royal Society of London. A. Mathematical and
  Physical Sciences}, 383(1784):43--54, 1982.

\bibitem{baxter1973potts}
R.~J. Baxter.
\newblock Potts model at the critical temperature.
\newblock {\em Journal of Physics C: Solid State Physics}, 6(23):L445, 1973.

\bibitem{baxter1978triangular}
R.~J. Baxter, H.~Temperley, and S.~E. Ashley.
\newblock Triangular {P}otts model at its transition temperature, and related
  models.
\newblock {\em Proceedings of the Royal Society of London. A. Mathematical and
  Physical Sciences}, 358(1695):535--559, 1978.

\bibitem{beltran2010tunneling}
J.~Beltran and C.~Landim.
\newblock Tunneling and metastability of continuous time {M}arkov chains.
\newblock {\em Journal of Statistical Physics}, 140(6):1065--1114, 2010.

\bibitem{beltran2009zerorange}
J.~Beltr{\'a}n and C.~Landim.
\newblock Metastability of reversible condensed zero range proceses on a finite
  set.
\newblock {\em Probability Theory and Related Fields}, 152:781--807, 2012.

\bibitem{beltran2012tunneling}
J.~Beltr{\'a}n and C.~Landim.
\newblock Tunneling and metastability of continuous time {M}arkov chains, the
  nonreversible case.
\newblock {\em Journal of Statistical Physics}, 149(4):598--618, 2012.

\bibitem{bet2021critical}
G.~Bet, A.~Gallo, and F.~R. Nardi.
\newblock Critical configurations and tube of typical trajectories for the
  {P}otts and {I}sing models with zero external field.
\newblock {\em Journal of Statistical Physics}, 184(30), 2021.

\bibitem{bet2021metastabilityneg}
G.~Bet, A.~Gallo, and F.~R. Nardi.
\newblock Metastability for the degenerate {P}otts {M}odel with negative
  external magnetic field under {G}lauber dynamics.
\newblock {\em arXiv:2105.14335}, 2021.

\bibitem{bet2021metastabilitypos}
G.~Bet, A.~Gallo, and F.~R. Nardi.
\newblock Metastability for the degenerate {P}otts {M}odel with positive
  external magnetic field under {G}lauber dynamics.
\newblock {\em arXiv:2108.04011}, 2021.

\bibitem{bet2020effect}
G.~Bet, V.~Jacquier, and F.~R. Nardi.
\newblock Effect of energy degeneracy on the transition time for a series of
  metastable states: application to probabilistic cellular automata.
\newblock {\em arXiv:2007.08342}, 2020.

\bibitem{bianchi2016metastable}
A.~Bianchi and A.~Gaudilliere.
\newblock Metastable states, quasi-stationary distributions and soft measures.
\newblock {\em Stochastic Processes and their Applications}, 126(6):1622--1680,
  2016.

\bibitem{bovier2016metastability}
A.~Bovier and F.~Den~Hollander.
\newblock {\em Metastability: a potential-theoretic approach}, volume 351.
\newblock Springer, 2016.

\bibitem{bovier2006sharp}
A.~Bovier, F.~den Hollander, and F.~R. Nardi.
\newblock Sharp asymptotics for {K}awasaki dynamics on a finite box with open
  boundary.
\newblock {\em Probability Theory and Related Fields}, 135(2):265--310, 2006.

\bibitem{bovier2002metastability}
A.~Bovier, M.~Eckhoff, V.~Gayrard, and M.~Klein.
\newblock Metastability and low lying spectral in reversible {M}arkov chains.
\newblock {\em Communications in Mathematical Physics}, 228(2):219--255, 2002.

\bibitem{bovier2004metastability}
A.~Bovier, M.~Eckhoff, V.~Gayrard, and M.~Klein.
\newblock Metastability in reversible diffusion processes {I}. {Sharp}
  asymptotics for capacities and exit times.
\newblock {\em Journal of the European Mathematical Society}, 2004.

\bibitem{boviermanzo2002metastability}
A.~Bovier and F.~Manzo.
\newblock Metastability in {{G}lauber} dynamics in the low-temperature limit:
  beyond exponential asymptotics.
\newblock {\em Journal of Statistical Physics}, 107(3-4):757--779, 2002.

\bibitem{cassandro1984metastable}
M.~Cassandro, A.~Galves, E.~Olivieri, and M.~E. Vares.
\newblock Metastable behavior of stochastic dynamics: a pathwise approach.
\newblock {\em Journal of Statistical Physics}, 35(5):603--634, 1984.

\bibitem{cirillo1998metastability}
E.~N. Cirillo and J.~L. Lebowitz.
\newblock Metastability in the two-dimensional {I}sing model with free boundary
  conditions.
\newblock {\em Journal of Statistical Physics}, 90(1):211--226, 1998.

\bibitem{cirillo2003metastability}
E.~N. Cirillo and F.~R. Nardi.
\newblock Metastability for a stochastic dynamics with a parallel heat bath
  updating rule.
\newblock {\em Journal of Statistical Physics}, 110(1):183--217, 2003.

\bibitem{cirillo2013relaxation}
E.~N. Cirillo and F.~R. Nardi.
\newblock Relaxation height in energy landscapes: an application to multiple
  metastable states.
\newblock {\em Journal of Statistical Physics}, 150(6):1080--1114, 2013.

\bibitem{cirillo2015metastability}
E.~N. Cirillo, F.~R. Nardi, and J.~Sohier.
\newblock Metastability for general dynamics with rare transitions: escape time
  and critical configurations.
\newblock {\em Journal of Statistical Physics}, 161(2):365--403, 2015.

\bibitem{cirillo2008competitive}
E.~N. Cirillo, F.~R. Nardi, and C.~Spitoni.
\newblock Competitive nucleation in reversible probabilistic cellular automata.
\newblock {\em Physical Review E}, 78(4):040601, 2008.

\bibitem{cirillo2008metastability}
E.~N. Cirillo, F.~R. Nardi, and C.~Spitoni.
\newblock Metastability for reversible probabilistic cellular automata with
  self-interaction.
\newblock {\em Journal of Statistical Physics}, 132(3):431--471, 2008.

\bibitem{cirillo2017sum}
E.~N. Cirillo, F.~R. Nardi, and C.~Spitoni.
\newblock Sum of exit times in a series of two metastable states.
\newblock {\em The European Physical Journal Special Topics},
  226(10):2421--2438, 2017.

\bibitem{cirillo1996metastability}
E.~N. Cirillo and E.~Olivieri.
\newblock Metastability and nucleation for the {B}lume-{C}apel model.
  {D}ifferent mechanisms of transition.
\newblock {\em Journal of Statistical Physics}, 83(3):473--554, 1996.

\bibitem{costeniuc2005complete}
M.~Costeniuc, R.~S. Ellis, and H.~Touchette.
\newblock Complete analysis of phase transitions and ensemble equivalence for
  the {C}urie--{W}eiss--{P}otts model.
\newblock {\em Journal of Mathematical Physics}, 46(6):063301, 2005.

\bibitem{dai2015fast}
P.~Dai~Pra, B.~Scoppola, and E.~Scoppola.
\newblock Fast mixing for the low temperature 2d {I}sing model through
  irreversible parallel dynamics.
\newblock {\em Journal of Statistical Physics}, 159(1):1--20, 2015.

\bibitem{de1991metastability}
F.~de~Aguiar, L.~Bernardes, and S.~G. Rosa.
\newblock Metastability in the {P}otts model on the {C}ayley tree.
\newblock {\em Journal of Statistical Physics}, 64(3):673--682, 1991.

\bibitem{den2003droplet}
F.~den Hollander, F.~Nardi, E.~Olivieri, and E.~Scoppola.
\newblock Droplet growth for three-dimensional {K}awasaki dynamics.
\newblock {\em Probability Theory and Related Fields}, 125(2):153--194, 2003.

\bibitem{den2012metastability}
F.~den Hollander, F.~Nardi, and A.~Troiani.
\newblock Metastability for {K}awasaki dynamics at low temperature with two
  types of particles.
\newblock {\em Electronic Journal of Probability}, 17, 2012.

\bibitem{den2018metastability}
F.~den Hollander, F.~R. Nardi, and S.~Taati.
\newblock Metastability of hard-core dynamics on bipartite graphs.
\newblock {\em Electronic Journal of Probability}, 23, 2018.

\bibitem{di1987potts}
F.~di~Liberto, G.~Monroy, and F.~Peruggi.
\newblock The {P}otts model on {B}ethe lattices.
\newblock {\em Zeitschrift f{\"u}r Physik B Condensed Matter}, 66(3):379--385,
  1987.

\bibitem{ellis1990limit}
R.~S. Ellis and K.~Wang.
\newblock Limit theorems for the empirical vector of the
  {C}urie-{W}eiss-{P}otts model.
\newblock {\em Stochastic Processes and their Applications}, 35(1):59--79,
  1990.

\bibitem{ellis1992limit}
R.~S. Ellis and K.~Wang.
\newblock Limit theorems for maximum likelihood estimators in the
  {C}urie-{W}eiss-{P}otts model.
\newblock {\em Stochastic Processes and their Applications}, 40(2):251--288,
  1992.

\bibitem{enting1982triangular}
I.~Enting and F.~Wu.
\newblock Triangular lattice {P}otts models.
\newblock {\em Journal of Statistical Physics}, 28(2):351--373, 1982.

\bibitem{fernandez2015asymptotically}
R.~Fernandez, F.~Manzo, F.~Nardi, and E.~Scoppola.
\newblock Asymptotically exponential hitting times and metastability: a
  pathwise approach without reversibility.
\newblock {\em Electronic Journal of Probability}, 20, 2015.

\bibitem{fernandez2016conditioned}
R.~Fernandez, F.~Manzo, F.~Nardi, E.~Scoppola, and J.~Sohier.
\newblock Conditioned, quasi-stationary, restricted measures and escape from
  metastable states.
\newblock {\em Annals of Applied Probability}, 26(2):760--793, 2016.

\bibitem{gandolfo2010limit}
D.~Gandolfo, J.~Ruiz, and M.~Wouts.
\newblock Limit theorems and coexistence probabilities for the
  {C}urie-{W}eiss-{P}otts model with an external field.
\newblock {\em Stochastic Processes and their Applications}, 120(1):84--104,
  2010.

\bibitem{gaudillierelandim2014}
A.~Gaudilliere and C.~Landim.
\newblock {A Dirichlet principle for non reversible Markov chains and some
  recurrence theorems}.
\newblock {\em {Probability Theory and Related Fields}}, 158:55--89, 2014.

\bibitem{gaudilliere2005nucleation}
A.~Gaudilliere, E.~Olivieri, and E.~Scoppola.
\newblock Nucleation pattern at low temperature for local {K}awasaki dynamics
  in two dimensions.
\newblock {\em Markov Processes and Related Fields}, 11:553--628, 2005.

\bibitem{hollander2000metastability}
F.~d. Hollander, E.~Olivieri, and E.~Scoppola.
\newblock Metastability and nucleation for conservative dynamics.
\newblock {\em Journal of Mathematical Physics}, 41(3):1424--1498, 2000.

\bibitem{jovanovski2017metastability}
O.~Jovanovski.
\newblock Metastability for the {I}sing model on the hypercube.
\newblock {\em Journal of Statistical Physics}, 167(1):135--159, 2017.

\bibitem{kim2021inclusion}
S.~Kim.
\newblock Second time scale of the metastability of reversible inclusion
  processes.
\newblock {\em Probability Theory and Related Fields}, 180:1135--1187, 2021.

\bibitem{kimseo2021inclusion}
S.~Kim and I.~Seo.
\newblock Condensation and metastable behavior of non-reversible inclusion
  processes.
\newblock {\em Communciations in Mathematical Physics}, 382:1343--1401, 2021.

\bibitem{kim2021metastability}
S.~Kim and I.~Seo.
\newblock Metastability of stochastic {I}sing and {P}otts models on lattices
  without external fields.
\newblock {\em arXiv:2102.05565}, 2021.

\bibitem{kotecky1994shapes}
R.~Koteck{\`y} and E.~Olivieri.
\newblock Shapes of growing droplets—a model of escape from a metastable
  phase.
\newblock {\em Journal of Statistical Physics}, 75(3):409--506, 1994.

\bibitem{landim2014totasymzrp}
C.~Landim.
\newblock Metastability for a non-reversible dynamics: the evolution of the
  condensate in totally asymmetric zero range processes.
\newblock {\em Communications in Mathematical Physics}, 330:1--32, 2014.

\bibitem{landim2019metastable}
C.~Landim.
\newblock Metastable {M}arkov chains.
\newblock {\em Probability Surveys}, 16:143--227, 2019.

\bibitem{landim2022resolvent}
C.~Landim, J.~Lee, and I.~Seo.
\newblock Metastability of overdamped langevin dynamics.
\newblock {\em \textit{in preparation}}, 2022.

\bibitem{landim2021resolvent}
C.~Landim, D.~Marcondes, and I.~Seo.
\newblock A resolvent approach to metastability.
\newblock {\em arXiv:2102.00998}, 2021.

\bibitem{manzo2004essential}
F.~Manzo, F.~R. Nardi, E.~Olivieri, and E.~Scoppola.
\newblock On the essential features of metastability: tunnelling time and
  critical configurations.
\newblock {\em Journal of Statistical Physics}, 115(1-2):591--642, 2004.

\bibitem{nardi2012sharp}
F.~Nardi and C.~Spitoni.
\newblock Sharp asymptotics for stochastic dynamics with parallel updating
  rule.
\newblock {\em Journal of Statistical Physics}, 146(4):701--718, 2012.

\bibitem{nardi1996low}
F.~R. Nardi and E.~Olivieri.
\newblock Low temperature stochastic dynamics for an {I}sing model with
  alternating field.
\newblock In {\em Markov Proc. Relat. Fields}, volume~2, pages 117--166, 1996.

\bibitem{nardi2019tunneling}
F.~R. Nardi and A.~Zocca.
\newblock Tunneling behavior of {I}sing and {P}otts models in the
  low-temperature regime.
\newblock {\em Stochastic Processes and their Applications},
  129(11):4556--4575, 2019.

\bibitem{nardi2016hitting}
F.~R. Nardi, A.~Zocca, and S.~C. Borst.
\newblock Hitting time asymptotics for hard-core interactions on grids.
\newblock {\em Journal of Statistical Physics}, 162(2):522--576, 2016.

\bibitem{neves1991critical}
E.~J. Neves and R.~H. Schonmann.
\newblock Critical droplets and metastability for a {G}lauber dynamics at very
  low temperatures.
\newblock {\em Communications in Mathematical Physics}, 137(2):209--230, 1991.

\bibitem{neves1992behavior}
E.~J. Neves and R.~H. Schonmann.
\newblock Behavior of droplets for a class of {G}lauber dynamics at very low
  temperature.
\newblock {\em Probability Theory and Related Fields}, 91(3-4):331--354, 1992.

\bibitem{olivieri1995markov}
E.~Olivieri and E.~Scoppola.
\newblock Markov chains with exponentially small transition probabilities:
  first exit problem from a general domain. the reversible case.
\newblock {\em Journal of Statistical Physics}, 79(3):613--647, 1995.

\bibitem{olivieri1996markov}
E.~Olivieri and E.~Scoppola.
\newblock Markov chains with exponentially small transition probabilities:
  first exit problem from a general domain. the general case.
\newblock {\em Journal of Statistical Physics}, 84(5):987--1041, 1996.

\bibitem{olivieri2005large}
E.~Olivieri and M.~E. Vares.
\newblock {\em Large deviations and metastability}, volume 100.
\newblock Cambridge University Press, 2005.

\bibitem{procacci2016probabilistic}
A.~Procacci, B.~Scoppola, and E.~Scoppola.
\newblock Probabilistic cellular automata for low-temperature 2-d {I}sing
  model.
\newblock {\em Journal of Statistical Physics}, 165(6):991--1005, 2016.

\bibitem{seo2019zerorange}
I.~Seo.
\newblock Condensation of non-reversible zero-range processes.
\newblock {\em Communciations in Mathematical Physics}, 366:781--839, 2019.

\bibitem{wang1994solutions}
K.~Wang.
\newblock Solutions of the variational problem in the {C}urie-{W}eiss-{P}otts
  model.
\newblock {\em Stochastic processes and their applications}, 50(2):245--252,
  1994.

\bibitem{zocca2019tunneling}
A.~Zocca.
\newblock Tunneling of the hard-core model on finite triangular lattices.
\newblock {\em Random Structures \& Algorithms}, 55(1):215--246, 2019.

\end{thebibliography}
\end{document}